\documentclass[11pt]{article}
\usepackage[english]{babel}
\usepackage[T1]{fontenc}
\usepackage[utf8]{inputenc}
\usepackage{amsbsy}
\usepackage{amsmath}
\usepackage{amssymb}
\usepackage{amsfonts}
\usepackage{amsthm}
\usepackage{bm}
\usepackage{graphicx}
\usepackage[active]{srcltx}
\usepackage{upgreek}
\usepackage{xcolor}

\newcommand{\1}{\bm{1}}

\newcommand{\C}{\mathbb{C}}

\newcommand{\N}{\mathbb{N}}

\newcommand{\R}{\mathbb{R}}

\newcommand{\T}{\mathbb{T}}

\newcommand{\Z}{\mathbb{Z}}

\newcommand{\boC}{\mathcal{C}}

\newcommand{\boE}{\mathcal{E}}

\newcommand{\boI}{\mathcal{I}}

\newcommand{\boU}{\mathcal{U}}
\newcommand{\boV}{\mathcal{V}}

\newcommand{\gI}{\mathfrak{I}}

\newcommand{\gU}{\mathfrak{U}}

\newcommand{\ga}{\mathfrak{a}}

\newcommand{\gc}{\mathfrak{c}}

\newcommand{\gp}{\mathfrak{p}}

\newcommand{\gu}{\mathfrak{u}}

\newtheorem{cor}{Corollary}
\newtheorem{lem}[cor]{Lemma}
\newtheorem{prop}[cor]{Proposition}
\newtheorem{thm}{Theorem}

\newtheorem*{thm*}{Theorem}
\theoremstyle{definition}
\newtheorem*{merci}{Acknowledgments}
\newtheorem{rem}[cor]{Remark}
\theoremstyle{remark}

\begin{document}

\title{Minimizing travelling waves for the Gross-Pitaevskii equation on $\R \times \T$}
\author{
\renewcommand{\thefootnote}{\arabic{footnote}}
Andr\'e de Laire\footnotemark[1],~Philippe Gravejat\footnotemark[2]~ and Didier Smets\footnotemark[3]}
\footnotetext[1]{
Univ.\ Lille, CNRS, Inria, UMR 8524 - Laboratoire Paul Painlev\'e, Inria, F-59000 Lille, France.
E-mail: {\tt andre.de-laire@univ-lille.fr}}
\footnotetext[2]{CY Cergy Paris Universit\'e, Laboratoire Analyse, G\'eom\'etrie, Mod\'elisation (UMR CNRS 8088), F-95302 Cergy-Pontoise, France. E-mail: {\tt philippe.gravejat@cyu.fr}}
\footnotetext[3]{Sorbonne Universit\'e, Laboratoire Jacques-Louis Lions (UMR CNRS 7598), F-75005 Paris, France. E-mail: {\tt didier.smets@sorbonne-universite.fr}}
\maketitle

\begin{abstract}
We study the Gross-Pitaevskii equation in dimension two with periodic conditions in one
direction, or equivalently on the product space $\R \times \T_L$ where
$L > 0$ and $\T_L = \R / L \Z.$ We focus on the variational problem
consisting in minimizing the Ginzburg-Landau energy under a fixed momentum
constraint. We prove that there exists a threshold value for $L$ below
which minimizers are the one-dimensional dark solitons, and above which
no minimizer can be one-dimensional.
\end{abstract}

%%%%%%%%%%%%%%
%%%%%%%%%%%%%%
%%%%%%%%%%%%%%
\section{Introduction}
%%%%%%%%%%%%%%
%%%%%%%%%%%%%%
%%%%%%%%%%%%%%

We are interested in the Gross-Pitaevskii equation
\begin{equation}
\label{GP}
\tag{GP}
i \partial_t \Psi = \Delta \Psi + \Psi \big( 1 - |\Psi|^2 \big).
\end{equation}
In Physics, this equation is a classical model for Bose-Einstein condensates, superfluidity or supraconductivity~\cite{Gross1, Pitaevs1}. It also gives account of the propagation of dark solitons in nonlinear optics~\cite{KivsLut1}.

Our attention in this paper is devoted to the case where the spatial domain
is the product space $\R \times \T_L$, where $L > 0$ and $\T_L
= \R / L\Z$, so that $\Psi \equiv \Psi(x,y,t) : (\R \times \T_L) \times \R
\to \C$. Solutions of the 1D Gross-Pitaevskii equation can of
course be considered as solutions in this 2D setting with a trivial
dependence on the $y$ variable. 

Dark solitons are special solutions of the 1D Gross-Pitaevskii equation. They are travelling waves of the form
$$
\Psi_c(x, t) = \gu_c(x - c t),
$$
where $c$ is any subsonic speed, i.e. $|c| \leq \sqrt{2}$. Their profile $\gu_c$ is solution to the ordinary differential equation
\begin{equation}
\label{eq:gu-p}
i \, c \, \gu_c' + \gu_c'' + \big( 1 - |\gu_c|^2 \big) \gu_c = 0,
\end{equation}
and is explicitly given by the expression
\begin{equation}
\label{eq:gu-val}
\gu_c(x) = \sqrt{\frac{2 - c^2}{2}} \tanh \bigg( \frac{\sqrt{2 - c^2}}{2} x
\bigg) + i \frac{c}{\sqrt{2}}.
\end{equation}
For $c = 0$, the profile $\gu_0$ vanishes and the corresponding soliton is called the black or kink soliton. The other solitons are called grey solitons.

Variational characterizations of the dark solitons were proved
in~\cite{BetGrSa2, BeGrSaS1}. These characterizations are based on two conserved
quantities. The first one is the 1D Ginzburg-Landau energy
\begin{equation}
\label{def:ge-energy}
E(\psi) := \frac{1}{2} \int_\R |\psi'|^2 + \frac{1}{4} \int_\R \big( 1 - |\psi|^2 \big)^2,
\end{equation}
which is the Hamiltonian of the Gross-Pitaevskii equation. Corresponding to this energy is the energy set
\begin{equation}
\label{def:gX-space}
X(\R) := \big\{ \psi \in H_\text{loc}^1(\R) : \psi' \in L^2(\R) \text{ and } 1 - |\psi|^2 \in L^2(\R) \big\},
\end{equation}
which provides the natural functional framework for analyzing the equation.

The second one is the momentum $P$, which is formally defined as
$$
P(\psi) = \frac{1}{2} \int_\R \langle i \psi', \psi \rangle_\C,
$$
where, here as in the sequel, the notation $\langle z_1, z_2 \rangle_\C := \textup{Re}(z_1 \bar{z}_2)$ stands for the canonical scalar product on the two-dimensional real vector space $\C$.
The expression of $P(\psi)$ above certainly makes sense if $\psi'$ is
compactly supported, but it is generally ill-defined for arbitrary $\psi \in X(\R)$ due to
the possible lack of integrability of the momentum density $\langle i \psi',
\psi \rangle_\C$ at infinity. 
It was shown in~\cite{BeGrSaS1} (see also Appendix~\ref{sec:dimension-one}
below) that a notion of momentum can be rigorously defined on the whole energy set
$X(\R)$ provided its value is understood in the quotient space $\R / \pi\Z.$ 
It was called the untwisted momentum in~\cite{BeGrSaS1}, and denoted by $[P]$. 
Whenever $\psi'$ has compact support, it holds
$$
[P](\psi) = \frac{1}{2} \int_\R \langle i \psi', \psi \rangle_\C \quad \text{ modulo } \pi.
$$

The characterization of the dark solitons on the line can be phrased as follows:

\begin{prop}[\cite{BetGrSa2, BeGrSaS1}]
\label{prop:exist-min1}
Let $p \in \R/\pi\Z$, $p \neq 0$. The minimizers of the variational problem
\begin{equation}
\label{def:gI-problem}
\gI(p) := \inf \big\{ E(\psi) : \psi \in X(\R) \textup{ s.t. } [P](\psi) = p
\big\}
\end{equation}
are exactly the dark soliton $\gu_{c_p}$ and the function obtained from $\gu_{c_p}$
by translation and constant phase shift. The value $c_p \in (-\sqrt{2},
\sqrt{2})$ is characterized by $[P](\gu_{c_p}) = p$.
\end{prop}

In the context of the Gross-Pitaevskii equation on the product space $\R \times \T_L$, we consider the vector space
\begin{equation}
\label{def:H-loc-1}
H_\text{loc}^1(\R \times \T_L) := \big\{ \psi \in H_\text{loc}^1(\R^2) : \psi
\text{ is } L\text{-periodic with respect to its second variable } y
\big\}.
\end{equation}
For our analysis, it is convenient to work on a fixed domain independently of $L.$ 
For that purpose, we write $\T$ instead of $\T_L$ when $L = 1$, and given
a function $\psi \in H_\text{loc}^1(\R \times \T)$ and a real parameter $\lambda > 0$, we introduce the rescaled version of the Ginzburg-Landau energy given by
\begin{equation}
\label{def:E-energy}
E_\lambda(\psi) := \frac{1}{2} \int_{\R \times \T} \big( |\partial_x \psi|^2 +
\lambda^2 |\partial_y \psi|^2 \big) + \frac{1}{4} \int_{\R \times \T}
\big( 1 - |\psi|^2 \big)^2.
\end{equation}
Up to a multiplicative factor $\lambda$, the
rescaled energy $E_\lambda(\psi)$ is equal to the Ginzburg-Landau energy of the
function $\psi_\lambda(x, y) = \psi(x, \lambda y)$ on the product space $\R
\times \T_{L}$, where $L = 1/\lambda.$ 

Corresponding to the rescaled Ginzburg-Landau energy $E_\lambda$ is the energy set
\begin{equation}
\label{def:X-space}
X(\R \times \T) := \big\{ \psi \in H_\text{loc}^1(\R \times \T) : \nabla \psi \in L^2(\R \times \T) \text{ and } 1 - |\psi|^2 \in L^2(\R \times \T) \big\}.
\end{equation}

The untwisted momentum $[P]$ along the direction $x$ can be extended to $X(\R
\times \T)$ (see Appendices~\ref{sec:energy-set} and~\ref{sec:momentum}). When
$\partial_x \psi$ has compact support, it satisfies the identity
$$
[P](\psi) = \frac{1}{2} \int_{\R \times \T} \langle i \partial_x \psi, \psi \rangle_\C \quad \text{ modulo } \pi.
$$

For $p \in \R/\pi\Z$, we next consider the minimization problem under constraint
\begin{equation}
\label{def:I-problem}
\boI_\lambda(p) := \inf \big\{ E_\lambda(\psi) : \psi \in X(\R \times \T) \text{ s.t. } [P](\psi) = p \big\}.
\end{equation}

Our main result is

\begin{thm}
\label{thm:crit-val}
Let $p \in \R/\pi \Z.$ There exists $\lambda_p > 0$ such that the following statements hold.

$(i)$ For any $\lambda \geq \lambda_p$, the minimal value $\boI_\lambda(p)$ is equal to
$$
\boI_\lambda(p) = \gI(p).
$$
The dark soliton $\gu_{c_p}$ is a minimizer of the corresponding minimization problem. When $\lambda > \lambda_p$, it is the unique minimizer up to translation and phase shift.

$(ii)$ For any $0 < \lambda < \lambda_p$, the minimal value $\boI_\lambda(p)$ satisfies
$$
\boI_\lambda(p) < \gI(p),
$$
and there does not exist any minimizer depending only on the variable $x$.
\end{thm}

Note that, when $0 < \lambda < \lambda_p$, Theorem~\ref{thm:crit-val} makes
no claim about the existence of minimizers for $\boI_\lambda(p)$, it only
asserts that potential candidates must be truly two-dimensional. The fact 
that minimizers do exist in such cases will be the object of a future work.

Note also that our arguments do not prevent the possible existence of a truly
two-dimensional minimizer for $\lambda = \lambda_p$. 

We have stated Theorem~\ref{thm:crit-val} in the case of the spatial 
domain $\R \times \T.$ With minor modifications, the proofs carry
over to the case of $\R \times \T^2,$ and presumably also to $\R \times M$ 
where $M$ is any compact Riemannian manifold of dimension $d \leq 2$.

Linear transverse instability of solitons for a number of dispersive models,
including the Gross-Pitaevskii equation, was proved by 
F.~Rousset and N.~Tzvetkov in~\cite{RousTzv3} (see also~\cite{RousTzv1}
for the general Hamiltonian framework concerning nonlinear transverse instability). In
particular, although they did not consider their variational characterization, it follows
from~\cite[Theorem~3.3]{RousTzv3} that given a dark soliton $\gu_{c_p}$, there 
exists $\lambda_p>0$ such that $\gu_{c_p}$ is not a minimizer for
$\boI_{\lambda}$ when $\lambda = \lambda_p / k$ for some $k \in \N^*.$ 

In the next section we sketch the main arguments in the proof of
Theorem~\ref{thm:crit-val}. We follow a strategy developed by 
S.~Terracini, N.~Tzvetkov and N.~Visciglia~\cite{TerTzVi1} 
in the different context of the nonlinear Schr\"odinger equations on 
product spaces. In Section~\ref{sec:proof-crit-val}, we provide 
the full details of our proofs. A number of properties and 
ingredients related to the energy spaces $X(\R)$ and $X(\R \times \T)$ 
as well as the untwisted momentum $[P]$, which we found of
independent interest, are gathered in Appendices~\ref{sec:dimension-one}, 
\ref{sec:energy-set} and~\ref{sec:momentum}.

%%%%%%%%%%%%%%%%%%%%%%%%%%%%%%%%%%%%
%%%%%%%%%%%%%%%%%%%%%%%%%%%%%%%%%%%%
%%%%%%%%%%%%%%%%%%%%%%%%%%%%%%%%%%%%
\section{Sketch of the proof of Theorem~\ref{thm:crit-val}}
\label{sec:sketch-crit-val}
%%%%%%%%%%%%%%%%%%%%%%%%%%%%%%%%%%%%
%%%%%%%%%%%%%%%%%%%%%%%%%%%%%%%%%%%%
%%%%%%%%%%%%%%%%%%%%%%%%%%%%%%%%%%%%

The starting point is to check that the minimal energy $\boI_\lambda(p)$ tends
to the 1D minimal energy $\gI(p)$ as $\lambda \to + \infty$. In
this limit, we show that suitable extractions of minimizing sequences tend to
the dark soliton $\gu_{c_p}$, up to possible translation and phase shift. The
key ingredient of the proof is then to check that these dark solitons are strict
local minimizers of the variational problem corresponding to the minimal energy
$\boI_\lambda(p)$. In this case, the functions in the previous minimizing
sequences must be equal to a dark soliton for $\lambda$ large enough. This
property is sufficient to conclude that the minimal energy $\boI_\lambda(p)$ is 
exactly the energy $\gI(p)$ of dark solitons. 

We describe now this strategy with additional details. In a number of places in the sequel, it is convenient to identify $\R / \pi\Z$ with the interval $(- \pi/2, \pi/2].$ The first argument is to describe the behaviour of the minimizing energy with respect to the constraint $p \in (- \pi/2, \pi/2]$. In this direction, we show

\begin{lem}
\label{lem:prop-I-p}
Let $\lambda > 0$ be fixed. The function $\boI_\lambda$ is well-defined on $(- \pi/2, \pi/2]$ and it satisfies
\begin{equation}
\label{eq:sym-I}
\boI_\lambda(- p) = \boI_\lambda(p),
\end{equation}
for any $p \in (- \pi/2, \pi/2)$. Moreover, this function is Lipschitz
continuous, with Lipschitz constant at most $\sqrt{2}$, and it is bounded by
\begin{equation}
\label{eq:estim-sup-I}
\boI_\lambda(p) \leq \gI(p) < \sqrt{2}\, p,
\end{equation}
for any $p \in (0, \pi/2]$.
\end{lem}

The identities and inequalities in Lemma~\ref{lem:prop-I-p} are based on the introduction of suitable test functions. An important tool for exhibiting these special functions is the following approximation result.

\begin{lem}
\label{lem:dens-I}
Let $\lambda > 0$ be fixed. Given a function $\psi \in X(\R \times \T)$, there exists a sequence $(\psi_n)_{n \geq 0}$ of smooth functions in $X(\R \times \T)$, which satisfies the following properties.

$(i)$ Given any integer $n \geq 0$, there exist two positive numbers $R_n^\pm$ and two numbers $\theta_n^\pm$ for which
$$
\psi_n(x, y) = e^{i \theta_n^\pm},
$$
for any $\pm x \geq \pm R_n^\pm$ and any $y \in \T$.

$(ii)$ We have
$$
[P](\psi_n) = [P](\psi),
$$
for any $n \geq 0$.

$(iii)$ We also have
$$
E_\lambda(\psi_n) \to E_\lambda(\psi),
$$
as $n \to \infty$.
\end{lem}

With Lemma~\ref{lem:dens-I} at hand, we can describe more precisely the minimizing problems $\boI_\lambda(p)$ in the limit $\lambda \to + \infty$. Fix a number $p \in (- \pi/2, \pi/2]$, with $p \neq 0$, consider a sequence $(\varepsilon_n)_{n \geq 0}$ such that $\varepsilon_n \to 0$ and another sequence $(\lambda_n)_{n \geq 0}$ such that $\lambda_n \to +\infty$. We can apply Lemma~\ref{lem:dens-I} to exhibit a sequence $(\psi_n)_{n \geq 0}$ of smooth functions in $X(\R \times \T)$, satisfying statement $(i)$ in Lemma~\ref{lem:dens-I}, and such that
\begin{equation}
\label{eq:cond-psi-n}
[P](\psi_n) = p \quad \text{ modulo } \pi, \text{ and } \quad E_{\lambda_n}(\psi_n) \leq \boI_{\lambda_n}(p) + \varepsilon_n,
\end{equation}
for any $n \geq 0$. This sequence allows us to obtain the following information
on the limit $\lambda_n \to + \infty$.

\begin{lem}
\label{lem:asympt-I}
We have
\begin{equation}
\label{eq:borne-I}
\boI_{\lambda_n}(p) \to \gI(p),
\end{equation}
and
\begin{equation}
\label{eq:limit-zero-y}
\lambda_n^2 \int_{\R \times \T} |\partial_y \psi_n|^2 \to 0,
\end{equation}
as $n \to \infty$.
\end{lem}

Once we have established the convergence of the minimizing energies
$\boI_{\lambda_n}(p)$ towards the 1D energy $\gI(p)$, it is natural
to establish the convergence of the minimizing sequences $(\psi_n)_{n \geq 0}$
towards a minimizer of the limit energy. In order to achieve a strong
convergence result, we show that we can assume that the 
functions $\psi_n$ satisfy the Pohozaev identity
\begin{equation}
\label{eq:Pohozaev}
\frac{1}{2} \int_{\R \times \T} |\partial_x \psi_n|^2 = \frac{\lambda_n^2}{2}
\int_{\R \times \T} |\partial_y \psi_n|^2 + \frac{1}{4} \int_{\R \times \T}
\big( 1 - |\psi_n|^2 \big)^2.
\end{equation}
This identity is strongly related to the first variation of the Ginzburg-Landau
energy under the group of dilations along the $x$ variable. More precisely, we
prove: 

\begin{lem}
\label{lem:Pohozaev}
Let $n \geq 0$. There exists a positive number $\tau_n$ such that the functions $\xi_n$ given by
\begin{equation}
\label{def:xi-n}
\xi_n(x, y) = \psi_n(\tau_n x, y), 
\end{equation}
for any $x \in \R$ and $y \in \T$ satisfy
$$
[P](\xi_n) = [P](\psi_n) \quad \text{and} \quad E(\xi_n) \leq E(\psi_n), 
$$
as well as the Pohozaev identity
$$
\frac{1}{2} \int_{\R \times \T} |\partial_x \xi_n|^2 = \frac{\lambda_n^2}{2} \int_{\R \times \T} |\partial_y \xi_n|^2 + \frac{1}{4} \int_{\R \times \T} \big( 1 - |\xi_n|^2 \big)^2.
$$
\end{lem}

In view of~\eqref{def:xi-n}, the functions $\xi_n$ remain smooth, belong to $X(\R \times \T)$ and satisfy statement $(i)$ in Lemma~\ref{lem:dens-I}. It also follows from Lemma~\ref{lem:Pohozaev} that they still satisfy the conditions in~\eqref{eq:cond-psi-n}. In particular, the conclusions of Lemma~\ref{lem:asympt-I} are available for these functions, and moreover, they satisfy the Pohozaev identity in~\eqref{eq:Pohozaev}. Replacing the functions $\psi_n$ by the functions $\xi_n$ if necessary, we will simply assume in the sequel that the functions $\psi_n$ satisfy the Pohozaev identity in~\eqref{eq:Pohozaev}.

With this identity at hand, we can establish that a subsequence of the functions $\psi_n$ converges towards a minimizer of the 1D problem $\gI(p)$. More precisely, we introduce the distance $d_{c_p}$ given by
$$
d_{c_p}(\psi_1, \psi_2)^2 = \big\| \nabla \psi_1 - \nabla \psi_2 \big\|_{L^2}^2
+ \big\| \eta_{c_p}^\frac{1}{2} (\psi_1 - \psi_2) \big\|_{L^2}^2 + \big\| (1 -
|\psi_1|^2) - (1 - |\psi_2|^2) \big\|_{L^2}^2,
$$
for functions $\psi_1$ and $\psi_2$ in $X(\R \times \T)$. In the second term,
the weight $\eta_{c_p}$ is given by the expression 
\begin{equation}
\label{def:eta-c}
\eta_c(x) = 1 - |\gu_c(x)|^2 = \frac{2 - c^2}{2 \cosh \big( \frac{\sqrt{2 -
c^2}}{2} x \big)^2},
\end{equation}
which is available for any $|c| < \sqrt{2}$. We refer to Appendix~\ref{sec:energy-set} for more detail about the metric structure corresponding to the distance $d_{c_p}$. Using this distance, we show

\begin{prop}
\label{prop:conv-min-p}
There exist a sequence of real numbers $(a_n)_{n \geq 0}$ and a number $\theta \in \R$ such that there exists an extraction $\varphi : \N \to \N$ for which
$$
d_{c_p} \big( e^{i \theta} \psi_{\varphi(n)}(\cdot - a_{\varphi(n)}, \cdot), \gu_{c_p} \big) \to 0,
$$
as $n \to \infty$.
\end{prop}

Given an arbitrary function $\psi \in X(\R \times \T)$, we can introduce its Fourier coefficients with respect to the variable $y$,
$$
\hat{\psi}_k(x) := \int_0^1 \psi(x, y) e^{- 2 i \pi k y} \, dy,
$$
and decompose the function $\psi$ as a Fourier series
$$
\psi(x, y) = \sum_{k \in \Z} \hat{\psi}_k(x) e^{2 i \pi k y}.
$$
We check that the Fourier coefficient $\hat{\psi}_0$ lies in $X(\R)$, while
the difference $w_0 = \psi - \hat{\psi}_0$ is in $H^1(\R \times \T)$. Due to the
orthogonality of the functions $\hat{\psi}_0$ and $w_0$, the 
untwisted momentum $[P]$ can be extended from $X(\R)$ to $X(\R \times \T)$ by the
expression 
\begin{equation}
\label{def:P-momentum}
[P](\psi) = [P](\hat{\psi}_0) + \frac{1}{2} \int_\R \langle i \partial_x
w_0, w_0 \rangle_\C \quad \text{ modulo } \pi.
\end{equation}
Given any positive number $\alpha$, we denote
\begin{equation}
\label{def:V-alpha}
\boV_p(\alpha) := \Big\{ \psi = \hat{\psi}_0 + w_0 \in X(\R \times \T) \text{ s.t. } \inf_{(a, \theta) \in \R^2} d_{c_p} \big( e^{i \theta} \hat{\psi}_0(\cdot - a), \gu_{c_p} \big) < \alpha \text{ and } \| w_0 \|_{H^1} < \alpha \Big\},
\end{equation}
where we have set as before $\hat{\psi}_0(x) = \int_\T \psi(x, y) \, dy$ for any
function $\psi \in X(\R \times \T)$. In view of statement $(i)$ in Lemma~\ref{lem:conv-1-2}, we can rephrase Proposition~\ref{prop:conv-min-p} as the fact that there exists an integer $N_\alpha$ such that
\begin{equation}
\label{eq:V-alpha-psi-n}
\psi_{\varphi(n)} \in \boV_p(\alpha),
\end{equation}
for any $n \geq N_\alpha$. Once this property is established, we show that the
profile $\gu_{c_p}$ minimizes the energy $E_\lambda$ at fixed momentum $p$ in
the open set $\boV_p(\alpha)$, provided $\alpha$ is sufficiently small and
$\lambda$ is sufficiently large. More precisely, we prove

\begin{prop}
\label{prop:loc-min}
Let $p \in (- \pi/2, \pi/2]$, with $p \neq 0$. There exist two positive numbers
$\alpha_p$ and $\lambda_p$ such that, given any function $\psi \in \boV_p(\alpha_p)$, with $[P](\psi) = p$ modulo $\pi$, we have
\begin{equation}
\label{eq:local-mini-p}
E_\lambda(\psi) \geq E_\lambda(\gu_{c_p}) = \gI(p),\qquad \forall \lambda \geq
\lambda_p.
\end{equation}
Moreover, equality holds in~\eqref{eq:local-mini-p} if and only if $\psi(x) = e^{- i \theta} \gu_{c_p}(x + a)$ for some $a \in \R$ and $\theta \in \R$.
\end{prop}

For Proposition~\ref{prop:loc-min} we use the stability properties of the solitons $\gu_{c_p}$ with respect to the 1D Gross-Pitaevskii flow. In~\cite{BetGrSm1, GravSme1}, the orbital stability of the dark solitons was derived from the coercivity of the functional $E - c_p [P]$ in the neighbourhood of the profiles $\gu_{c_p}$. Extending this coercivity property to the sets $\boV_p(\alpha)$ for $\alpha$ small enough requires to control the dependence on the variable $y$. For $\lambda$ large enough, this can be done by using the coercivity provided by the term $\lambda^2 \int_{\R \times \T} |\partial_y \psi|^2$ in the energy $E_\lambda(\psi)$. The functional $E_\lambda - c_p [P]$ is then coercive on the sets $\boV_p(\alpha)$ and we obtain~\eqref{eq:local-mini-p} when the untwisted momentum $[P]$ is moreover fixed. 

Combining~\eqref{eq:V-alpha-psi-n} and~\eqref{eq:local-mini-p}, we are finally led to
$$
\boI_\lambda(p) \geq \gI(p),
$$
for $\lambda \geq \lambda_p$. In view of~\eqref{eq:estim-sup-I}, these two quantities are equal as we have claimed in statement $(i)$ of Theorem~\ref{thm:crit-val}.

Before concluding the proof of Theorem~\ref{thm:crit-val}, we need to precise the behaviour of the minimal energy $\boI_\lambda(p)$ with respect to the parameter $\lambda$. In this direction, we establish

\begin{lem}
\label{lem:prop-I-lambda}
Let $p \in (- \pi/2, \pi/2]$ be fixed. The function $\lambda \mapsto \boI_\lambda(p)$ is non-decreasing and continuous on $\R_+^*$, with
\begin{equation}
\label{eq:lim-I-0}
\boI_\lambda(p) \to 0,
\end{equation}
as $\lambda \to 0$.
\end{lem}

With Lemma~\ref{lem:prop-I-lambda} at hand, we are in position to complete the proof Theorem~\ref{thm:crit-val}.

\begin{proof}[End of the proof of Theorem~\ref{thm:crit-val}]
Set $\Lambda := \{ \lambda \in (0, + \infty) \text{ s.t. } \boI_\mu(p) = \gI(p)
\text{ for any } \mu \geq \lambda \}$ and $\lambda_p = \inf \Lambda$. We have
just shown in Proposition~\ref{prop:loc-min} that $\Lambda$ is non-empty. Its infimum $\lambda_p$ cannot be equal to $0$ due to~\eqref{eq:lim-I-0}. Hence, $\lambda_p$ is positive and moreover, a minimum by continuity of the map $\lambda \mapsto \boI_\lambda(p)$.

Since this map is also non-decreasing, the minimal value $\boI_\lambda(p)$ is strictly less than $\gI(p)$ when $0 < \lambda < \lambda_p$. Moreover, if a function $\psi \in X(\R \times \T)$, with $[P](\psi) = p$ modulo $\pi$, only depends on the variable $x$, then it follows from Proposition~\ref{prop:exist-min1} that
$$
E_\lambda(\psi) = E(\psi) \geq \gI(p) > \boI_\lambda(p).
$$
Therefore, a possible minimizer cannot only depend on the variable $x$.

When $\lambda \geq \lambda_p$ instead, we have $\boI_\lambda(p) = \gI(p) = E_\lambda(\gu_{c_p})$, so that the profile $\gu_{c_p}$ is a minimizer of the minimization problem~\eqref{def:I-problem}. For $\lambda > \lambda_p$, assume for the sake of a contradiction the existence of a minimizer $\psi \in X(\R \times \T)$ such that $E_\lambda(\psi) = \boI_\lambda(p) = \gI(p)$, $[P](\psi) = p$ modulo $\pi$ and
$$
\int_{\R \times \T} |\partial_y \psi|^2 \neq 0.
$$
For $\lambda_p < \mu < \lambda$, we obtain
$$
\boI_\mu(p) \leq E_\mu(\psi) = E_\lambda(\psi) + \frac{\mu^2 - \lambda^2}{2} \int_{\R \times \T} |\partial_y \psi|^2 < E_\lambda(\psi) = \gI(p),
$$
which contradicts the definition of the minimum $\lambda_p$. Therefore, a possible minimizer cannot depend on the variable $y$, so that it minimizes the 1D Ginzburg-Landau energy at fixed untwisted momentum. In view of Proposition~\ref{prop:exist-min1}, the profile $\gu_{c_p}$ is therefore the unique minimizer of the minimization problem~\eqref{def:I-problem} up to translation and phase shift. This concludes the proof of Theorem~\ref{thm:crit-val}.
\end{proof}

%%%%%%%%%%%%%%%%%%%%%%%%%%%%%%%%%%%
%%%%%%%%%%%%%%%%%%%%%%%%%%%%%%%%%%%
%%%%%%%%%%%%%%%%%%%%%%%%%%%%%%%%%%%
\section{Details in the proof of Theorem~\ref{thm:crit-val}}
\label{sec:proof-crit-val}
%%%%%%%%%%%%%%%%%%%%%%%%%%%%%%%%%%%
%%%%%%%%%%%%%%%%%%%%%%%%%%%%%%%%%%%
%%%%%%%%%%%%%%%%%%%%%%%%%%%%%%%%%%%

In Appendix~\ref{sec:dimension-one}, we introduce the set of
non-vanishing functions
$$
NV\!X(\R) := \big\{ \psi \in X(\R) \text{ s.t. } \inf_{x \in \R} |\psi(x)| > 0 \big\},
$$
and show that, for $\psi = \rho e^{i\theta} \in NV\!X(\R)$, the momentum 
$$
P(\psi) = \frac12 \int_\R (1-\rho^2) \theta' 
$$
is well-defined and satisfies $P(\psi) = [P](\psi)$ modulo $\pi.$ This set is
used throughout for the proof of Theorem~\ref{thm:crit-val}.

%%%%%%%%%%%%%%%%%%%%%%%%%%%%
%%%%%%%%%%%%%%%%%%%%%%%%%%%%
\subsection{Proof of Lemma~\ref{lem:dens-I}}
%%%%%%%%%%%%%%%%%%%%%%%%%%%%
%%%%%%%%%%%%%%%%%%%%%%%%%%%%

Adapting the argument in~\cite[Lemma 3.3]{BetGrSa1} we have:

\begin{lem}
\label{lem:scal-test-funct}
Let $\lambda > 0$, $p \in \R$ and $\alpha \in \R$ be fixed. There exists a
sequence $(\psi_n)_{n \in \N}$ of smooth functions in $NV\!X(\R)$ such that the functions $\psi_n - e^{i \alpha}$ are compactly supported, and with
$$
P(\psi_n) = p \quad \text{and} \quad E_\lambda(\psi_n) = E(\psi_n) \to
\sqrt{2}\, |p|,
$$
as $n \to \infty$.
\end{lem}

This lemma relies on constructing special test functions by a scaling argument. For the sake of completeness, we provide the following details.

\begin{proof}
We argue as in the proof of~\cite[Lemma 3.3]{BetGrSa1}. Assume first that $p$ is positive. Consider a function $\xi \in \boC_c^\infty(\R)$ and two positive numbers $\mu$ and $\varepsilon$ such that $\mu \varepsilon \| \partial_x \xi \|_{L^\infty(\R)} < 1$. Set
$$
\rho(x, y) = 1 - \mu \varepsilon \, \partial_x \xi(\varepsilon x), \quad \theta(x, y) = \alpha + \sqrt{2} \mu \, \xi(\varepsilon x) \quad \text{and} \quad \psi(x, y) = \rho(x, y) e^{i \theta(x, y)},
$$
for any $(x, y) \in \R \times \T$. We compute
\begin{equation}
\label{demarre}
\frac{1}{2} \int_{\R \times \T} |\partial_x \psi|^2 = \frac{1}{2} \int_{\R \times \T} \big( (\partial_x \rho)^2 + \rho^2 (\partial_x \theta)^2 \big) = \frac{\mu^2 \varepsilon^3}{2} \int_\R (\partial_{xx} \xi)^2 + \mu^2 \varepsilon \int_\R \big( 1 - \mu \varepsilon \, \partial_x \xi \big)^2 (\partial_x \xi)^2,
\end{equation}
and
\begin{equation}
\label{turgis}
\frac{1}{4} \int_{\R \times \T} \big( 1 - |\psi|^2 \big)^2 = \mu^2 \varepsilon \int_{\R \times \T} (\partial_x \xi)^2 - \mu^3 \varepsilon^2 \int_{\R \times \T} (\partial_x \xi)^3 + \frac{\mu^4 \varepsilon^3}{4} \int_{\R \times \T} (\partial_x \xi)^4.
\end{equation}
Hence, the function $\psi$ lies in $NV\!X(\R)$, so that, by definition~\eqref{def:P-NV} and Lemma~\ref{lem:P-1-2}, its momentum is given by
$$
P(\psi) = \frac{1}{2} \int_{\R \times \T} \big( 1 - \rho^2 \big) \partial_x \phi = \sqrt{2} \mu^2 \varepsilon \int_{\R \times \T} (\partial_x \xi)^2 - \frac{\mu^3 \varepsilon^2}{\sqrt{2}}\int_{\R \times \T} (\partial_x \xi)^3.
$$
We now assume that the $L^2$-norm of the derivative $\partial_x \xi$ is equal to
$1$ and we choose $\mu_n = n$ for a given integer $n$. At least when $n$
is large enough, we can find a positive number $\varepsilon_n$ such that
$P(\psi) = p$. Moreover, we have
$$
\varepsilon_n \sim \frac{p}{\sqrt{2} n^2},
$$
as $n \to \infty$. In particular, we check that
$$
\mu_n \varepsilon_n \to 0,
$$
as $n \to \infty$, so that the condition $\mu_n \varepsilon_n \| \partial_x \xi \|_{L^\infty(\R)} < 1$ is indeed satisfied for $n$ large enough. In view of~\eqref{demarre} and~\eqref{turgis}, we also obtain
$$
E_\lambda(\psi) \underset{n \to \infty}{\to} \sqrt{2} p.
$$
In conclusion, the functions $\psi_n = \psi$ satisfy all the statements in
Lemma~\ref{lem:scal-test-funct} for $p$ positive. When $p$ is negative,
the functions $\tilde \psi_n = e^{2 i \alpha} \overline{\psi_n}$ also
satisfy these conclusions, while for $p=0$ it suffices to take $\psi_n =
1$. This completes the proof of Lemma~\ref{lem:scal-test-funct}.
\end{proof}

Combining Lemma~\ref{lem:scal-test-funct} with Corollary~\ref{cor:dens-X}, we are now in position to complete the proof of Lemma~\ref{lem:dens-I}.

\begin{proof}[End of the proof of Lemma~\ref{lem:dens-I}]
Going to Corollary~\ref{cor:dens-X} and Lemma~\ref{lem:cont-P2}, we can find a
sequence of smooth functions $\tilde{\psi}_n$ in $X(\R \times \T)$,
which satisfy statements $(i)$ and $(iii)$ of Lemma~\ref{lem:dens-I} for
numbers $\tilde R_n^\pm$ and $\tilde \theta_n^\pm$, as well as
\begin{equation}
\label{senechal}
[P](\tilde{\psi}_n) \to [P](\psi),
\end{equation}
in the limit $n \to \infty$. Hence we are reduced to check that we can modify
the functions $\tilde{\psi}_n$, so that their momentum is exactly equal
to the momentum of $\psi$. When these two quantities are actually equal,
we simply set $\psi_n = \tilde{\psi}_n$. When they are not, we invoke
Lemma~\ref{lem:scal-test-funct} with $\alpha = \tilde{\theta}_n^+$ and $p_n \in
(- \pi/2, \pi/2]$ such that $p_n = [P](\tilde{\psi}_n) - [P](\psi)$
modulo $\pi$. This provides a smooth function $\check{\psi}_n$ such that
the function $\check{\psi}_n - e^{i \tilde{\theta}_n^+}$ is compactly supported
in an interval of the form $[- \check{R}_n^-, \check{R}_n^+]$, with
$P(\check{\psi}_n) = p_n$ and $E_\lambda(\check{\psi}_n) \leq \sqrt{2} |p_n|$. We next set
$$
\psi_n(x, y) = \begin{cases} \tilde{\psi}_n(x, y) & \text{if } x \leq \tilde{R}_n^+ + 1, \\ \check{\psi}_n(x - \tilde{R}_n^+ - \check{R}_n^- - 2, y) & \text{if } x \geq \tilde{R}_n^+ + 1. \end{cases}
$$
By construction, the function $\psi_n$ is smooth, belongs to $X(\R \times \T)$ and satisfies statement $(i)$ of Lemma~\ref{lem:dens-I}. Moreover, it follows from Lemmas~\ref{lem:def-P-2} and~\ref{lem:val-P-2} that
$$
[P](\psi_n) = [P](\tilde{\psi}_n) + P(\check{\psi}_n) = [P](\psi),
$$
modulo $\pi$. Finally, we also derive from~\eqref{senechal} that
$$
E_\lambda(\psi_n) = E_\lambda(\tilde{\psi}_n) + E_\lambda(\check{\psi}_n) \to E_\lambda(\psi),
$$
as $n \to \infty$. This concludes the proof of Lemma~\ref{lem:dens-I}.
\end{proof}

%%%%%%%%%%%%%%%%%%%%%%%%%%%%%
%%%%%%%%%%%%%%%%%%%%%%%%%%%%%
\subsection{Proof of Lemma~\ref{lem:prop-I-p}}
%%%%%%%%%%%%%%%%%%%%%%%%%%%%%
%%%%%%%%%%%%%%%%%%%%%%%%%%%%%

For an arbitrary $p \in (- \pi/2, \pi/2]$, there exists at least some function in $X(\R)$
(and hence in $X(\R \times \T)$) with untwisted momentum equal to $p$. Indeed,
if $p \neq 0$ it suffices to consider $\gu_{c_p}$, and for $p = 0$ one can take
a constant function with unit modulus. As a consequence, the infimum of the minimization problem~\eqref{def:I-problem} is finite and the corresponding minimal energy $\boI_\lambda(p)$ is well-defined on $(- \pi/2, \pi/2]$.

Observe next that the energy $E_\lambda$ is non-negative, so that the minimal energy $\boI_\lambda$ is also non-negative. Note similarly that when a function $\psi \in X(\R \times \T)$ is decomposed as $\psi = \hat{\psi}_0 + w_0$ according to Proposition~\ref{prop:X-1-2}, then its conjugate $\overline{\psi}$ is given by $\overline{\psi} = \overline{\hat{\psi}_0} + \overline{w_0}$. In view of Lemma~\ref{lem:def-P-1}, ~\eqref{eq:dev-E-1-2} and~\eqref{def:P-2}, this gives
$$
E_\lambda \big( \overline{\psi} \big) = E_\lambda \big( \psi \big) \quad \text{and} \quad [P] \big( \overline{\psi} \big) = - [P]\big( \psi \big),
$$
so that by~\eqref{def:I-problem}, the function $\boI_\lambda$ satisfies~\eqref{eq:sym-I}.

In order to establish~\eqref{eq:estim-sup-I}, we consider the minimizer
$\gu_{c_p}$ of Proposition~\ref{prop:exist-min1} as a test function for the
minimization problem $\boI_\lambda(p)$. Since $[P](\gu_{c_p}) = p$ modulo
$\pi$, we infer from Proposition~\ref{prop:boI-problem} that 
\begin{equation}
\label{eq:gE-val}
\gI(p) = E(\gu_{c_p}) = \frac{1}{3} \big(2 - c_p^2 \big)^\frac{3}{2},
\end{equation}
and therefore
$$
\boI_\lambda(p) \leq E_\lambda(\gu_{c_p}) = E(\gu_{c_p}) = \gI(p).
$$
At this stage, consider the function $\Xi$ given by
\begin{equation}
\label{def:Xi}
\Xi(c) = \frac{\pi}{2} - \arctan \Big( \frac{c}{\sqrt{2 - c^2}} \Big) - \frac{c}{2} \sqrt{2 - c^2},
\end{equation}
for $0 \leq c < \sqrt{2}$. This function is smooth on $[0, \sqrt{2})$, with
\begin{equation}
\label{eq:deriv-Chi}
\Xi'(c) = - \sqrt{2 - c^2} \neq 0.
\end{equation}
Since $\Xi(c) \to \pi/2$ as $c \to 0$, and $\Xi(c) \to 0$ as $c \to \sqrt{2}$,
the function $\Xi$ is a smooth diffeomorphism from $[0, \sqrt{2})$ to $(0,
\pi/2]$. Recall that $c_p$ was defined by the identity $[P](\gu_{\gc_p}) = p$ modulo
$\pi$. By Proposition~\ref{prop:boI-problem}, we have $c_p := \text{sign}(p) \gc_p$ where $\gc_p$ is the unique number in $[0, \sqrt{2})$ such that
\begin{equation}
\label{def:gc-p}
\frac{\pi}{2} - \arctan \bigg( \frac{\gc_p}{(2 - \gc_p^2)^{1/2}} \bigg) -
\frac{\gc_p}{2} (2 - \gc_p^2)^{1/2} = |p|.
\end{equation}
The function $p \to \gc_p$ is therefore well-defined and smooth on $(0, \pi/2]$, with
\begin{equation}
\label{eq:deriv-cp}
\frac{d \gc_p}{dp} = - \frac{1}{(2 - \gc_p^2)^{1/2}}.
\end{equation}
Moreover, we also have that $\gc_p \to \sqrt{2}$ as $p \to 0$. Going back to~\eqref{eq:gE-val}, we deduce that the function $\gI$ is smooth on $(0, \pi/2]$, with
$$
\gI(p) \to 0 = \gI(0),
$$
as $p \to 0$, and
\begin{equation}
\label{charles-V}
|\gI'(p)| = c_p < \sqrt{2},
\end{equation}
for $p \in (0, \pi/2]$. Hence we obtain that $\gI(p) < \sqrt{2} p.$

We finally turn to the Lipschitz continuity of the function $\boI_\lambda$. We
argue as in the proof of~\cite[Lemma 3.4]{BetGrSa1}. Fix two numbers $p$ and $q$
in $(0, \pi/2]$, with $p \neq q$. Since there exists a function $\psi \in X(\R
\times \T)$ such that $[P](\psi) = p$ modulo $\pi$, we can invoke
Lemma~\ref{lem:dens-I} to exhibit a minimizing sequence of smooth functions
$\tilde{\psi}_n$ for the problem $\boI_\lambda(p)$, which moreover satisfy
statement $(i)$ in Lemma~\ref{lem:dens-I} for numbers $\tilde{R}_n^\pm$ and
$\tilde{\theta}_n^\pm$. Going back to Lemma~\ref{lem:scal-test-funct}, there
also exists a sequence of smooth functions $\check{\psi}_n \in NV\!X(\R)$ such
that the function $\check{\psi}_n - e^{i \tilde{\theta}_n^+}$ is compactly supported in an interval of the form $[- \check{R}_n^-, \check{R}_n^+]$, with $[P](\check{\psi}_n) = q - p$ modulo $\pi$ and $E_\lambda(\check{\psi}_n) \to \sqrt{2} |q - p|$, as $n \to \infty$. Setting
$$
\psi_n(x, y) = \begin{cases} \tilde{\psi}_n(x, y) & \text{if } x \leq \tilde{R}_n^+ + 1, \\ \check{\psi}_n(x - \tilde{R}_n^+ - \check{R}_n^- - 2, y) & \text{if } x \geq \tilde{R}_n^+ + 1, \end{cases}
$$
we check that the function $\psi_n$ is smooth and belongs to $X(\R \times \T)$, with
$$
E_\lambda(\psi_n) = E_\lambda(\tilde{\psi}_n) + E_\lambda(\check{\psi}_n).
$$
Moreover, we can apply Lemmas~\ref{lem:def-P-2} and~\ref{lem:val-P-2} in order to compute
\begin{align*}
[P](\psi_n) = \frac{1}{2} \int_{\R \times \T} \big\langle i \partial_x \psi_n, \psi_n \big\rangle_\C + \frac{1}{2} \big( \theta_n^+ - \theta_n^- \big) = [P](\tilde{\psi}_n) + [P](\check{\psi}_n) = q,
\end{align*}
modulo $\pi$. Hence, we obtain
$$
\boI_\lambda(q) \leq E_\lambda(\psi_n) \leq E_\lambda(\tilde{\psi}_n) + E_\lambda(\check{\psi}_n),
$$
so that, in the limit $n \to \infty$,
$$
\boI_\lambda(q) \leq \boI_\lambda(p) + \sqrt{2} |q - p|.
$$
Since the numbers $p$ and $q$ are arbitrary chosen, we can permute them in the previous inequality and conclude that the function $\boI_\lambda$ is $\sqrt{2}$-Lipschitz on the interval $(0, \pi/2]$.

In the general case where $- \pi/2 < p, q \leq \pi/2$, we can combine~\eqref{eq:sym-I} and~\eqref{eq:estim-sup-I} in order to write
$$
\big| \boI_\lambda(q) - \boI_\lambda(p) \big| = \big| \boI_\lambda(|q|) - \boI_\lambda(|p|) \big| \leq \sqrt{2} \big| |q| - |p| \big| \leq \sqrt{2} |q - p|.
$$
Hence, the function $\boI_\lambda$ is Lipschitz continuous on $(- \pi/2, \pi/2]$, with Lipschitz constant at most $\sqrt{2}$. This completes the proof of Lemma~\ref{lem:prop-I-p}. \qed

%%%%%%%%%%%%%%%%%%%%%%%%%%%%%
%%%%%%%%%%%%%%%%%%%%%%%%%%%%%
\subsection{Proof of Lemma~\ref{lem:asympt-I}}
%%%%%%%%%%%%%%%%%%%%%%%%%%%%%
%%%%%%%%%%%%%%%%%%%%%%%%%%%%%

Our starting point is inequality~\eqref{eq:estim-sup-I}, which gives the upper bound 
\begin{equation}
\label{jean-le-bon}
E_{\lambda_n}(\psi_n) \leq \gI(p) + \varepsilon_n,
\end{equation}
in view of the second assumption in~\eqref{eq:cond-psi-n}. We next rely on formula~\eqref{eq:val-P-2} to obtain an easy-to-handle expression of the momentum $[P](\Psi_n)$. Going back to statement $(i)$ in Lemma~\ref{lem:dens-I}, there exist positive numbers $R_n^\pm$ and real numbers $\theta_n^\pm$ for which
$$
\psi_n(x, y) = e^{i \theta_n^\pm},
$$
for any $\pm x \geq \pm R_n^\pm$ and any $y \in \T$. According to Lemma~\ref{lem:val-P-2}, the momentum $[P](\psi_n)$ is then given by
\begin{equation}
\label{philippe-le-bon}
[P](\psi_n) = \gp(\psi_n) := \int_\T p_n(y) \, dy \quad \text{ modulo } \pi,
\end{equation}
with
\begin{equation}
\label{charles-le-mauvais}
p_n(y) := \frac{1}{2} \int_\R \langle i \partial_x \psi_n(x, y), \psi_n(x, y) \rangle_\C \, dx + \frac{1}{2} \big( \theta_n^+ - \theta_n^- \big) = [P] \big( \psi_n(\cdot, y) \big) \quad \text{ modulo } \pi,
\end{equation}
for almost any $y \in \T$. Since the functions $\psi_n$ are smooth and their derivatives are compactly supported, the functions $p_n$ in the previous definition are well-defined and smooth on $\T$, with
$$
p_n'(y) = \int_\R \langle i \partial_x \psi_n(x, y), \partial_y \psi_n(x, y) \rangle_\C \, dx,
$$
by integration by parts. Hence, we infer from the Cauchy-Schwarz inequality and~\eqref{jean-le-bon} that
$$
\int_\T |{p'_n}(y)| \, dy \leq \frac{1}{\lambda_n} E_{\lambda_n}(\psi_n) \leq \frac{1}{\lambda_n} \big( \gI(p) + \varepsilon_n \big),
$$
so that the Poincar\'e-Wirtinger inequality in~\cite{Brezis0} provides
\begin{equation}
\label{philippe-le-hardi}
\big\| p_n - \gp(\psi_n) \big\|_{L^\infty(\T)} \leq \int_\T |p_n'(y)| \, dy \to 0,
\end{equation}
as $n \to \infty$. At this stage, we write
\begin{equation}
\label{philippe-le-bel}
E_{\lambda_n}(\psi_n) = \int_\T E \big( \psi_n(\cdot, y) \big) \, dy +
\frac{\lambda_n^2}{2} \int_{\R \times \T} |\partial_y \psi_n|^2 \geq \int_\T \gI \big( q_n(y) \big) \, dy + \frac{\lambda_n^2}{2} \int_{\R \times \T} |\partial_y \psi_n^2|.
\end{equation}

For $p \neq \pi/2$, the number $q_n(y)$ is the unique number in the interval $(-
\pi/2, \pi/2]$ such that $q_n(y) = p_n(y)$ modulo $\pi$. Since $\gp(\psi_n) = p$
modulo $\pi$, we can invoke~\eqref{charles-V} and~\eqref{philippe-le-hardi} in 
order to obtain
\begin{equation}
\label{charles-VIII}
\big| \gI \big( q_n(y) \big) - \gI \big( p \big) \big| \leq \sqrt{2} \, \big| q_n(y) - p \big| = \sqrt{2} \, \big| p_n(y) - \gp(\psi_n) \big|,
\end{equation}
for $n$ large enough.

For $p = \pi/2$, the number $q_n(y)$ is the absolute value of the unique number $\tilde{q}_n(y)$ in the interval $(- \pi/2, \pi/2]$ such that $\tilde{q}_n(y) = p_n(y)$ modulo $\pi$. Using the fact that the minimal energy $\gI$ is an even function, formula~\eqref{philippe-le-bel} remains available for this definition of the numbers $q_n(y)$. Moreover, we can argue as in the case $p \neq \pi/2$, so as to get again
$$
\Big| \gI \Big( q_n(y) \Big) - \gI \Big( \frac{\pi}{2} \Big) \Big| \leq \sqrt{2} \, \Big| q_n(y) - \frac{\pi}{2} \Big| = \sqrt{2} \, |p_n(y) - \gp(\psi_n)|.
$$

In view of~\eqref{philippe-le-bel}, we conclude that
$$
E_{\lambda_n}(\psi_n) \geq \gI \big( p \big) - \sqrt{2} \big\| p_n - \gp(\psi_n) \big\|_{L^\infty(\T)} + \frac{\lambda_n^2}{2} \int_{\R \times \T} |\partial_y \psi_n^2|.
$$
Combining with~\eqref{jean-le-bon} and~\eqref{philippe-le-hardi}, we first deduce that
\begin{equation}
\label{jean-sans-peur}
\lambda_n^2 \int_{\R \times \T} |\partial_y \psi_n^2| \to 0,
\end{equation}
as $n \to \infty$. Moreover, it also follows from~\eqref{eq:estim-sup-I} and~\eqref{eq:cond-psi-n} that
$$
\gI(p) \geq \boI_{\lambda_n}(p) \geq \gI \big( p \big) - \varepsilon_n - \sqrt{2} \big\| p_n - \gp(\psi_n) \big\|_{L^\infty(\T)} + \frac{\lambda_n^2}{2} \int_{\R \times \T} |\partial_y \psi_n^2|.
$$
In view of~\eqref{philippe-le-hardi} and~\eqref{jean-sans-peur}, this is enough to obtain the convergence in~\eqref{eq:borne-I} and complete the proof of Lemma~\ref{lem:asympt-I}. \qed

%%%%%%%%%%%%%%%%%%%%%%%%%%%%%%
%%%%%%%%%%%%%%%%%%%%%%%%%%%%%%
\subsection{Proof of Lemma~\ref{lem:Pohozaev}}
%%%%%%%%%%%%%%%%%%%%%%%%%%%%%%
%%%%%%%%%%%%%%%%%%%%%%%%%%%%%%

It is classical that the Pohozaev identity in~\eqref{eq:Pohozaev} is based on applying the scaling $(x, y) \mapsto (\tau x, y)$ for positive numbers $\tau$. For a fixed integer $n \geq 0$, we therefore introduce the functions
$$
\xi_\tau(x, y) = \psi_n(\tau x, y).
$$
Using the notation of the proof of Lemma~\ref{lem:asympt-I}, we check that the functions $\xi_\tau$ are smooth on $\R \times \T$ and satisfy statement $(i)$ in Lemma~\ref{lem:dens-I} for the same numbers $\theta_n^\pm$ as the function $\psi_n$. Arguing as for~\eqref{philippe-le-bon}, their untwisted momentum $[P](\xi_\tau)$ is given by the formula
$$
[P](\xi_\tau) = \gp(\xi_\tau) := \frac{1}{2} \int_{\R \times \T} \langle i
\partial_x \xi_\tau(x, y), \xi_\tau(x, y) \rangle_\C \, dx \, dy + \frac{1}{2}
\big( \theta_n^+ - \theta_n^- \big) \quad \text{ modulo } \pi.
$$
By definition of the functions $\xi_\tau$ and by~\eqref{philippe-le-bon}, this quantity reduces to
\begin{equation}
\label{charles-le-fou}
\gp(\xi_\tau) = \frac{1}{2} \int_{\R \times \T} \langle i \partial_x \psi_n(x, y), \psi_n(x, y) \rangle_\C \, dx \, dy + \frac{1}{2} \big( \theta_n^+ - \theta_n^- \big) = \gp(\psi_n) = [P](\psi_n) \quad \text{ modulo } \pi,
\end{equation}
for any positive number $\tau$. Similarly, we compute
\begin{equation}
\label{charles-le-temeraire}
\frac{1}{2} \int_{\R \times \T} |\partial_x \xi_\tau|^2 = \frac{\tau}{2} \int_{\R \times \T} |\partial_x \psi_n|^2 := A_n \tau,
\end{equation}
and
\begin{equation}
\label{charles-le-travaillant}
\frac{\lambda_n^2}{2} \int_{\R \times \T} |\partial_y \xi_\tau|^2 + \frac{1}{4} \int_{\R \times \T} \big( 1 - |\xi_\tau|^2 \big)^2 = \frac{1}{\tau} \bigg( \frac{\lambda_n^2}{2} \int_{\R \times \T} |\partial_y \psi_n|^2 + \frac{1}{4} \int_{\R \times \T} \big( 1 - |\psi_n|^2 \big)^2 \bigg) := \frac{B_n}{\tau}.
\end{equation}
Observe here that $A_n \neq 0$. Otherwise, the function $\psi_n$ would not
depend on the variable $x$, so that the numbers $\theta_n^-$ and $\theta_n^-$
would also be equal. As a consequence, the quantity $\gp(\psi_n)$ in~\eqref{philippe-le-bon} would be equal to $0$, and not to $p$ modulo $\pi$. Since $A_n \neq 0$, we can combine~\eqref{charles-le-temeraire} and~\eqref{charles-le-travaillant} to derive that the energies $E(\xi_\tau)$ are minimal for $\tau$ being chosen as
$$
\tau_n = \sqrt{\frac{B_n}{A_n}}.
$$
It suffices then to set $\xi_n = \xi_{\tau_n}$ in order to obtain
$$
E(\xi_n) \leq E(\psi_n),
$$
by minimality, as well as the Pohozaev identity
$$
\frac{1}{2} \int_{\R \times \T} |\partial_x \xi_n|^2 = \sqrt{A_n B_n} = \frac{\lambda_n^2}{2} \int_{\R \times \T} |\partial_y \xi_n|^2 + \frac{1}{4} \int_{\R \times \T} \big( 1 - |\xi_n|^2 \big)^2 
$$
by~\eqref{charles-le-temeraire} and~\eqref{charles-le-travaillant}. In view of~\eqref{charles-le-fou}, this completes the proof of Lemma~\ref{lem:Pohozaev}. \qed

%%%%%%%%%%%%%%%%%%%%%%%%%%%%%%%%%%
%%%%%%%%%%%%%%%%%%%%%%%%%%%%%%%%%%
\subsection{Proof of Proposition~\ref{prop:conv-min-p}}
%%%%%%%%%%%%%%%%%%%%%%%%%%%%%%%%%%
%%%%%%%%%%%%%%%%%%%%%%%%%%%%%%%%%%

We go on with the notation of the proofs of Lemmas~\ref{lem:asympt-I} and~\ref{lem:Pohozaev}. Our first goal is to exhibit a number $y_* \in \T$ such that, up to a possible subsequence, the functions $\psi_n(\cdot, y_*)$ form an almost minimizing sequence for the 1D minimization problem $\gI(p)$. In view of~\eqref{charles-le-mauvais}, the untwisted momentum $[P](\psi_n(\cdot, y))$ is equal to $p_n(y)$ modulo $\pi$ for almost any $y \in \T$, so that our aim is to find a number $y_* \in \T$ such that
$$
p_n(y_*) \to p \quad \text{ modulo } \pi, \text{ and } \quad e_n(y_*) := E(\psi_n(\cdot, y_*)) \to \gI(p),
$$
as $n \to \infty$. In this direction, we first recall that $p = [P](\psi_n) = \gp(\psi_n)$ modulo $\pi$ by~\eqref{philippe-le-bon}. Going back to the proof of Lemma~\ref{lem:asympt-I}, and more precisely to~\eqref{philippe-le-hardi}, it follows that
\begin{equation}
\label{louis-XI}
p_n(y) \to p \quad \text{ modulo } \pi,
\end{equation}
as $n \to \infty$, uniformly with respect to $y \in \T$. We similarly deduce from~\eqref{charles-VIII} that
$$
\int_\T |e_n(y) - \gI(p)| \, dy \leq \int_\T |e_n(y) - \gI(q_n(y))| \, dy +
\sqrt{2} \int_\T |p_n(y) - \gp(\psi_n)| \, dy.
$$
Since $e_n(y) \geq \gI(q_n(y))$ by definition of the 1D minimal energy $\gI$, we infer again from~\eqref{charles-VIII} that
$$
\int_\T |e_n(y) - \gI(p)| \, dy \leq E_{\lambda_n}(\psi_n) - \gI(p) + 2 \sqrt{2}
\int_\T |p_n(y) - \gp(\psi_n)| \, dy.
$$
Invoking~\eqref{eq:cond-psi-n},~\eqref{eq:borne-I} and~\eqref{philippe-le-hardi}, we are led to
$$
\int_\T |e_n(y) - \gI(p)| \, dy \to 0,
$$
as $n \to \infty$. As a consequence, we can find a number $y_* \in \T$ such that, up to a possible subsequence, we have
$$
e_n(y_*) \to \gI(p),
$$
as $n \to \infty$. In view of~\eqref{louis-XI}, we conclude that the functions
$\psi_n(\cdot,y_*)$ form a minimizing sequence for $\gI(p)$. In particular, we
can apply the compactness results in~\cite[Theorem 3]{BetGrSa2}
and~\cite[Theorem 4]{BeGrSaS1} to this sequence. This provides a sequence of real numbers $(a_n)_{n \geq 0}$, as well as a number $\theta \in \R$, such that, up to a further subsequence,
\begin{equation}
\label{philippe-VI}
\begin{split}
& e^{i \theta} \, \psi_n(\cdot - a_n, y_*) \to \gu_{c_p} \text{ in } L_\text{loc}^\infty(\R), \\[5pt]
& 1 - |e^{i \theta} \, \psi_n(\cdot - a_n, y_*)|^2 \to 1 - |\gu_{c_p}|^2 \text{ in } L^2(\R), \\[5pt]
& e^{i \theta} \, \partial_x \psi_n(\cdot - a_n, y_*) \to \gu_{c_p}' \text{ in } L^2(\R),
\end{split}
\end{equation}
as $n \to \infty$. 

We now extend the convergence to any number $y \in \T$. This follows from the smoothness of the functions $\psi_n$, which guarantees that
$$
e^{i \theta} \, \psi_n(x - a_n, y) - e^{i \theta} \, \psi_n(x - a_n, y_*) = e^{i \theta} \int_{y_*}^y \partial_y \psi_n(x - a_n, y') \, dy'.
$$
Invoking the Cauchy-Schwarz inequality, we are led to
$$
\int_{- R}^R \big| e^{i \theta} \, \psi_n(x - a_n, y) - e^{i \theta} \, \psi_n(x - a_n, y_*) \big|^2 \, dx \leq \int_{\R \times \T} |\partial_y \psi_n(x', y')|^2 \, dx' \, dy',
$$
for any positive number $R$. In view of~\eqref{eq:limit-zero-y} and the first convergence in~\eqref{philippe-VI}, we deduce that
$$
\int_{- R}^R \big| e^{i \theta} \, \psi_n(x - a_n, y) - \gu_{c_p}(x) \big|^2 \, dx \to 0,
$$
as $n \to \infty$, uniformly with respect to $y \in \T$. This is enough to guarantee that the functions $e^{i \theta} \, \psi_n(\cdot - a_n, \cdot)$ converge to the function $\gu_{c_p}$ in $L_\text{loc}^2(\R \times \T)$.

At this stage, we rely on the second condition in~\eqref{eq:cond-psi-n} and the convergence in~\eqref{eq:borne-I} in order to claim that both the sequences $(e^{i \theta} \, \nabla \psi_n(\cdot - a_n, \cdot))_{n \geq 0}$ and $(1 - |\psi_n(\cdot - a_n, \cdot)|^2)_{n \geq 0}$ are bounded in $L^2(\R \times \T)$. Up to a further subsequence, we can find two functions $\Xi \in L^2(\R \times \T)$ and $\eta \in L^2(\R \times \T)$ such that
\begin{equation}
\label{charles-VII}
e^{i \theta} \, \nabla \psi_n(\cdot - a_n, \cdot) \rightharpoonup \Xi \text{ in } L^2(\R \times \T), \quad \text{and} \quad 1 - |\psi_n(\cdot - a_n, \cdot)|^2 \rightharpoonup \eta \text{ in } L^2(\R \times \T),
\end{equation}
as $n \to \infty$. Since $|z| \leq 1 + \big| 1 - |z|^2 \big|$ for any complex number $z$, the sequence $(e^{i \theta} \, \psi_n(\cdot - a_n, \cdot))_{n \geq 0}$ is also bounded in $H_\text{loc}^1(\R \times \T)$. Applying the Rellich theorem, we can find another function $\psi_\infty \in H_\text{loc}^1(\R \times \T)$ such that, up to a further subsequence,
\begin{equation}
\label{louis-XII}
e^{i \theta} \, \psi_n(\cdot - a_n, \cdot) \to \psi_\infty \text{ in } L_\text{loc}^q(\R \times \T),
\end{equation}
as $n \to \infty$, for any number $1 \leq q < + \infty$. Since this convergence holds for $q = 2$, the function $\psi_\infty$ is equal to $\gu_{c_p}$, and we deduce from~\eqref{charles-VII} and~\eqref{louis-XII} that $\Xi = \nabla \gu_{c_p}$ and $\eta = 1 - |\gu_{c_p}|^2$.

We now transform the weak convergences in~\eqref{charles-VII} into strong convergences. We first observe that
\begin{equation}
\label{henri-II}
e^{i \theta} \, \partial_y \psi_n(\cdot - a_n, \cdot) \to \partial_y \gu_{c_p} = 0 \text{ in } L^2(\R \times \T),
\end{equation}
by~\eqref{eq:limit-zero-y}. We next rely on the Pohozaev identity~\eqref{eq:Pohozaev} in order to obtain
\begin{align*}
E_{\lambda_n}(\psi_n) & = \int_{\R \times \T} |e^{i \theta} \, \partial_x \psi_n(x - a_n, y)|^2 \, dx \, dy\\
& = \lambda_n^2 \int_{\R \times \T} |\partial_y \psi_n|^2 + \frac{1}{2} \int_{\R \times \T} \big( 1 - |\psi_n(x - a_n, y)|^2 \big)^2 \, dx \, dy.
\end{align*}
In view of~\eqref{eq:cond-psi-n} and~\eqref{eq:borne-I}, we note that $E_{\lambda_n}(\psi_n) \to \gI(p)$ as $n \to \infty$. Combining with~\eqref{eq:limit-zero-y}, we are led to
\begin{equation}
\label{francois-I}
\int_{\R \times \T} |e^{i \theta} \, \partial_x \psi_n(x - a_n, y)|^2 \, dx \, dy \to \gI(p), \text{ and } \frac{1}{2} \int_{\R \times \T} \big( 1 - |\psi_n(x - a_n, y)|^2 \big)^2 \, dx \, dy \to \gI(p),
\end{equation}
as $n \to \infty$. We finally express the quantity $\gI(p)$ in terms of the travelling-wave profile $\gu_{c_p}$. Recall that this profile solves~\eqref{eq:gu-p} with $c = c_p$. We multiply this equation by the derivative $\gu_{c_p}'$ and integrate it taking into account the exponential decay of the functions $\gu_{c_p}'$ and $1 - |\gu_{c_p}|^2$. This gives
$$
\frac{1}{2} |\gu_{c_p}'|^2 = \frac{1}{4} \big( 1 - |\gu_{c_p}|^2 \big)^2.
$$
It is then enough to invoke Proposition~\ref{prop:exist-min1} in order to obtain
$$
\gI(p) = E(\gu_{c_p}) = \int_\R |\gu_{c_p}'|^2 = \frac{1}{2} \int_\R \big( 1 - |\gu_{c_p}|^2 \big)^2.
$$
In view of~\eqref{francois-I}, we deduce that
$$
\big\| e^{i \theta} \, \partial_x \psi_n(\cdot - a_n, \cdot) \big\|_{L^2(\R \times \T)} \to \big\| \gu_{c_p}' \big\|_{L^2(\R \times \T)},
$$
and
$$
\big\| 1 - |\psi_n(\cdot - a_n, \cdot)|^2 \big\|_{L^2(\R \times \T)} \to \big\| 1 - |\gu_{c_p}|^2 \big\|_{L^2(\R \times \T)},
$$
as $n \to \infty$. Combining with~\eqref{henri-II}, we conclude that the convergences in~\eqref{charles-VII} are actually strong.

In order to complete the proof of Proposition~\ref{prop:conv-min-p}, it only remains to establish that
\begin{equation}
\label{montauban}
\int_{\R \times \T} \eta_{c_p} \big| e^{i \theta} \, \psi_n(\cdot - a_n, \cdot) - \gu_{c_p} \big|^2 \to 0,
\end{equation}
as $n \to \infty$. Consider a positive number $R$ and write the decomposition
\begin{equation}
\label{marsan}
\int_{\R \times \T} \eta_{c_p} \big| e^{i \theta} \, \psi_n(\cdot - a_n, \cdot) - \gu_{c_p} \big|^2 = I_R + J_R,
\end{equation}
with
\begin{equation}
\label{rouen}
I_R := \int_{(- R, R) \times \T} \eta_{c_p} \big| e^{i \theta} \, \psi_n(\cdot - a_n, \cdot) - \gu_{c_p} \big|^2 \to 0,
\end{equation}
as $n \to \infty$ by~\eqref{louis-XII}, and
$$
J_R := \int_{(- R, R)^c \times \T} \eta_{c_p} \big| e^{i \theta} \, \psi_n(\cdot - a_n, \cdot) - \gu_{c_p} \big|^2.
$$
Concerning this integral, we have
\begin{equation}
\label{mont}
J_R \leq 2 \int_{(- R, R)^c \times \T} \eta_{c_p} \Big( 2 + \big| e^{i \theta} \,\psi_n(\cdot - a_n, \cdot) \big|^2 - 1 + \big| \gu_{c_p} \big|^2 - 1 \Big).
\end{equation}
Since $\eta_{c_p} \in L^2(\R \times \T)$, we infer from~\eqref{charles-VII} that
$$
\int_{(- R, R)^c \times \T} \eta_{c_p} \Big( 2 + \big| e^{i \theta} \,\psi_n(\cdot - a_n, \cdot) \big|^2 - 1 + \big| \gu_{c_p} \big|^2 - 1 \Big) \to 2 \int_{(- R, R)^c \times \T} \eta_{c_p} \big| \gu_{c_p} \big|^2.
$$
in the limit $n \to \infty$. The right-hand side of this limit can be made as small as necessary for $R$ large enough. Combining with~\eqref{marsan},~\eqref{rouen} and~\eqref{mont} is enough to complete the proof of~\eqref{montauban}. This concludes the proof of Proposition~\ref{prop:conv-min-p}. \qed

%%%%%%%%%%%%%%%%%%%%%%%%%%%%%%%%%%%%%%%%%%%%%%%
%%%%%%%%%%%%%%%%%%%%%%%%%%%%%%%%%%%%%%%%%%%%%%%
\subsection{Proof of Proposition~\ref{prop:loc-min} for $p \neq \frac{\pi}{2}$}
%%%%%%%%%%%%%%%%%%%%%%%%%%%%%%%%%%%%%%%%%%%%%%%
%%%%%%%%%%%%%%%%%%%%%%%%%%%%%%%%%%%%%%%%%%%%%%%

The proof of Proposition~\ref{prop:loc-min} is based on a coercivity estimate related to the orbital stability of the dark solitons in dimension one. The technical derivation of this estimate turns out to be different for the grey solitons on the one hand, and the black soliton on the other hand. This claim originates in the fact that we can use the hydrodynamical framework for handling the grey solitons, which is no more possible for the black soliton. This is the reason why we split the proof of Proposition~\ref{prop:loc-min} into two parts dealing first with the case of the grey solitons for $p \neq \pi/2$.

Given a positive number $\alpha$, consider a function $\psi$ in $\boV_p(\alpha)$. In view of Proposition~\ref{prop:X-1-2}, we can decompose this function as $\psi = \hat{\psi}_0 + w_0$, with $\hat{\psi}_0 \in X(\R)$ and $w_0 \in H^1(\R \times \T)$. Moreover, it follows from~\eqref{def:V-alpha} that $\| w_0 \|_{H^1} < \alpha$ and
$$
\inf_{(a, \theta) \in \R^2} d_{c_p} \big( e^{i \theta} \hat{\psi}_0(\cdot - a), \gu_{c_p} \big) < \alpha.
$$
We first use this control on the function $\hat{\psi}_0$ in order to estimate the difference between the energies $E_\lambda(\psi)$ and $E(\hat{\psi}_0)$. More precisely, we show the following inequality, which is still available for $p = \pi/2$.

\begin{lem}
\label{lem:reduc-E}
Let $p \in (- \pi/2, \pi/2]$, with $p \neq 0$. There exists a positive number $\alpha_p$ for which we can find a positive number $C_p$ such that we have
\begin{equation}
\label{eq:cont-E-1-2}
E_\lambda(\psi) \geq E(\hat{\psi}_0) + \frac{1}{2} \int_{\R \times \T} \Big( |\partial_x w_0|^2 + (\lambda^2 - C_p) |\partial_y w_0|^2 + |w_0|^2 \Big),
\end{equation}
for any function $\psi \in \boV_p(\alpha_p)$.
\end{lem}

\begin{proof}
The proof relies on the expansion of the energy $E_\lambda(\psi)$ in~\eqref{eq:dev-E-1-2}. Due to the identity
$$
\langle \hat{\psi}_0, w_0 \rangle_\C^2 + |w_0|^2 \langle \hat{\psi}_0, w_0
\rangle_\C + \frac{1}{4} |w_0|^4 = \Big( \langle \hat{\psi}_0, w_0
\rangle_\C + \frac{1}{2} |w_0|^2 \Big)^2,
$$
we indeed deduce from~\eqref{eq:dev-E-1-2} that
\begin{equation}
\label{djeco}
E_\lambda(\psi) - E(\hat{\psi}_0) \geq \frac{1}{2} \int_{\R \times \T} \big( |\partial_x w_0|^2 + \lambda^2 |\partial_y w_0|^2 \big) - \frac{1}{2} \int_{\R \times \T} |w_0|^2 (1 - |\hat{\psi}_0|^2).
\end{equation}
Invoking Lemma~\ref{lem:cont-modulus}, we can find a positive number $\alpha_p$ such that, when $\psi$ is in $\boV_p(\alpha_p)$, we get
$$
\big\| (1 - |e^{i \theta} \hat{\psi}_0(\cdot - a)|^2) - (1 - |\gu_{c_p}|^2) \big\|_{L^\infty} < 1, 
$$
for given numbers $(a, \theta) \in \R^2$. As a consequence, we obtain
$$
\frac{1}{2} \int_{\R \times \T} |w_0|^2 (1 - |\hat{\psi}_0|^2) \leq \frac{1}{2} \Big( 1 + \big\| 1 - |\gu_{c_p}|^2 \big\|_{L^\infty} \Big) \int_{\R \times \T} |w_0|^2,
$$
and we can invoke the Poincar\'e-Wirtinger inequality in order to find a positive number $C_p$ such that
$$
\frac{1}{2} \int_{\R \times \T} |w_0|^2 + \frac{1}{2} \int_{\R \times \T} |w_0|^2 (1 - |\hat{\psi}_0|^2) \leq \frac{C_p}{2} \int_{\R \times \T} |\partial_y w_0|^2.
$$
Combining with~\eqref{djeco}, we obtain~\eqref{eq:cont-E-1-2}. This completes the proof of Lemma~\ref{lem:reduc-E}.
\end{proof}

Our next goal is to provide a similar control for the momentum. When $p \neq \pi/2$, it follows from~\eqref{eq:gE-val} that the energy $E(\gu_{c_p})$ is strictly less than $2 \sqrt{2}/3$. Combining~\eqref{def:V-alpha} and the continuity of the Ginzburg-Landau energy $E$ on $X(\R)$ (see Appendix~\ref{sec:dimension-one}), we can decrease, if necessary, the value of the number $\alpha_p$ so that the energy $E(\hat{\psi}_0)$ is strictly less than $2 \sqrt{2}/3$ when $\psi \in \boV_p(\alpha_p)$. In view of Lemma~\ref{lem:min-E-0}, this guarantees that the function $\hat{\psi}_0$ lies in the non-vanishing set $NV\!X(\R)$ defined in~\eqref{def:NVX} below. As a consequence, the set $\boV_p(\alpha_p)$ is a subset of $Y(\R \times \T)$ and the momentum $P$ in statement $(ii)$ of Lemma~\ref{lem:def-P-2} is well-defined on this set. Moreover, we can show

\begin{lem}
\label{lem:reduc-P}
Let $p \in (- \pi/2, \pi/2)$, with $p \neq 0$. There exist a positive number $\alpha_p$ such that
\begin{equation}
\label{eq:cont-P-1-2}
\big| P(\psi) - P(\hat{\psi}_0) \big| \leq \frac{1}{4} \int_{\R \times \T} |\partial_x w_0|^2 + \frac{1}{2 \pi} \int_{\R \times \T} |\partial_y w_0|^2,
\end{equation}
for any function $\psi \in \boV_p(\alpha_p)$. Moreover, when $[P](\psi)
= p$ modulo $\pi$, the momentum $P(\psi)$ in this inequality is equal to
\begin{equation}
\label{eq:val-P-p}
P(\psi) = p.
\end{equation}
\end{lem}

\begin{proof}
The proof is based on the definition of the momentum $P(\psi)$ in~\eqref{def:P-2}, which gives
$$
\big| P(\psi) - P(\hat{\psi}_0) \big| \leq \frac{1}{2} \int_{\R \times \T} |\partial_x w_0| |w_0| \leq \frac{1}{4} \int_{\R \times \T} |\partial_x w_0|^2 + \int_{\R \times \T} |w_0|^2.
$$
Inequality~\eqref{eq:cont-P-1-2} then follows from the Poincar\'e inequality. Observe that the smallness of the number $\alpha_p$ is only used here in order that the quantity $P(\psi)$ and $P(\hat{\psi}_0)$ make sense. 

Concerning~\eqref{eq:val-P-p}, we recall that the energy $E(\hat{\psi}_0)$ is strictly less than $2 \sqrt{2}/3$ when $\alpha_p$ is small enough. Hence, it follows from Proposition~\ref{prop:boI-problem} that
\begin{equation}
\label{vilac}
\big| P(\hat{\psi}_0) \big| < \frac{\pi}{2}.
\end{equation}
Moreover, we know that $[P](\psi) = P(\psi)$ modulo $\pi$ on the one hand, and $[P](\psi) = p$ modulo $\pi$ on the other hand. As a consequence, there exists an integer $k \in \Z$ such that $P(\psi) = p + k \pi$. In view of~\eqref{eq:cont-P-1-2}, we are led to
$$
\big| p + k \pi - P(\hat{\psi}_0) \big| \leq \frac{1}{4} \| w_0 \|_{H^1}^2 < \frac{\alpha_p^2}{4}.
$$
Combining with~\eqref{vilac}, we can decrease the value of the number $\alpha_p$ if necessary so that $k = 0$ and
$P(\psi) = p$. This completes the proof of Lemma~\ref{lem:reduc-P}.
\end{proof}

Collecting~\eqref{eq:cont-E-1-2} and~\eqref{eq:cont-P-1-2}, we obtain
\begin{equation}
\label{eq:cont-E-P-1-2}
E_\lambda(\psi) - c_p P(\psi) \geq E(\hat{\psi}_0) - c_p P(\hat{\psi}_0) + \frac{1}{2} \int_{\R \times \T} \bigg( \Big( 1 - \frac{|c_p|}{2} \Big) |\partial_x w_0|^2 + \Big( \lambda^2 - C_p - \frac{|c_p|}{\pi} \Big) |\partial_y w_0|^2 + |w_0|^2 \bigg).
\end{equation}
Since $|c_p| < \sqrt{2}$, the last term in this inequality is non-negative for $\lambda > \sqrt{C_p + \frac{\sqrt{2}}{\pi}}$.
Under this condition, it vanishes if and only $w_0$ is identically equal to $0$.

Our goal is now to control from below the term $E(\hat{\psi}_0) - c_p
P(\hat{\psi}_0)$. Since the function $\hat{\psi}_0$ is in $NV\!X(\R)$, we can
rely on the hydrodynamical formulation $\hat{\psi}_0 = \rho_0 e^{i \theta_0}$
and analyze the quantities $E(\hat{\psi}_0)$ and $P(\hat{\psi}_0)$ in terms of
the variables $\eta_0 := 1 -\rho_0^2$ and $v_0 := \theta_0'$. In view
of~\eqref{eq:hydro-E} and~\eqref{def:P-NV}, the energy $E(\hat{\psi}_0)$ and the
momentum $P(\hat{\psi}_0)$ are then given by
$$
E \big( \hat{\psi}_0 \big) = E(\eta_0, v_0) := \frac{1}{8} \int_\R \frac{(\eta_0')^2}{1 - \eta_0} + \frac{1}{2} \int_\R (1 - \eta_0) v_0^2 + \frac{1}{4} \int_\R \eta_0^2,
$$
and
$$
P \big( \hat{\psi}_0 \big) = P(\eta_0, v_0) := \frac{1}{2} \int_\R \eta_0 v_0.
$$
Recall also that the pair $(\eta_0, v_0)$ belongs to the non-vanishing set
$NV(\R)$ defined in~\eqref{def:NV}.

Similarly, we can lift the profile $\gu_{c_p}$ as $\gu_{c_p} = \rho_{c_p} e^{i \theta_{c_p}}$ and introduce the corresponding variables $\eta_{c_p} := 1 - \rho_{c_p}^2$ and $v_{c_p} := \theta_{c_p}'$, which are also in $NV(\R)$. With this notation at hand, we can consider the neighbourhoods of the pair $(\eta_{c_p}, v_{c_p})$ given by 
\begin{equation}
\label{def:U-beta}
\boU_p(\beta) := \Big\{ (\eta, v) \in NV(\R) \text{ s.t. } \inf_{a \in \R} \big(\| \eta_0(\cdot - a) - \eta_{c_p} \|_{H^1}^2 + \| v_0(\cdot - a) - v_{c_p} \|_{L^2}^2 \big)< \beta^2 \Big\},
\end{equation}
for any positive number $\beta$. We first show that the pair $(\eta_0, v_0)$ lies in one of these neighbourhoods when $\hat{\psi}_0$ is in $\boV_p(\alpha_p)$. More precisely, we show

\begin{lem}
\label{lem:U-V}
Given any positive number $\beta$, there exists a positive number $\alpha \leq \alpha_p$ such that
$$
(\eta_0, v_0) \in \boU_p(\beta),
$$
for any function $\psi \in \boV_p(\alpha)$.
\end{lem}

\begin{proof}
Consider a positive number $\alpha$ such that $\alpha \leq \alpha_p$. Under this condition, the function $\hat{\psi}_0$ is in $NV\!X(\R)$ when $\psi$ belongs to $\boV_p(\alpha)$. In particular, the functions $\eta_0$ and $v_0$ are well-defined. In view of~\eqref{def:V-alpha}, we can also find numbers $(a, \theta) \in \R^2$ such that $d_{c_p}(e^{i \theta} \hat{\psi}_0(\cdot - a),\gu_{c_p}) < \alpha$, so that by~\eqref{def:d-c},
\begin{equation}
\label{janod}
\big\| \eta_0(\cdot - a) - \eta_{c_p} \big\|_{L^2} \leq d_{c_p} \big( e^{i \theta} \hat{\psi}_0(\cdot - a), \gu_{c_p} \big) < 2 \alpha.
\end{equation}
We next write
$$
\eta_0'(\cdot - a) - \eta_{c_p}' = -2 \langle \hat{\psi}_0'(\cdot - a) -
\gu_{c_p}', \hat{\psi}_0(\cdot - a) \rangle_\C - 2 \langle \gu_{c_p}', \hat{\psi}_0(\cdot - a) - \gu_{c_p} \rangle_\C.
$$
Invoking Lemma~\ref{lem:cont-modulus}, we can decrease the value of the number $\alpha$ if necessary, so that
\begin{equation}
\label{hasbro}
\big\| \hat{\psi}_0 \big\|_{L^\infty}^2 \leq 1 + \big\| \gu_{c_p} \big\|_{L^\infty}^2.
\end{equation}
Moreover, it follows from~\eqref{eq:gu-p} and~\eqref{def:eta-c} that
\begin{equation}
\label{fisher}
\big| \gu_{c_p}' \big|^2 = \frac{\eta_{c_p}^2}{2} \leq \frac{2 - c_p^2}{4} \eta_{c_p}.
\end{equation}
Hence, we are led to
$$
\big\| \eta'_0(\cdot - a) - \eta_{c_p}' \big\|_{L^2} \leq 2 \big( 1 + \big\| \gu_{c_p} \big\|_{L^\infty}^2 \big)^\frac{1}{2} \big\| \hat{\psi}_0'(\cdot - a) - \gu_{c_p}' \big\|_{L^2} + \big( 2 - c_p^2 \big)^\frac{1}{2} \Big\| \eta_{c_p}^\frac{1}{2} \big( \hat{\psi}_0(\cdot - a) - \gu_{c_p} \big) \Big\|_{L^2},
$$
and there exists a positive number $C_p$, depending only on $p$, such that
\begin{equation}
\label{price}
\big\| \eta'_0(\cdot - a) - \eta_{c_p}' \big\|_{L^2} \leq C_p \, d_{c_p} \big( e^{i \theta} \hat{\psi}_0(\cdot - a), \gu_{c_p} \big) < 2 C_p \alpha.
\end{equation}
Similarly, we write
\begin{align*}
v_0(\cdot - a) - v_{c_p} = \frac{1}{|\psi_0(\cdot - a)|^2} \bigg( & \big\langle i (\gu_{c_p}' - \hat{\psi}_0(\cdot - a)'), \hat{\psi}_0(\cdot - a) \big\rangle_\C + \big\langle i \gu_{c_p}', \gu_{c_p} - \hat{\psi}_0(\cdot - a) \big\rangle_\C \\
& + \frac{\big\langle i \gu_{c_p}', \gu_{c_p} \big\rangle_\C}{|\gu_{c_p}|^2} \big( \eta_{c_p} - \eta_0(\cdot - a) \big) \bigg).
\end{align*}
Invoking again Lemma~\ref{lem:cont-modulus} and using~\eqref{eq:gu-val}, we can decrease the value of the number $\alpha$ if necessary, so that
$$
\inf_{x \in \R} |\hat{\psi}_0(x - a)|^2 \geq \inf_{x \in \R}|\gu_{c_p}(x)|^2 - \frac{c^2}{4} = \frac{c^2}{4}.
$$
Combining with~\eqref{hasbro} and~\eqref{fisher}, we deduce that
$$
\big\| v_0(\cdot - a) - v_{c_p} \big\|_{L^2} \leq C_p \, d_{c_p} \big( e^{i \theta} \hat{\psi}_0(\cdot - a), \gu_{c_p} \big) < 2 C_p \alpha,
$$
for a further positive number $C_p$. In view of~\eqref{janod} and~\eqref{price}, we conclude that
$$
\big\| \eta_0(\cdot - a) - \eta_{c_p} \big\|_{H^1}^2 + \big\| v_0(\cdot - a) - v_{c_p} \big\|_{L^2}^2 < (4 + 8 C_p^2) \alpha^2.
$$
It is then enough to fix the choice of $\alpha \leq \beta/(4 + 8 C_p)^{1/2}$ in order to complete the proof of Lemma~\ref{lem:U-V}.
\end{proof}

The sets $\boU_p(\beta)$ were already introduced in~\cite{BetGrSm1} in order to prove the orbital stability of chains of $N$ solitons. All the results in~\cite{BetGrSm1} are stated for an arbitrary integer $N \geq 1$, and in particular, hold for a single soliton. We now explicit the results in~\cite{BetGrSm1} on which we rely for completing the proof of Proposition~\ref{prop:loc-min}.

We begin by~\cite[Proposition 2]{BetGrSm1}, which provides a decomposition of each pair $(\eta_0, v_0)$ in $\boU_p(\beta)$ as the sum of a modulated soliton plus a remainder term satisfying suitable orthogonality conditions. More precisely, we can rephrase this proposition as

\begin{lem}[\cite{BetGrSm1}]
\label{lem:modul-param}
There exist two positive numbers $\beta_1$ and $C_1$, depending only on $c_p$,
and two functions $\ga \in \boC^1(\boU_p(\beta_1), \R)$ and $\gc \in
\boC^1(\boU_p(\beta_1), (- \sqrt{2}, 0) \cup (0, \sqrt{2}))$ such that, for any pair $(\eta_0, v_0) \in \boU_p(\beta_1)$, the function
\begin{equation}
\label{def:vareps}
\varepsilon := (\varepsilon_\eta, \varepsilon_v) := \big( \eta_0(\cdot - a) - \eta_c, v_0(\cdot - a) - v_c \big),
\end{equation}
with $a := \ga(\eta_0, v_0)$ and $c := \gc(\eta_0, v_0)$, satisfies the orthogonality conditions
\begin{equation}
\label{eq:cond-ortho}
\big\langle (\varepsilon_\eta, \varepsilon_v), (\eta_c', v_c') \big\rangle_{L^2 \times L^2} = dP(\eta_c,v_c)(\varepsilon_\eta, \varepsilon_v) = 0.
\end{equation}
Moreover, if there exist numbers $a_* \in \R$ and $\beta \leq \beta_1$ such that
$$
\| (\eta_0(\cdot - a_*), v_0(\cdot - a_*)) - (\eta_{c_p}, v_{c_p}) \|_{H^1 \times L^2} < \beta,
$$
then
\begin{equation}
\label{eq:cont-modul}
\big\| \varepsilon \big\|_{H^1 \times L^2} + \big| c - c_p \big| + \big| a - a_* \big| \leq C_1 \beta. 
\end{equation}
\end{lem}

We use the decomposition in Lemma~\ref{lem:modul-param} to expand the quantities $E(\eta_0, v_0)$ and $P(\eta_0, v_0)$ at second order. Using the invariance by translation of the energy $E(\eta_0, v_0)$, we first obtain
\begin{equation}
\label{eq:dev-E}
E \big( \eta_0, v_0 \big) = E \big( (\eta_c, v_c) + \varepsilon \big) = E \big( \eta_c, v_c \big) + dE \big( \eta_c, v_c \big)(\varepsilon) + \frac{1}{2} d^2E \big( \eta_c, v_c \big)(\varepsilon, \varepsilon) + R_c(\varepsilon),
\end{equation}
with $c = \gc(\eta_0, v_0)$. In this identity, we have set
$$
dE \big( \eta_c, v_c \big)(\varepsilon) := \frac{1}{2} \int_\R \Big( \frac{(\eta_c')^2 \varepsilon_\eta}{4 (1 - \eta_c)^2} + \frac{\eta_c' \varepsilon_\eta'}{2 (1 - \eta_c)} - v_c^2 \varepsilon_\eta + 2 (1 - \eta_c) v_c \varepsilon_v + \eta_c \varepsilon_\eta \Big),
$$
$$
d^2 E \big( \eta_c, v_c \big)(\varepsilon, \varepsilon) := \int_\R \Big( \frac{(\varepsilon_\eta')^2}{4 (1 - \eta_c)} + \frac{\eta_c' \varepsilon_\eta \varepsilon_\eta'}{2 (1 - \eta_c)^2} + \frac{(\eta_c')^2 \varepsilon_\eta^2}{4 (1 - \eta_c)^3} - 2 v_c \varepsilon_\eta \varepsilon_v + (1 - \eta_c) \varepsilon_v^2 + \frac{1}{2} \varepsilon_\eta^2 \Big),
$$
and
$$
R_c(\varepsilon) := \frac{1}{2} \int_\R \Big( \frac{(\varepsilon_\eta')^2 \varepsilon_\eta}{4 (1 - \eta_c) (1 - \eta_c -\varepsilon_\eta)} + \frac{\eta_c' \varepsilon_\eta^2 \varepsilon_\eta'}{2 (1 - \eta_c)^2 (1 - \eta_c -\varepsilon_\eta)} + \frac{(\eta_c')^2 \varepsilon_\eta^3}{4 (1 - \eta_c)^3 (1 - \eta_c -\varepsilon_\eta)} - \varepsilon_\eta \varepsilon_v^2 \Big).
$$
Similarly, the invariance by translation of the momentum $P(\eta_0, v_0)$ provides
\begin{equation}
\label{eq:dev-P}
P \big( \eta_0, v_0 \big) = P \big( (\eta_c, v_c) + \varepsilon \big) = P \big( \eta_c, v_c \big) + dP \big( \eta_c, v_c \big)(\varepsilon) + \frac{1}{2} d^2 P \big( \eta_c, v_c \big)(\varepsilon, \varepsilon),
\end{equation}
with
\begin{equation}
\label{eq:dev-P-1-2}
dP \big( \eta_c, v_c \big)(\varepsilon) := \frac{1}{2} \int_\R \big( \eta_c \varepsilon_v + v_c \varepsilon_\eta \big), \quad \text{and} \quad d^2 P \big( \eta_c, v_c \big)(\varepsilon, \varepsilon) := \int_\R \varepsilon_\eta \varepsilon_v.
\end{equation}
The previous identities give an expansion at second order of the quantity
$E(\eta_0, v_0) - c_p P(\eta_0, v_0).$ We now estimate each term in this expansion in order to bound from below this quantity.

\begin{lem}
\label{lem:coer-E-P}
Consider a function $(\eta_0, v_0) \in \boU_p(\beta_1)$, where $\beta_1$ is the positive number in Lemma~\ref{lem:modul-param}, and set $\varepsilon = \big( \eta_0(\cdot - a) - \eta_c, v_0(\cdot - a) - v_c \big)$, with $a = \ga(\eta_0, v_0)$ and $c = \gc(\eta_0, v_0)$. There exist two positive numbers $\beta_2 \leq \beta_1$ and $K_2$, depending only on $c_p$, such that
\begin{equation}
\label{eq:estim-0}
E \big( \eta_c, v_c \big) - c_p P \big( \eta_c, v_c \big) \geq E \big( \gu_{c_p} \big) - c_p P \big( \gu_{c_p} \big) - K_2 \big|c - c_p|^2,
\end{equation}
\begin{equation}
\label{eq:ident-1}
dE \big( \eta_c, v_c \big)(\varepsilon) - c_p \, dP \big( \eta_c, v_c \big)(\varepsilon) = 0,
\end{equation}
\begin{equation}
\label{eq:estim-2}
d^2 E \big( \eta_c, v_c \big)(\varepsilon, \varepsilon) - c_p \, d^2 P \big( \eta_c, v_c \big)(\varepsilon, \varepsilon) \geq K_2 \Big( \big\| \varepsilon \big\|_{H^1 \times L^2}^2 - \big| c - c_p \big|^2 \Big),
\end{equation}
and
\begin{equation}
\label{eq:estim-R}
R_c(\varepsilon) \geq - K_2 \big\| \varepsilon \big\|_{H^1 \times L^2}^3,
\end{equation}
when $(\eta_0, v_0) \in \boU_p(\beta_2)$.
\end{lem}

\begin{proof}
Concerning~\eqref{eq:estim-0}, recall that the modulated speed $c$ lies in $(- \sqrt{2}, 0) \cup (0, \sqrt{2})$ by Lemma~\ref{lem:modul-param}. Hence, it follows from~\cite[Proposition 1]{BetGrSa2} that the energy $E(\eta_c, v_c)$ and the momentum $P(\eta_c, v_c)$ are given by
\begin{equation}
\label{eq:form-E-P-c}
E(\eta_c, v_c) = \frac{1}{3} \big( 2 - c^2 \big)^\frac{3}{2}, \quad \text{ and } \quad P(\eta_c, v_c) = \text{sign}(c) \Big( \frac{\pi}{2} - \arctan \bigg( \frac{|c|}{\sqrt{2 - c^2}} \bigg) - \frac{|c|}{2} \sqrt{2 - c^2} \Big).
\end{equation}
In view of~\eqref{eq:cont-modul}, we can decrease if necessary the value of the
number $\beta_1$ such that all the modulated speeds $c$ corresponding to
pairs in $\boU_p(\beta_1)$ are in a compact subset of the interval,
either $(- \pi/2 , 0)$, or $(0, \pi/2)$, containing the speed $c_p$. In
this case, we can use the smoothness of the maps $c \mapsto E(\eta_c,
v_c)$ and $c \mapsto P(\eta_c, v_c)$ on both these intervals in order to find a positive number $K$, depending only on $c_p$, such that
\begin{align*}
E \big( \eta_c, v_c \big) - c_p P \big( \eta_c, v_c \big) - & \Big( E \big( \eta_{c_p}, v_{c_p} \big) - c_p P \big( \eta_{c_p}, v_{c_p} \big) \Big) \\
\geq & \frac{d}{dc} \Big( E \big( \eta_c, v_c \big) \Big)_{|c = c_p} - c_p \frac{d}{dc} \Big( P \big( \eta_c, v_c \big) \Big)_{|c = c_p} - K \big( c - c_p \big)^2.
\end{align*}
Since $E(\eta_{c_p}, v_{c_p}) - c_p P(\eta_{c_p}, v_{c_p}) = E(\gu_{c_p}) - c_p P(\gu_{c_p})$ by definition, the estimate in~\eqref{eq:estim-0} follows from the property that
$$
\frac{d}{dc} \Big( E \big( \eta_c, v_c \big) \Big)_{|c = c_p} = - c_p \big( 2 - c_p^2 \big)^\frac{1}{2} = c_p \frac{d}{dc} \Big( P \big( \eta_c, v_c \big) \Big)_{|c = c_p},
$$
which results from~\eqref{eq:deriv-Chi}.

For the proof of~\eqref{eq:ident-1}, we first use the second orthogonality condition in~\eqref{eq:cond-ortho} in order to write
$$
dE \big( \eta_c, v_c \big)(\varepsilon) - c_p \, dP \big( \eta_c, v_c \big)(\varepsilon) = dE \big( \eta_c, v_c \big)(\varepsilon) = dE \big( \eta_c, v_c \big)(\varepsilon) - c \, dP \big( \eta_c, v_c \big)(\varepsilon).
$$
We next rephrase the equation satisfied by the profile $\gu_c$ in terms of the hydrodynamic pair $(\eta_c, v_c)$. In view of~\eqref{eq:gu-p}, we are led to the system
$$
\begin{cases} \frac{\eta_{c}''}{2 (1 - \eta_{c})} + \frac{(\eta_{c}')^2}{4 (1 - \eta_{c})^2} + c v_{c} + v_{c}^2 - \eta_{c} = 0, \\[5pt]
(1 - \eta_{c}) v_{c} = \frac{c}{2} \eta_{c}. \end{cases}
$$
It is then enough to multiply the first equation in this system by $\varepsilon_\eta$, the second one by $\varepsilon_v$, and to integrate by parts in order to obtain
$$
dE \big( \eta_c, v_c \big)(\varepsilon) - c \, dP \big( \eta_c, v_c \big)(\varepsilon) = 0,
$$
and therefore,~\eqref{eq:ident-1}.

We now turn to~\eqref{eq:estim-2}. We rewrite the second order term as
\begin{equation}
\label{corolle}
\begin{split}
d^2 E \big( \eta_c, v_c \big)(\varepsilon, \varepsilon) - c_p \, d^2 P \big( \eta_c, v_c \big)(\varepsilon, \varepsilon) = & d^2 E \big( \eta_c, v_c \big)(\varepsilon, \varepsilon) - c \, d^2 P \big( \eta_c, v_c \big)(\varepsilon, \varepsilon) \\
& + \big( c - c_p \big) d^2 P \big( \eta_c, v_c \big)(\varepsilon, \varepsilon).
\end{split}
\end{equation}
In view of~\eqref{eq:dev-P-1-2}, we have
\begin{equation}
\label{clementoni}
\big( c - c_p \big) d^2 P \big( \eta_c, v_c \big)(\varepsilon, \varepsilon) \geq - \big| c- c_p \big| \big\| \varepsilon \|_{H^1 \times L^2}^2 \geq - \frac{1}{2 \delta} \big( c- c_p \big)^2 - \frac{\delta}{2} \big\| \varepsilon \|_{H^1 \times L^2}^4,
\end{equation}
for any positive number $\delta$. Recall that the function $\varepsilon$
satisfies the two orthogonal conditions in~\eqref{eq:cond-ortho},
whereas by~\eqref{eq:cont-modul}, the modulated speeds $c$ lie in a
compact subset of the interval $(- \pi/2 , 0)$ or $(0, \pi/2)$,
containing the speed $c_p$. As a consequence, we can apply~\cite[Proposition 1]{BetGrSm1} in order to find a positive number $K$, depending only on $c_p$, such that
$$
d^2 E \big( \eta_c, v_c \big)(\varepsilon, \varepsilon) - c \, d^2 P \big( \eta_c, v_c \big)(\varepsilon, \varepsilon) \geq K \big\| \varepsilon \|_{H^1 \times L^2}^2.
$$
Combining with~\eqref{corolle} and~\eqref{clementoni}, we obtain
$$
d^2 E \big( \eta_c, v_c \big)(\varepsilon, \varepsilon) - c_p \, d^2 P \big( \eta_c, v_c \big)(\varepsilon, \varepsilon) \geq K \big\| \varepsilon \|_{H^1 \times L^2}^2 - \frac{1}{2 \delta} \big( c- c_p \big)^2 - \frac{\delta}{2} \big\| \varepsilon \|_{H^1 \times L^2}^4.
$$
At this stage, we can decrease if necessary the value of the number $\beta_2$ so that $\| \varepsilon \|_{H^1 \times L^2} \leq 1$ by~\eqref{eq:cont-modul}. It is then enough to choose $\delta = K/2$ in order to obtain~\eqref{eq:estim-2}.

Finally, the estimate in~\eqref{eq:estim-R} essentially results from the Sobolev embedding theorem. In view of~\eqref{def:eta-c}, there indeed exists a positive number $\kappa \leq 1$, depending only on $c_p$, such that
$$
1 - \eta_c \geq \kappa,
$$
for any modulated speed $c$ in a compact subset of either $(- \pi/2 , 0)$, or
$(0, \pi/2)$, containing $c_p$. Decreasing if necessary the value of the number $\beta_2$, we deduce from~\eqref{eq:cont-modul} and the Sobolev embedding theorem that
$$
1 - \eta_c - \varepsilon_\eta \geq \frac{\kappa}{2}.
$$
In view of~\eqref{def:eta-c}, the derivative $\eta_c'$ is also uniformly bounded by a positive number depending only on $c_p$. Using once again the Sobolev embedding theorem, we are led to
$$
R_c(\varepsilon) \geq - \frac{K}{\kappa^4} \big\| \varepsilon \big\|_{H^1 \times L^2}^3,
$$
where, as before, $K$ only depends on $c_p$. This completes the proof of~\eqref{eq:estim-R}, as well as of Lemma~\ref{lem:coer-E-P}. 
\end{proof}

We are now in position to conclude the proof of Proposition~\ref{prop:loc-min} when $p \neq \pi/2$.

\begin{proof}[End of the proof of Proposition~\ref{prop:loc-min}]
Going back to~\eqref{eq:dev-E} and~\eqref{eq:dev-P} and invoking Lemma~\ref{lem:coer-E-P}, we can write
\begin{equation}
\label{eq:bound-E-P-c}
\begin{split}
E \big( \hat{\psi}_0 \big) - c_p P \big( \hat{\psi}_0 \big) = & E \big( \eta_0, v_0 \big) - c_p P \big( \eta_0, v_0 \big) \\
\geq & E \big( \gu_{c_p} \big) - c_p P \big( \gu_{c_p} \big) + K_2 \Big( \big\| \varepsilon \big\|_{H^1 \times L^2}^2 - \big\| \varepsilon \big\|_{H^1 \times L^2}^3 - 2 \big|c - c_p|^2 \Big).
\end{split}
\end{equation}
In order to estimate the difference $c - c_p$, we rely on the formula in~\eqref{eq:form-E-P-c} for the momentum $P(\eta_c, v_c)$. Since the modulated speed $c$ lies in a compact subset containing $c_p$ by~\eqref{eq:cont-modul}, we infer from~\eqref{eq:deriv-Chi} and~\eqref{eq:form-E-P-c} the existence of a positive number $K$, depending only on $c_p$, such that
\begin{equation}
\label{mahut}
\big| c - c_p \big| \leq K \big| P(\eta_c, v_c) - P(\eta_{c_p}, v_{c_p}) \big|.
\end{equation}
Combining~\eqref{eq:dev-P} with~\eqref{eq:cond-ortho} and~\eqref{eq:dev-P-1-2}, we check that
$$
\big| P(\eta_c, v_c) - P(\eta_0, v_0) \big| \leq \frac{1}{2} \big\| \varepsilon \big\|_{H^1 \times L^2}^2.
$$
On the other hand, it follows from~\eqref{eq:cont-P-1-2} and~\eqref{eq:val-P-p} that
$$
\big| P(\eta_{c_p}, v_{c_p}) - P(\eta_0, v_0) \big| = \big| p - P(\hat{\psi}_0) \big| \leq \frac{1}{4} \big\| \nabla w_0 \big\|_{L^2}^2.
$$
Hence, we obtain
$$
\big| c - c_p \big| \leq \frac{K}{2} \Big( \big\| \varepsilon \big\|_{H^1 \times L^2}^2 + \big\| \nabla w_0 \big\|_{L^2}^2 \Big).
$$
Introducing this inequality into~\eqref{eq:bound-E-P-c}, we are led to
$$
E \big( \hat{\psi}_0 \big) - c_p P \big( \hat{\psi}_0 \big) \geq E \big( \gu_{c_p} \big) - c_p P \big( \gu_{c_p} \big) + K_2 \big\| \varepsilon \big\|_{H^1 \times L^2}^2 - K_2 \big\| \varepsilon \big\|_{H^1 \times L^2}^3 - K^2 \big\| \varepsilon \big\|_{H^1 \times L^2}^4 - K^2 \big\| \nabla w_0 \big\|_{L^2}^4.
$$
At this stage, we can again decrease the value of the number $\beta_2$ so that~\eqref{eq:cont-modul} provides the inequality
$$
K_2 \big\| \varepsilon \big\|_{H^1 \times L^2}^2 - K_2 \big\| \varepsilon \big\|_{H^1 \times L^2}^3 - K^2 \big\| \varepsilon \big\|_{H^1 \times L^2}^4 \geq \frac{K_2}{2} \big\| \varepsilon \big\|_{H^1 \times L^2}^2.
$$
As a consequence, we obtain
$$
E \big( \hat{\psi}_0 \big) - c_p P \big( \hat{\psi}_0 \big) \geq E \big( \gu_{c_p} \big) - c_p P \big( \gu_{c_p} \big) + \frac{K_2}{2} \big\| \varepsilon \big\|_{H^1 \times L^2}^2 - K^2 \big\| \nabla w_0 \big\|_{L^2}^4.
$$
We next invoke Lemma~\ref{lem:U-V} in order to find a number $\alpha$ such that $(\eta_0, v_0) \in \boU_p(\beta_2)$ when $\psi \in \boV_p(\alpha)$. In this case, we derive from~\eqref{def:V-alpha} and~\eqref{eq:cont-E-P-1-2} that
\begin{align*}
E_\lambda \big( \psi \big) - c_p P \big( \psi \big) \geq & E \big( \gu_{c_p} \big) - c_p P \big( \gu_{c_p} \big) + \frac{K_2}{2} \big\| \varepsilon \big\|_{H^1 \times L^2}^2 + \frac{1}{2} \Big( 1 - \frac{|c_p|}{2} - 2 K^2 \alpha^2 \Big) \big\| \partial_x w_0 \big\|_{L^2}^2 \\
& + \frac{1}{2} \Big( \lambda^2 - C_p - \frac{|c_p|}{\pi} - 2 K^2 \alpha^2 \Big) \big\| \partial_y w_0 \big\|_{L^2}^2 + \frac{1}{2} \big\| w_0 \big\|_{L^2}^2.
\end{align*}
We finally fix the choice of the number $\alpha_p$ so that $1 - \sqrt{2}/2 - 2 K^2 \alpha_p^2 > 0$, and the choice of the number $\lambda_p$ so that $\lambda_p^2 - C_p - \sqrt{2}/\pi - 2 K^2 \alpha_p^2 > 0$. The previous choices guarantee that
$$
E_\lambda \big( \psi \big) - c_p P \big( \psi \big) \geq E \big( \gu_{c_p} \big) - c_p P \big( \gu_{c_p} \big),
$$
when $\psi \in \boV_p(\alpha_p)$ and $\lambda \geq \lambda_p$. This inequality is exactly~\eqref{eq:local-mini-p} due to the facts that $P(\psi) = P(\gu_{c_p}) = p$ and $E(\gu_{c_p}) = E_\lambda(\gu_{c_p})$. Moreover, equality holds if and only if
$$
\big\| \varepsilon \big\|_{H^1 \times L^2} = \big\| w_0 \big\|_{H^1} = 0.
$$
In this case, we observe that $(\eta_0, v_0) = (\eta_c(\cdot + a), v_c(\cdot + a))$, so that there exists a number $\theta \in \R$ for which $\hat{\psi}_0 = e^{- i \theta} \gu_c(\cdot + a)$. As a consequence, we have
$$
\psi = \hat{\psi}_0 + w_0 = e^{- i \theta} \gu_c(\cdot + a) + 0 = e^{- i \theta} \gu_c(\cdot + a).
$$
Since $p = P(\psi)$ by Lemma~\ref{lem:reduc-P}, we deduce that $P(\gu_c) = p$, and we conclude that $c = c_p$. This completes the proof of Proposition~\ref{prop:loc-min} for $p \neq \pi/2$.
\end{proof}

%%%%%%%%%%%%%%%%%%%%%%%%%%%%%%%%%%%%%%%%%%%%%%
%%%%%%%%%%%%%%%%%%%%%%%%%%%%%%%%%%%%%%%%%%%%%%
\subsection{Proof of Proposition~\ref{prop:loc-min} for $p = \frac{\pi}{2}$}
%%%%%%%%%%%%%%%%%%%%%%%%%%%%%%%%%%%%%%%%%%%%%%
%%%%%%%%%%%%%%%%%%%%%%%%%%%%%%%%%%%%%%%%%%%%%%

For $p = \pi/2$, Proposition~\ref{prop:loc-min} also relies on a coercivity estimate, but for the black soliton $\gu_0$. This estimate was derived in~\cite[Proposition 1]{GravSme1} for revisiting the orbital stability of $\gu_0$. We can rephrase it as

\begin{lem}[\cite{GravSme1}]
\label{lem:coer-0}
For $\psi = \gu_0 + \varepsilon \in X(\R)$, set $\eta_\varepsilon := - 2 \langle \gu_0, \varepsilon \rangle_\C - |\varepsilon|^2$. There exists a universal positive number $\Lambda_0$ such that
\begin{equation}
\label{eq:coer-E0}
E(\psi) - E(\gu_0) \geq \Lambda_0 \, \big( \| \varepsilon \|_{H_0}^2 + \| \eta_\varepsilon \|_{L^2}^2 \big) - \frac{1}{\Lambda_0} \, \| \varepsilon \|_{H_0}^3,
\end{equation}
as soon as
\begin{equation}
\label{cond:ortho-0}
\int_\R \langle \varepsilon, \gu_0' \rangle_\C = \int_\R \langle \varepsilon, i \, \gu_0' \rangle_\C = \int_\R \langle \varepsilon, i \, \gu_0 \rangle_\C \big( 1 - |\gu_0|^2 \big) = 0.
\end{equation}
\end{lem}

The orthogonality conditions in~\eqref{cond:ortho-0} are necessary to control one negative and two null directions of the energy $E$ in the neighbourhood of the black soliton $\gu_0$. As in Lemma~\ref{lem:modul-param}, they can be imposed by introducing suitable modulation parameters related to the speed of the solitons and their invariance by translation and phase shift. These properties were already invoked for constructing modulation parameters in~\cite[Proposition 2]{GravSme1}.
Setting
$$
\gU_0(\beta) := \Big\{ \psi \in X(\R) \text{ s.t. } \inf_{(a, \theta) \in \R^2} d_0 \big( e^{i \theta} \psi(\cdot - a), \gu_0 \big) < \beta \Big\},
$$
for any positive number $\beta$, we can summarize this construction as follows.

\begin{lem}[\cite{GravSme1}]
\label{lem:modul-0}
There exist two positive numbers $\beta_0$ and $A_0$, and three continuously differentiable functions $\ga \in \boC^1(\gU_0(\beta_0), \R)$, $\vartheta \in \boC^1(\gU_0(\beta_0), \R / 2 \pi \Z)$ and $\gc \in \boC^1(\gU_0(\beta_0), (- \sqrt{2}, \sqrt{2}))$ such that for any $\psi \in \gU_0(\beta_0)$, the function
$$
\varepsilon := e^{i \theta} \psi(\cdot - a) - \gu_c,
$$
with $a = \ga(\psi)$, $\theta = \vartheta(\psi)$ and $c = \gc(\psi)$, satisfies the orthogonality conditions
\begin{equation}
\label{cond:ortho-0-c}
\int_\R \langle \varepsilon, \gu_c' \rangle_\C = \int_\R \langle \varepsilon, i \, \gu_c' \rangle_\C = \int_\R \langle \varepsilon, i \, \textup{Re}(\gu_c) \rangle_\C \big( 1 - |\gu_c|^2 \big) = 0.
\end{equation}
Moreover, if
$$
\big\| e^{i \theta_*} \, \psi \big( \cdot - a_* \big) - \gu_0 \big\|_{H_0} \leq \beta,
$$
for numbers $a_* \in \R$, $\theta_* \in \R$ and $\beta \leq \beta_0$, then,
\begin{equation}
\label{eq:est-modul-0}
\big\| \varepsilon \big\|_{H_0} + \big| c \big| + \big| a - a_* \big| + \big| e^{i \theta} - e^{i \theta_*} \big| \leq A_0 \beta.
\end{equation}
\end{lem}

\begin{rem}
Here, the smoothness of the maps $\ga$, $\vartheta$ and $\gc$ must be understood
with respect to the differential structure provided by the vector space $H(\R)$.
\end{rem}

The orthogonality conditions in~\eqref{cond:ortho-0-c} differ from the ones in~\eqref{cond:ortho-0}. However, a coercivity estimate similar to~\eqref{eq:coer-E0} remains available under these latest conditions. Corollary 1 in~\cite{GravSme1} indeed guarantees that

\begin{lem}[\cite{GravSme1}]
\label{lem:coer-0c}
For $|c| < \sqrt{2}$ and $\psi = \gu_c + \varepsilon \in X(\R)$, set $\eta_\varepsilon := - 2 \langle \gu_c, \varepsilon \rangle_\C - |\varepsilon|^2$. Given any number $0 < \sigma < \sqrt{2}$, there exists a positive number $\Lambda_\sigma$, depending only on $\sigma$, such that
\begin{equation}
\label{eq:coer-Ec}
E(\psi) - E(\gu_0) \geq \Lambda_\sigma \, \big( \| \varepsilon \|_{H_0}^2 + \| \eta_\varepsilon \|_{L^2}^2 \big) - \frac{1}{\Lambda_\sigma} \, \big( c^2 + \| \varepsilon \|_{H_0}^3 \big),
\end{equation}
as soon as $|c| \leq \sigma$, and $\varepsilon$ satisfies the orthogonality conditions in~\eqref{cond:ortho-0-c}.
\end{lem}

At this stage, consider a function $\psi \in \boV_{\pi/2}(\alpha)$ for a number
$0 < \alpha < \beta_0$. By definition, the function $\hat{\psi}_0$ lies in the
subset $\gU_0(\alpha)$ of $\gU_0(\beta_0)$. Applying Lemma~\ref{lem:modul-0}, we
can find numbers $a_0 \in \R$, $\theta_0 \in \R$ and $c_0 \in (- \sqrt{2},
\sqrt{2})$ such that the function $\varepsilon_0 := e^{i \theta_0}
\hat{\psi}_0(\cdot - a_0) - \gu_{c_0}$ satisfies the orthogonality conditions
in~\eqref{cond:ortho-0-c}. Combining~\eqref{eq:est-modul-0}
and~\eqref{eq:coer-Ec}, and decreasing if necessary the value of the number
$\alpha$, we find a positive number $\Lambda_\alpha$, depending only on $\alpha$, such that
$$
E \big( \hat{\psi}_0 \big) - E \big( \gu_0 \big) \geq \Lambda_\alpha \, \big( \| \varepsilon_0 \|_{H_0}^2 + \| \eta_{\varepsilon_0} \|_{L^2}^2 \big) - \frac{c_0^2}{\Lambda_\alpha},
$$
with $\eta_{\varepsilon_0} := - 2 \langle \gu_{c_0}, \varepsilon_0 \rangle_\C - |\varepsilon_0|^2$ as before. Assuming that $\alpha \leq \alpha_{\pi/2}$, where the number $\alpha_{\pi/2}$ is given by Lemma~\ref{lem:reduc-E}, we infer from this lemma that
\begin{equation}
\label{champion}
E_\lambda \big( \psi \big) \geq E_\lambda \big( \gu_0 \big) + \frac{1}{2} \int_{\R \times \T} \Big( |\partial_x w_0|^2 + (\lambda^2 - C_p) |\partial_y w_0|^2 + |w_0|^2 \Big) + \Lambda_\alpha \, \big( \| \varepsilon_0 \|_{H_0}^2 + \| \eta_{\varepsilon_0} \|_{L^2}^2 \big) - \frac{c_0^2}{\Lambda_\alpha}.
\end{equation}
As a consequence, we are essentially reduced to control the modulated speed $c_0$ with respect to the various norms of the functions $w_0$, $\varepsilon_0$ and $\eta_{\varepsilon_0}$. As in the previous case $p \neq \pi/2$, we derive this control from the property that $[P](\psi) = \pi/2$ modulo $\pi$. In this direction, our main tool is the following consequence of Propositions 4 and 5 in~\cite{GravSme1}. 

\begin{lem}[\cite{GravSme1}]
\label{lem:dev-P-0}
There exist two positive numbers $\beta_1 < \beta_0$ and $A_1$ such that any function $\psi \in \gU_0(\beta_1)$ satisfies
\begin{equation}
\label{eq:dev-[P]-0}
[P](\psi) = [P](\gu_c) - \int_\R \langle i \gu_c', \varepsilon \rangle_\C + R_c(\varepsilon) \quad \text{ modulo } \pi,
\end{equation}
with
\begin{equation}
\label{eq:est-Rc}
\big| R_c(\varepsilon) \big| \leq A_1 \Big( \| \varepsilon \|_{H_0}^2 + \| \eta_\varepsilon \|_{L^2}^2 \Big).
\end{equation}
In the previous formulae, we have set, as before, $\varepsilon = e^{i \theta} \psi(\cdot - a) - \gu_c$, with $a = \ga(\psi)$, $\theta = \vartheta(\psi)$ and $c = \gc(\psi)$, as well as $\eta_\varepsilon := - 2 \langle \gu_c, \varepsilon \rangle_\C - |\varepsilon|^2$.
\end{lem}

With Lemma~\ref{lem:dev-P-0} at hand, we are in position to conclude the proof of Proposition~\ref{prop:loc-min} for $p = \pi/2$.

\begin{proof}[End of the proof of Proposition~\ref{prop:loc-min} for $p = \pi/2$]
Decreasing if necessary the value of $\alpha$, we can apply Lemma~\ref{lem:dev-P-0} to the function $\hat{\psi}_0$. In view of the second orthogonality condition in~\eqref{cond:ortho-0-c}, this provides the identity
$$
[P] \big( \hat{\psi}_0 \big) = [P] \big( \gu_{c_0} \big) + R_{c_0} \big( \varepsilon_0 \big) \quad \text{ modulo } \pi,
$$
with $R_{c_0}(\varepsilon_0)$ satisfying~\eqref{eq:est-Rc} for $\varepsilon = \varepsilon_0$ and $\eta_\varepsilon = \eta_{\varepsilon_0}$. Going to~\eqref{def:P-2}, we deduce that
\begin{equation}
\label{bruguera}
[P] \big( \psi \big) - [P] \big( \gu_{c_0} \big) = \frac{1}{2} \int_{\R \times \T} \langle i \partial_x w_0, w_0 \rangle_\C + R_{c_0} \big( \varepsilon_0 \big) \quad \text{ modulo } \pi.
\end{equation}
Recall now that $[P](\psi) = \pi/2 = \Xi(0)$ modulo $\pi$, while the value modulo $\pi$ of $[P](\gu_{c_0})$ is equal to $\text{sign}(c_0) \, \Xi(|c_0|)$ by~\cite[Proposition 1]{BetGrSa2}. Here, $\Xi$ refers to the function in~\eqref{def:Xi}. Moreover, for $\alpha$ small enough, the right-hand side of~\eqref{bruguera} is small by~\eqref{eq:est-Rc}, so as the modulated speed $c_0$ by~\eqref{eq:est-modul-0}. As a consequence, we derive from the identity modulo $\pi$ in~\eqref{bruguera} that
$$
\Big| \Xi(0) - \Xi(|c_0|) \Big| = \bigg| \frac{1}{2} \int_{\R \times \T} \langle i \partial_x w_0, w_0 \rangle_\C + R_{c_0} \big( \varepsilon_0 \big) \bigg|.
$$
Using~\eqref{eq:deriv-Chi}, we can argue as for~\eqref{mahut} in order to derive from~\eqref{eq:est-modul-0} and~\eqref{eq:est-Rc} the existence of a positive number $A_\alpha$, depending only on $\alpha$, such that
$$
|c_0| = \big| 0 - |c_0| \big| \leq A_\alpha \Big( \| \varepsilon_0 \|_{H_0}^2 + \| \eta_{\varepsilon_0} \|_{L^2}^2 + \delta \| w_0 \|_{L^2}^2 + \frac{1}{\delta} \| \partial_x w_0 \|_{L^2}^2 \Big).
$$
for any positive number $\delta$. It then remains to introduce this inequality
into~\eqref{champion} and to choose the number $\delta$ large enough in order to 
deduce from the Poincaré inequality that
$$
E_\lambda \big( \psi \big) \geq E_\lambda \big( \gu_0 \big) + \Lambda \Big( \| \varepsilon_0 \|_{H_0}^2 + \| \eta_{\varepsilon_0} \|_{L^2}^2 + \| w_0 \|_{H^1}^2 \Big) \geq E_\lambda \big( \gu_0 \big),
$$
for $\alpha$ small enough, $\lambda$ large enough, and a further positive number $\Lambda$, depending only on $\alpha$ and $\lambda$. This concludes the proof of~\eqref{eq:local-mini-p}.

Moreover, this inequality is an equality if and only if $\varepsilon_0 = 0$ and $w_0 = 0$, that is if and only if $\psi = e^{- i \theta_0} \gu_{c_0}(\cdot + a_0)$. In view of~\eqref{def:gc-p}, the only possibility for the untwisted momentum $[P](\psi)$ to be equal to $\pi/2$ modulo $\pi$ is that $c_0 = 0$. In conclusion, equality can only hold if $\psi = e^{- i \theta_0} \gu_0(\cdot + a_0)$. This completes the proof of Proposition~\ref{prop:loc-min} for $p = \pi/2$.
\end{proof}

%%%%%%%%%%%%%%%%%%%%%%%%%%%%%%%%%
%%%%%%%%%%%%%%%%%%%%%%%%%%%%%%%%%
\subsection{Proof of Lemma~\ref{lem:prop-I-lambda}}
%%%%%%%%%%%%%%%%%%%%%%%%%%%%%%%%%
%%%%%%%%%%%%%%%%%%%%%%%%%%%%%%%%%

Consider a function $\psi \in X(\R \times \T)$ such that $[P](\psi) = p$ modulo $\pi$. Given two positive numbers $\lambda_1$ and $\lambda_2$, with $\lambda_1 < \lambda_2$, we have
$$
E_{\lambda_1}(\psi) \leq E_{\lambda_2}(\psi) \leq \Big( \frac{\lambda_2}{\lambda_1} \Big)^2 E_{\lambda_1}(\psi).
$$
In view of~\eqref{def:I-problem}, we obtain
$$
\boI_{\lambda_1}(p) \leq \boI_{\lambda_2}(p) \leq \Big( \frac{\lambda_2}{\lambda_1} \Big)^2 \boI_{\lambda_1}(p),
$$
which is enough to guarantee that the map $\lambda \mapsto \boI_\lambda(p)$ is non-decreasing and continuous on $\R_+^*$.

Concerning the proof of~\eqref{eq:lim-I-0}, we rely on the scaling
\begin{equation}
\label{eq:scale-L}
\psi_L(x, y) = \psi(x, \lambda y),
\end{equation}
which transforms a function $\psi \in X(\R \times \T)$ in a function $\psi_L \in X(\R \times \T_L)$. Here, we have set $L = 1/\lambda$. The notation $\T_L$ refers to the torus of size $L$ and the energy set $X(\R \times \T_L)$ is defined according to~\eqref{def:X-space}, with $\T$ replaced by $\T_L$. In the limit $\lambda \to 0$, the length $L$ tends to $+ \infty$ and the minimization problem $\boI_\lambda(p)$ can be related to the problem of minimizing the Ginzburg-Landau energy in the whole plane $\R^2$ for a fixed large momentum.

Indeed, we can compute
\begin{equation}
\label{monfils}
E \big( \psi_L \big) := \frac{1}{2} \int_{\R \times \T_L} |\nabla \psi_L|^2 + \frac{1}{4} \int_{\R \times \T_L} \big( 1 - |\psi_L|^2 \big)^2 = L E_\lambda(\psi).
\end{equation}
Going to Lemma~\ref{lem:def-P-2}, we also check that the definition of the untwisted momentum on the set $X(\R \times \T)$ extends literally to the set $X(\R \times \T_L)$, up to the fact that this quantity is now valued into $\R/\pi L \Z$. Moreover, we can derive from Lemma~\ref{lem:def-P-2} that the untwisted momentum $[P]_L(\psi_L)$ is equal to
\begin{equation}
\label{medvedev}
[P]_L(\psi_L) = L [P](\psi) \quad \text{ modulo } \pi L.
\end{equation}
As a consequence, we obtain
$$
\boI_\lambda(p) = \frac{1}{L} \inf \Big\{ E(\psi_L) : \psi_L \in X(\R \times \T_L) \text{ s.t. } [P]_L(\psi_L) = p L \text{ modulo } \pi L \Big\}.
$$
At least formally, the previous infimum is related to the limit $q \to + \infty$ of the minimal value of the Ginzburg-Landau energy in $\R^2$ with fixed momentum equal to $q$. This latter minimization problem was solved in~\cite{BetGrSa1}. It follows from~\cite{BethSau1} that the limit $q \to + \infty$ of this problem is divergent as $2 \pi \ln(q)$. This asymptotics is based on the property that the corresponding minimizer is a pair of vortices in uniform translation. We are now going to use this special configuration as a test function in order to show~\eqref{eq:lim-I-0}.

In order to clarify the construction, we now identify the space $\R^2$ to the complex plane $\C$ by setting $z = x + i y$ in the sequel. We introduce the complex-valued function $\xi$ defined on the disc $D(0, 2) := \{ z \in \C \text{ s.t. } |z| < 2 \}$ by
\begin{equation}
\label{verdasco}
\xi(z) = \frac{\overline{z - i}}{|z - i|} \, \frac{z + i}{|z + i|} \, e^{i \varphi(z)}.
\end{equation}
In this expression, $\varphi$ refers to a real-valued harmonic function on $D(0, 2)$ such that $\xi = 1$ on the circle $\partial D(0, 2)$. We can check that the value of $\varphi$ can be fixed so that
\begin{equation}
\label{isner}
\varphi(z) = \arctan \Big( \frac{2 \textup{Re}(z)}{1 - |z|^2} \Big),
\end{equation}
for any $z \in \partial D(0, 2)$. Observe that $\varphi$ is even with respect to the variable $\textup{Im}(z)$. Observe also that $f$ has exactly two vortices with opposite degrees at the points $\pm i$. Given a number $R \geq 1$, we next introduce the rescaled and regularized version $\xi_R$ of $\xi$ given by
\begin{equation}
\label{robredo}
\xi_R(z) = \begin{cases} 1 \text{ if } |z| \geq 2 R, \\
|z \pm i R| \, \xi \big( \frac{z}{R} \big) \text{ if } |z \pm i R| < 1, \\
\xi \big( \frac{z}{R} \big) \text{ otherwise}. \end{cases}
\end{equation}
The function $\xi_R$ is well-defined and continuous on $\R^2$. Given a number $L \geq 4 R$, we can consider its restriction to the set $\{ z \in \C : |\textup{Im}(z)| \leq L/2 \}$ and extend it as a $L$-periodic function with respect to the variable $y$. Denote by $\psi_L$ the corresponding extension and define a function $\psi : \R \times \T \to \C$ according to the scaling in~\eqref{eq:scale-L}.

The extension $\psi_L$ belongs to $H_\text{loc}^1(\R \times \T_L)$, where this
set is defined as in~\eqref{def:H-loc-1}, with $1$-periodic functions replaced by $L$-periodic functions. It is even with respect to the variable $y$ and identically equal to $1$ outside the disc $D(0, 2 R)$. We now estimate the value of its energy $E(\psi_L)$. A direct computation first provides
\begin{equation}
\label{roddick}
\frac{1}{4} \int_{\R \times \T_L} \big( 1 - |\psi_L|^2 \big)^2 = \frac{1}{2} \int_{|z - i R| < 1} \big( 1 - |z - i R|^2 \big)^2 \, dx \, dy = \frac{\pi}{6}.
\end{equation}
Concerning the gradient $\nabla \psi_L$, we next check that
\begin{align*}
|\nabla \psi_L(z)|^2 = & 1 + \frac{4 R^2}{|z + i R|^2} + \frac{1}{R^2} \Big| \nabla \varphi \Big( \frac{z}{R}\Big) \Big|^2 |z - i R|^2 - \frac{2}{R} \partial_x \varphi \Big( \frac{z}{R} \Big) \Big( R - y + (y +R) \frac{|z - i R|^2}{|z + i R|^2} \Big) \\
& - \frac{2}{R} \partial_y \varphi \Big( \frac{z}{R} \Big) \Big( x - x \frac{|z - i R|^2}{|z + i R|^2} \Big),
\end{align*}
for any $|z - i R| < 1$. Using the inequality $2 a b \leq a^2 + b^2$ and the fact that $|z + i R| \geq R \geq 1$ for $y \geq 0$, we can bound this quantity by
$$
|\nabla \psi_L(z)|^2 \leq 13 + \frac{2}{R^2} \Big| \nabla \varphi \Big( \frac{x}{R} \Big) \Big|^2,
$$
when $|z - i R|< 1$. Hence, we obtain
\begin{equation}
\label{gasquet}
\frac{1}{2} \int_{|z - i R| < 1} |\nabla \psi_L|^2 \leq \frac{13 \pi}{2} + \int_{|z - i| < 1/R} |\nabla \varphi|^2.
\end{equation}
By symmetry with respect to the axis $x$, the same inequality is true replacing $|z - i R| < 1$ by $|z + i R| < 1$ in the left-hand side, and $|z - i| < 1/R$ by $|z + i| < 1/R$ in the right-hand side. Similarly, we compute
\begin{align*}
|\nabla \psi_L(z)|^2 & = \frac{1}{|z - i R|^2} + \frac{1}{|z + i R|^2} + 2 \frac{R^2 - |z|^2}{|z - i R|^2 |z + i R|^2} + \frac{1}{R^2} \Big| \nabla \varphi \Big( \frac{z}{R} \Big) \Big|^2\\
& + \frac{2}{R} \partial_x \varphi \Big( \frac{z}{R} \Big) \Big( \frac{y - R}{|z - i R|^2} - \frac{y + R}{|z + i R|^2} \Big) + \frac{2}{R} \partial_y \varphi \Big( \frac{z}{R} \Big) \Big( \frac{x}{|z + i R|^2} - \frac{x}{|z - i R|^2} \Big),
\end{align*}
for any $z \in \omega_R := \{ z \in D(0, 2 R) \text{ s.t. } |z - i R| > 1 \text{ and } |z + i R| > 1 \big\}$. As a consequence, we can write 
\begin{equation}
\label{tsonga}
\frac{1}{2} \int_{\omega_R} |\nabla \psi_L(z)|^2 \leq I_1 +\frac{1}{2} \int_{\omega_1} |\nabla \varphi|^2 + I_2,
\end{equation}
with $\omega_1 := \{ z \in D(0, 2) \text{ s.t. } |z - i| > 1/R \text{ and } |z + i| > 1/R \big\}$. In this inequality, we have set
$$
I_1 := \frac{1}{2} \int_{\omega_R} \Big( \frac{1}{|z - i R|^2} + \frac{1}{|z + i R|^2} + 2 \frac{R^2 - |z|^2}{|z - i R|^2 |z + i R|^2} \Big) \, dx \, dy,
$$
and
$$
I_2:= \frac{1}{R} \int_{\omega_R} \bigg( \partial_x \varphi \Big( \frac{z}{R} \Big) \Big( \frac{y - R}{|z - i R|^2} - \frac{y + R}{|z + i R|^2} \Big) + \partial_y \varphi \Big( \frac{x}{R} \Big) \Big( \frac{x}{|z + i R|^2} - \frac{x}{|z - i R|^2} \Big) \bigg) \, dx \, dy.
$$
We first estimate the integral $I_1$ using the fact that its integrand is symmetric with respect to the variable $y$. Setting $\omega_R^+ := \{ z \in \omega_R \text{ s.t. } y \geq 0 \}$, we combine the inequality $|z + i R| \geq R \geq 1$ for $y \geq 0$ and the identity $R^2 - |z|^2 = 2 R (R - y) - |z - i R|^2$ in order to get
\begin{align*}
I_1 = & \int_{\omega_R^+} \Big( \frac{1}{|z - i R|^2} - \frac{1}{|z + i R|^2} + \frac{4 R (R - y)}{|z - i R|^2 |z + i R|^2} \Big) \, dx \, dy
\\ \leq & \int_{\omega_R^+} \Big( \frac{1}{|z - i R|^2} + \frac{4}{R |z - i R|} \Big) \, dx \, dy.
\end{align*}
When $z \in \omega_R^+$ and $|z - i R| \geq R$, we have
$$
\frac{1}{|z - i R|^2} + \frac{4}{R |z - i R|} \leq \frac{5}{R^2},
$$
so that
\begin{equation}
\label{melzer}
I_1 \leq 5 \pi + \int_{D(0, R) \setminus D(0, 1)} \Big( \frac{1}{|z|^2} + \frac{4}{R |z|} \Big) \, dx \, dy \leq 2 \pi \ln(R) + 13 \pi.
\end{equation}
We next integrate by parts the integral $I_2$ in order to obtain
$$
I_2 = \int_{\partial \omega_R} \bigg( \nu_x(z) \Big( \frac{y - R}{|z - i R|^2} - \frac{y + R}{|z + i R|^2} \Big) + \nu_y(z) \Big( \frac{x}{|z + i R|^2} - \frac{x}{|z - i R|^2} \Big) \bigg) \varphi \Big( \frac{z}{R} \Big) \, d\gamma(z),
$$
where $\nu(z) = (\nu_x(z), \nu_y(z))$ is the outward unit normal vector to $\partial \omega_R$ and $d\gamma$ is the infinitesimal length element of the curve $\partial \omega_R$. Recall at this stage that the function $\varphi$ is harmonic on the disc $D(0, 2)$. In view of~\eqref{isner}, it follows from the maximum principle that
\begin{equation}
\label{nadal}
\| \varphi \|_{L^\infty(D(0, 2))} \leq \frac{\pi}{2},
\end{equation}
so that
$$
I_2 \leq \frac{\pi}{2} \bigg( \int_{\partial D(0, 2 R)} \frac{4}{R} \, d\gamma(z) + 2 \int_{\partial D(0, 1)} 4 \, d\gamma(z) \bigg) \leq 16 \pi^2.
$$
Combining with~\eqref{gasquet},~\eqref{tsonga} and~\eqref{melzer}, we finally get
\begin{equation}
\label{murray}
\frac{1}{2} \int_{\R \times \T_L} |\nabla \psi_L|^2 \leq 2 \pi \ln(R) + 16 \pi^2 + 26 \pi + \int_{D(0, 2)} |\nabla \varphi|^2.
\end{equation}
In view of~\eqref{roddick}, we deduce the existence of a universal positive constant $C$ such that
\begin{equation}
\label{djokovic}
E(\psi_L) \leq 2 \pi \ln(R) + C.
\end{equation}
Note in particular that the function $\psi_L$ lies in $X(\R \times \T_L)$, so that we are allowed to define its untwisted momentum $[P]_L(\psi_L)$ according to Lemma~\ref{lem:def-P-2}.

In order to compute this quantity, we first rely on~\eqref{robredo} from which
we derive that the function $[\hat{\psi}_L]_0$ is identically equal to $1$ for
$|x| \geq R$. As a consequence, the function $\theta_0 = 0$ is one of its phase functions on the intervals $I_R^\pm$. In view of~\eqref{def:P-1} and~\eqref{def:P-2}, we obtain
$$
P_{\theta_0}(\psi_L) = \frac{1}{2} \int_\R \langle i [\hat{\psi}_L]_0', [\hat{\psi}_L]_0 \rangle_\C + \frac{1}{2} \int_{\R \times \T_L} \langle i \partial_x w_0, w_0 \rangle_\C,
$$
with $w_0 = \psi_L - [\hat{\psi}_L]_0$ as before. Due to the orthogonality of the functions $[\hat{\psi}_L]_0$ and $w_0$, and the compactly supported nature of their derivatives, the previous formula can be simplified as
$$
P_{\theta_0}(\psi_L) = \frac{1}{2} \int_{\R \times \T_L} \langle i \partial_x \psi_L, \psi_L \rangle_\C.
$$
Going back to~\eqref{robredo}, we derive from the local integrability of the map $z \mapsto y/|z|^2$ that
\begin{equation}
\label{delpotro}
P_{\theta_0}(\psi_L) = J_R + J_+ + J_-,
\end{equation}
where we have set
$$
J_R := \frac{1}{2} \int_{D(0, 2 R)} \bigg( - \frac{y - R}{|z - i R|^2} + \frac{y + R}{|z + i R|^2} - \frac{1}{R} \partial_x \varphi \Big( \frac{z}{R} \Big) \bigg) \, dx \, dy,
$$
and
\begin{equation}
\label{andujar}
J_\pm := \frac{1}{2} \int_{D(\pm i R, 1)} \big( 1 - |z \mp i R|^2 \big) \bigg( \frac{y - R}{|z - i R|^2} - \frac{y + R}{|z + i R|^2} + \frac{1}{R} \partial_x \varphi \Big( \frac{z}{R} \Big) \bigg) \, dx \, dy.
\end{equation}
Integrating by parts, we check that
$$
\frac{1}{2 R} \int_{D(\pm i R, 1)} \big( 1 - |z \mp i R|^2 \big) \, \partial_x \varphi \Big( \frac{z}{R} \Big) \, dx \, dy = \int_{D(\pm i R, 1)} x \, \varphi \Big( \frac{z}{R} \Big) \, dx \, dy,
$$
so that by~\eqref{nadal}, we obtain
\begin{equation}
\label{fish}
\big| J_\pm \big| \leq \frac{1}{2} \int_{D(0, 1)} \frac{dx \, dy}{|z|} + \frac{\pi}{2} + \frac{\pi^2}{2} \leq \frac{3 \pi}{2} + \frac{\pi^2}{2}.
\end{equation}
On the other hand, a direct scaling provides
\begin{equation}
\label{ljubicic}
J_R = R J_1 := \frac{R}{2} \int_{D(0, 2)} \Big( - \frac{y - 1}{|z - i|^2} + \frac{y + 1}{|z + i|^2} - \partial_x \varphi(z) \Big) \, dx \, dy.
\end{equation}
Applying the Fubini theorem, we can write the integral $J_1$ as
$$
J_1 = \frac{1}{2} \int_{- 2}^2 j_1(y) \, dy,
$$
with
$$
j_1(y) = \int_{- \sqrt{4 - y^2}}^{\sqrt{4 - y^2}} \Big( - \frac{y - 1}{x^2 + (y - 1)^2} + \frac{y + 1}{x^2 + (y + 1)^2} - \partial_x \varphi(x, y) \Big) \, dx,
$$
for $y \neq \pm 1$. In view of~\eqref{isner}, the integrals $j_1(y)$ are equal to
$$
j_1(y) = - 2 \arctan \Big( \frac{\sqrt{4 - y^2}}{y - 1} \Big) + 2 \arctan \Big( \frac{\sqrt{4 - y^2}}{y + 1} \Big) + 2 \arctan \Big( \frac{2 \sqrt{4 - y^2}}{3} \Big).
$$
At this stage, we can check that
$$
j_1'(y) = 0,
$$
when $y \neq \pm 1$, so that
$$
j_1(y) = \begin{cases} \lim_{y \to 2} j_1(y) = 0 & \text{for } 1 < y \leq 2, \\
 \lim_{y \to 1^-} j_1(y) = 2 \pi & \text{for } - 1 < y < 1, \\
\lim_{y \to - 2} j_1(y) = 0 & \text{for } - 2 \leq y < - 1. \end{cases}
$$
By the Fubini theorem, the integral $J_1$ is then equal to $J_1 = 2 \pi$, so that $J_R = 2 \pi R$. In view of~\eqref{delpotro} and~\eqref{fish}, we obtain
\begin{equation}
\label{hewitt}
\big| P_{\theta_0}(\psi_L) - 2 \pi R \big| \leq 3 \pi + \pi^2.
\end{equation}

On the other hand, we can derive from~\eqref{delpotro},~\eqref{andujar} and~\eqref{ljubicic} that the map $R \mapsto P_{\theta_0}(\psi_L)$ is continuous on $[1, 4 L]$. In view of~\eqref{hewitt}, the range of this function covers the interval $[5 \pi + \pi^2, \pi L/2 - 3 \pi - \pi^2]$. In particular, given a fixed number in $(0, \pi/2)$, we can find, for $L$ large enough, a positive number $R_L$ such that $[P]_L(\psi_L) = P_{\theta_0}(\psi_L) = p L$ modulo $L \pi$, and
$$
\Big| R_L - \frac{p L}{2 \pi} \Big| \leq \frac{3 + \pi}{2}.
$$
In this case, we deduce from~\eqref{murray} that
$$
E(\psi_L) \leq 2 \pi \ln(L) + 2 \pi \ln(p) + C,
$$
where $C$ is a further universal constant. As a consequence, the function $\psi$ corresponding to $\psi_L$ by the scaling in~\eqref{eq:scale-L} lies in $X(\R \times \T)$, with $[P](\psi) = p$ modulo $\pi$ by~\eqref{medvedev}. Using~\eqref{monfils}, we are led to
$$
\boI_\lambda(p) \leq E_\lambda(\psi) \leq \frac{1}{L} \Big( 2 \pi \ln(L) + 2 \pi \ln(p) + C \Big),
$$
so that $\boI_\lambda(p)$ tends to $0$ when $\lambda = 1/L \to 0$. Since the minimal value $\boI_\lambda$ is an even function of $p$ by~\eqref{eq:sym-I}, the same property holds for $p \in (- \pi/2, 0)$. This completes the proof of Lemma~\ref{lem:prop-I-lambda} when $p \neq \pi/2$.

For $p = \pi/2$, it follows from the non-negativity and the Lipschitz continuity of the function $\boI_\lambda$ that
$$
0 \leq \boI_\lambda \Big( \frac{\pi}{2} \Big) \leq \boI_\lambda(p) + \sqrt{2} \Big( \frac{\pi}{2} - p \Big),
$$
for any $0 < p < \pi/2$. In the limit $\lambda \to 0$, this gives
$$
0 \leq \liminf_{\lambda \to 0} \boI_\lambda \Big( \frac{\pi}{2} \Big) \leq \limsup_{\lambda \to 0} \boI_\lambda \Big( \frac{\pi}{2} \Big) \leq \sqrt{2} \Big( \frac{\pi}{2} - p \Big).
$$
Letting $p \to \pi/2$, we conclude that the quantity $\boI_\lambda(\pi/2)$ also tends to $0$ as $\lambda \to 0$. This completes the proof of Lemma~\ref{lem:prop-I-lambda}. \qed

%%%%%%%%%%%%%%%%%%%%%%%%%%%%%%%%%
%%%%%%%%%%%%%%%%%%%%%%%%%%%%%%%%%
%%%%%%%%%%%%%%%%%%%%%%%%%%%%%%%%%
\appendix
\numberwithin{thm}{section}
\numberwithin{cor}{section}
\numberwithin{equation}{section}
\section{Energy set and momentum in dimension one}
\label{sec:dimension-one}
%%%%%%%%%%%%%%%%%%%%%%%%%%%%%%%%%
%%%%%%%%%%%%%%%%%%%%%%%%%%%%%%%%%
%%%%%%%%%%%%%%%%%%%%%%%%%%%%%%%%%

In this section, we collect useful results concerning the energy set $X(\R)$ and the momentum $P$ in dimension one. In particular, we recall several statements established in~\cite{BetGrSa2, Gerard2, BetGrSm2, GravSme1}.

In dimension one, the energy set is defined as
$$
X(\R) = \big\{ \psi \in H_\text{loc}^1(\R) : \psi' \in L^2(\R) \text{ and } 1 - |\psi|^2 \in L^2(\R) \big\}.
$$
As a consequence of the Sobolev embedding theorem, a function $\psi$ in this set is actually $1/2$-H\"older continuous on $\R$. Moreover, this function is bounded (see~\cite{Gerard2}), so that the energy set is a subset of the Zhidkov space
$$
Z^1(\R) := \big\{ \psi \in \boC_b^0(\R) : \psi' \in L^2(\R) \big\}.
$$
This property guarantees that the function $\eta := 1 - |\psi|^2$ belongs to the Sobolev space $H^1(\R)$, so that it owns a vanishing limit at $\pm \infty$. In particular, we can find a positive number $R$ such that $\rho(x) := |\psi(x)| \geq 1/2$ for $|x| \geq R$. We can therefore lift the function $\psi$ as $\psi = \rho e^{i \theta}$ on both the intervals $I_R^- = (- \infty, - R]$ and $I_R^+ = [R, + \infty)$. The phase function $\theta$ is continuous on these intervals, with a derivative $\theta'$ in $L^2(I_R^\pm)$. Note that this phase function is defined up to two factors in $2 \pi \Z$, one on each interval $I_R^\pm$. 

This double indeterminacy is removed when the function $\psi$ does not vanish on
the whole line, that is belongs to the non vanishing energy set
\begin{equation}
\label{def:NVX}
NV\!X(\R) := \big\{ \psi \in X(\R) \text{ s.t. } \inf_{x \in \R} |\psi(x)| > 0 \big\}.
\end{equation}
In this case, the phase function $\theta$ is defined up to only one phase factor in $2 \pi \Z$. Moreover, the energy $E(\psi)$ is given by the hydrodynamical expression
\begin{equation}
\label{eq:hydro-E}
E(\psi) = \frac{1}{8} \int_\R \frac{(\eta')^2}{1 - \eta} + \frac{1}{2} \int_\R (1 - \eta) v^2 + \frac{1}{4} \int_\R \eta^2,
\end{equation}
in which we have set $v := \theta'$. In particular, there is a natural
correspondence between the fact that the function $\psi$ is in $NV\!X(\R)$ and the property that the pair $(\eta, v)$ lies in 
\begin{equation}
\label{def:NV}
NV(\R) := \Big\{ (\eta, v) \in H^1(\R) \times L^2(\R) \text{ s.t. } \inf_{x \in \R} \eta(x) < 1 \Big\}. 
\end{equation}

Concerning the definition of the momentum $P$, it is formally given by the integral
$$
\frac{1}{2} \int_\R \langle i \partial_x \psi, \psi \rangle_\C.
$$
Due to a possible lack of integrability at infinity, this quantity is not necessarily well-defined when $\psi \in X(\R)$. In order to give it a rigorous meaning, we assume first that the function $\psi$ can be lifted as $\psi = \rho e^{i \theta}$ and write the hydrodynamical expression
\begin{equation}
\label{eq:density-P}
\frac{1}{2} \langle i \partial_x \psi, \psi \rangle_\C = - \frac{1}{2} \rho^2 \theta' = \frac{1}{2} \eta \theta' - \frac{1}{2} \theta'.
\end{equation}
When the pair $(\eta, \theta')$ lies in $NV(\R)$, the function $\eta \theta'$ is
integrable on $\R$, but in general, the derivative $\theta'$ is not. We refer to~\cite{BetGrSa2, BeGrSaS1} for a discussion about several ways to by-pass this difficulty. A convenient way to define the momentum, in the sense that the quantity defined in this way will satisfy the natural properties of the momentum, is simply to drop the term containing the derivative $\theta'$ and to set
\begin{equation}
\label{def:P-NV}
P(\psi) = \frac{1}{2} \int_\R \eta \theta'.
\end{equation}
Once the decision is made to choose this definition, it is necessary to extend
it to functions which can vanish. A natural way to perform this extension is to rely on the property that the functions $\psi \in X(\R)$ can be lifted at least on intervals of the form $I_R^\pm$ for $R$ large enough. Hence, we can expect that the previous formula for the momentum will be available on these intervals.

In order to check this claim, we introduce a smooth cut-off function $\chi : \R \to [0, 1]$ with $\chi(x) = 0$ for $|x| \leq 1$ and $\chi(x) = 1$ for $|x| \geq 2$, and we set $\chi_r(x) = \chi(x/r)$ for any positive number $r$. When the function $\psi$ does not vanish on $\R$, the expression for its momentum in~\eqref{def:P-NV} can be rephrased as
\begin{equation}
\label{def:P-1}
P(\psi) = P_\theta(\psi) := \frac{1}{2} \int_\R \big( \langle i \psi', \psi \rangle_\C + (\chi_r \, \theta)' \big),
\end{equation}
in view of~\eqref{eq:density-P}. This identity is true for any choice of the positive number $r$. Given an arbitrary function $\psi \in X(\R)$, we can fix this choice so that the right-hand side $P_\theta(\psi)$ of the previous formula makes sense. Note however that this quantity possibly depends on the choice of the phase function $\theta$. This leads to the following definition of the momentum.

\begin{lem}
\label{lem:def-P-1}
Given a function $\psi \in X(\R)$, consider a positive number $R$ such that $|\psi(x)| \geq 1/2$ for $|x| \geq R$ and a phase function $\theta \in \boC^0(I_R^\pm)$ such that $\psi = |\psi| e^{i \theta}$ on $I_R^\pm$. Choose a smooth cut-off function $\chi : \R \to [0, 1]$ such that $\chi(x) = 0$ for $|x| \leq 1$ and $\chi(x) = 1$ for $|x| \geq 2$, and set $\chi_r(x) = \chi(x/r)$ for a number $r > R$.

$(i)$ The quantity $P_\theta(\psi)$ given by formula~\eqref{def:P-1} is well-defined and does not depend on the choice of neither the function $\chi$, nor the number $r$.

$(ii)$ When the function $\psi$ is in $NV\!X(\R)$, the momentum $P_\theta(\psi)$ does not depend on the choice of the phase function $\theta$.

$(iii)$ Given an arbitrary function $\psi \in X(\R)$, the value modulo $\pi$ of the quantity $P_\theta(\psi)$ does not depend on the choice of the phase function $\theta$, and it is possible to fix this choice such that $P_\theta(\psi) \in (- \pi/2, \pi/2]$. In particular, the untwisted momentum $[P] : X(\R) \to \R / \pi \Z$ defined by $[P](\psi) = P_\theta(\psi)$ modulo $\pi$ is well-defined.
\end{lem}

In the sequel, we drop the dependence on the phase function $\theta$ of the momentum $P_\theta(\psi)$ when the function $\psi$ is in $NV\!X(\R)$. This quantity is only defined on $NV\!X(\R)$. Since it is the only one to be defined without ambiguity, this is also the only one which we will call momentum.

\begin{proof}
The fact that the quantity $P_\theta(\psi)$ is well-defined follows from the property that $\psi$ belongs to $H_\text{loc}^1(\R)$ and from the identity
\begin{equation}
\label{castres}
\langle i \psi', \psi \rangle_\C + (\chi_r \, \theta)' = \eta \, \theta',
\end{equation}
which holds on the intervals $I_{2 r}^\pm$. In view of~\eqref{eq:hydro-E}, the
derivative $\theta'$ indeed lies in $L^2(I_{2 r}^\pm)$, while the function $\eta$ is in $L^2(\R)$. This is enough to guarantee that the function in~\eqref{castres} is integrable on $I_{2 r}^\pm$, so that the quantity $P_\theta(\psi)$ is well-defined. Moreover, its value does not depend on the choice of either the function $\chi$, or the number $r$, since
$$
\frac{1}{2} \int_\R \big( (\chi_r - \tilde{\chi}_{\tilde{r}}) \, \theta \big)' = 0,
$$
when the function $\tilde{\chi}$ and the number $\tilde{r}$ satisfy the assumptions of Lemma~\ref{lem:def-P-1}.

Note finally that
$$
\frac{1}{2} \int_{\R_-} \big( 2 \pi k_- \chi_r \big)' + \frac{1}{2} \int_{\R_+}
\big( 2 \pi k_+ \chi_r \big)' = \pi \big( k_+ - k_-),
$$
for $(k_-,k_+) \in \Z^2$. Statement $(ii)$ then follows from the fact that the phase function $\theta$ is defined up to a single phase factor $2 k \pi = 2 k_- \pi = 2 k_+ \pi$, when $\psi$ does not vanish. In the general case, we can add any phase factors $2 \pi k_\pm$ to the value of the phase $\theta$ on the intervals $I_R^\pm$. The previous computation then guarantees that we can fix this choice such that the quantity $P_\theta(\psi)$ lies in the interval $(- \pi/2, \pi/2]$, but also that this quantity is only known modulo $\pi$. This completes the proof of Lemma~\ref{lem:def-P-1}.
\end{proof}

We now turn to the regularity properties of the momentum $P$ and untwisted momentum $[P]$. In order to establish their continuity, we endow the energy set $X(\R)$ with a suitable metric structure. For a fixed number $0 \leq c < \sqrt{2}$, we introduce the weighted Sobolev space
$$
H_c(\R) := \big\{ \psi \in \boC^0(\R) \text{ s.t. } \psi' \in L^2(\R) \text{ and } \eta_c^{1/2} \psi \in L^2(\R) \big\}.
$$
This space is a Hilbert space for the norm given by the formula
$$
\| \psi \|_{H_c}^2 := \int_\R \big( |\psi'|^2 + \eta_c |\psi|^2 \big),
$$
where $\eta_c$ is given, as before, by~\eqref{def:eta-c}. Using the exponential decay of the functions $\eta_c$ and the $1/2$-H\"older continuity of the functions $\psi$ in $H_c(\R)$, we can check that all the norms $\| \cdot \|_{H_c}$ are equivalent. As a consequence, the space $H_c(\R)$ does not depend on $c$, and we set $H(\R) := H_c(\R)$ for simplicity. The energy set $X(\R)$ then appears as the subset of $H(\R)$ given by
$$
X(\R) = \big\{ \psi \in H(\R) \text{ s.t. } \eta = 1 - |\psi|^2 \in L^2(\R) \big\},
$$
and we can endow it with the metric structure corresponding to the distances
\begin{equation}
\label{def:d-c}
d_c(\psi_1, \psi_2) := \Big( \| \psi_1 - \psi_2 \|_{H_c}^2 + \| \eta_1 - \eta_2 \|_{L^2}^2 \Big)^\frac{1}{2}.
\end{equation}

This metric structure guarantees the continuity of the Ginzburg-Landau energy $E$, and it is also very convenient for dealing with the continuity of the momentum and the stability of the dark solitons (see e.g.~\cite{BetGrSa2, Gerard2, BeGrSaS1, GravSme1}). On the other hand, it is badly taylored to deal with the differentiability properties of the momentum (see~\cite{Gerard2}). This is the reason why we use an alternative approach to establish the differentiability of this quantity. This approach is based on the observation that the energy set $X(\R)$ is stable by addition of functions in $H^1(\R)$ (see~\cite[Lemma 1]{Gerard1}). In particular, given a function $\psi \in X(\R)$, the affine space $\psi + H^1(\R)$ provides a natural framework for tackling the differentiability of the momentum around the function $\psi$.

Before going into more details, we observe that the metric structure corresponding to the distances $d_c$ guarantees a uniform control on the modulus of the functions $\psi \in X(\R)$.

\begin{lem}
\label{lem:cont-modulus}
Let $0 \leq c < \sqrt{2}$ and consider a function $\psi_0 \in X(\R)$. Given any positive number $\varepsilon$, there exists a positive number $\delta$ such that, if $d_c(\psi, \psi_0) < \delta$, then
\begin{equation}
\label{eq:estim-eta}
\big\| |\psi|^2 - |\psi_0|^2 \big\|_{L^\infty} < \varepsilon.
\end{equation}
\end{lem}

\begin{proof}
We aim at establishing an $H^1$-control on the difference between the functions $\eta = 1 - |\psi|^2$ and $\eta_0 = 1 - |\psi_0|^2$. An $L^2$-control on this difference is directly provided by~\eqref{def:d-c}, so that we focus on the differences
\begin{equation}
\label{brive}
\eta' - \eta_0' = 2 \big( \langle \psi, \psi_0' - \psi' \rangle_\C + \langle \psi_0 - \psi, \psi_0' \rangle_\C \big).
\end{equation}
Observe first that
$$
\big\| 1 - |\psi| \big\|_{L^2} \leq \| \eta \|_{L^2} \leq \| \eta_0 \|_{L^2} + \delta \quad \text{and} \quad \big\| |\psi|' \big\|_{L^2} \leq \| \psi' \|_{L^2} \leq \| \psi_0' \|_{L^2} + \delta,
$$
when $d_c(\psi, \psi_0) < \delta$. Hence, by the Sobolev embedding theorem, there exists a positive number $C$ such that
\begin{equation}
\label{biarritz}
\big\| 1 - |\psi| \big\|_{L^\infty} \leq C \big( \| \psi_0' \|_{L^2} + \| \eta_0 \|_{L^2} + \delta \big). 
\end{equation}
Note in particular that the function $1 - |\psi_0|$ satisfies this inequality. With these bounds at hand, we estimate~\eqref{brive} as
\begin{equation}
\label{clermont}
\begin{split}
\big\| \eta' - \eta_0' \big\|_{L^2} \leq & 2 \| \psi \|_{L^\infty} \| \psi' -
\psi_0' \|_{L^2} + 2\| \psi - \psi_0 \|_{L^\infty([- R, R])} \| \psi_0' \|_{L^2}\\
& + 2\big( \| \psi \|_{L^\infty} + \| \psi_0 \|_{L^\infty} \big) \, \big( \| \psi_0' \|_{L^2(I_R^-)} + \| \psi_0' \|_{L^2(I_R^+)} \big).
\end{split}
\end{equation}
We next fix the choice of the positive number $R$ in this inequality such that 
$$
\| \psi_0' \|_{L^2(I_R^-)} + \| \psi_0' \|_{L^2(I_R^+)} \leq \delta.
$$
We then derive from~\eqref{def:eta-c},~\eqref{def:d-c} and the Sobolev embedding theorem the existence of a positive number $C$, depending only on $c$ and $R$, such that
$$
\| \psi - \psi_0 \|_{L^\infty([- R, R])} \leq C \| \psi - \psi_0 \|_{H_c} \leq C d_c(\psi, \psi_0).
$$
In view of~\eqref{biarritz} and~\eqref{clermont}, we are led to
$$
\big\| \eta' - \eta_0' \big\|_{L^2} \leq C \big( 1 + \| \psi_0' \|_{L^2} + \|
\eta_0 \|_{L^2} + \delta \big) \, d_c(\psi, \psi_0).
$$
Since $\| \eta - \eta_0 \|_{L^2} \leq d_c(\psi, \psi_0) < \delta$ by~\eqref{def:d-c}, we infer from the Sobolev embedding theorem that
$$
\big\| |\psi|^2 - |\psi_0|^2 \big\|_{L^\infty} = \big\| \eta - \eta_0 \big\|_{L^\infty} \leq C \big( 1 + \| \psi_0' \|_{L^2} + \| \eta_0 \|_{L^2} + \delta \big) \, d_c(\psi, \psi_0).
$$
for a further positive number $C$. In order to obtain~\eqref{eq:estim-eta}, we finally fix the choice of the positive number $\delta$ such that $C (1 + \| \psi_0' \|_{L^2} + \| \eta_0 \|_{L^2} + \delta) \delta < \varepsilon$. This completes the proof of Lemma~\ref{lem:cont-modulus}.
\end{proof}

We deduce from Lemma~\ref{lem:cont-modulus} that $NV\!X(\R)$ 
is an open subset of $X(\R)$. We also infer from this lemma that the momentum $P$ is continuous on this set. We additionally show that its natural differential at a function $\psi \in NV\!X(\R)$ is given by the function $i \psi'$.

\begin{lem}
\label{lem:cont-P}
The momentum $P$ is continuous on the non-vanishing energy set $NV\!X(\R)$. Moreover, given a function $\psi \in NV\!X(\R)$, there exists a positive number $\delta$ such that the ball $B(\psi, \delta) := \{ \psi + h : h \in H^1(\R) \text{ s.t. } \| h \|_{H^1} < \delta \}$ is a subset of $NV\!X(\R)$ on which
\begin{equation}
\label{eq:dev-P-H1}
P(\psi + h) = P(\psi) + \int_\R \langle i \psi', h \rangle_\C + \frac{1}{2} \int_\R \langle i h', h \rangle_\C.
\end{equation}
In particular, the restriction of the momentum $P$ to the ball $B(\psi, \delta)$ is continuously~\footnote{With respect to the metric structure induced by the $H^1$-norm.} differentiable, with
$$
dP(\psi)(h) = \int_\R \langle i \psi', h \rangle_\C,
$$
for any function $h \in H^1(\R)$.
\end{lem}

\begin{proof}
Recall that the momentum $P$ is well-defined on $NV\!X(\R)$ by the formula
$$
P(\psi) = \frac{1}{2} \int_\R \eta \theta',
$$
in which we have set, as before, $\psi = \rho e^{i \theta}$ and $\eta = 1 - \rho^2$. In particular, the continuity of this quantity will follow from the continuity from $NV\!X(\R)$ to $L^2(\R)$ of the maps $\psi \mapsto \eta$ and $\psi \mapsto \theta'$. Since the continuity of the first one is a direct consequence of~\eqref{def:d-c}, we focus on the continuity of the latter one.

Given a fixed function $\psi_0 = \rho_0 e^{i \theta_0} \in NV\!X(\R)$, we compute
$$
\theta_0' = - \frac{\langle i \psi_0', \psi_0 \rangle_\C}{\rho_0^2}.
$$
Extending this formula to an arbitrary function $\psi$ of $NV\!X(\R)$, we are led to the expression
$$
\theta' - \theta'_0 = - \frac{\langle i (\psi' - \psi_0'), \psi \rangle_\C}{\rho^2} - \langle i \psi_0', \psi \rangle_\C \frac{\rho_0^2 - \rho^2}{\rho^2 \rho_0^2} - \frac{\langle i \psi_0', \psi - \psi_0 \rangle_\C}{\rho_0^2}.
$$
For a positive number $\delta$ small enough, we deduce from Lemma~\ref{lem:cont-modulus} that
$$
\inf_{x \in \R} \rho(x) \geq \frac{m_0}{2} := \frac{1}{2} \inf_{x \in \R} \rho_0(x),
$$
when $d_c(\psi, \psi_0) < \delta$. Hence, we obtain
$$
\big\| \theta' - \theta_0' \big\|_{L^2} \leq \frac{1}{m_0^3} \Big( 2 m_0^2 \big\| \psi' - \psi_0' \big\|_{L^2} + 2 \| \psi_0' \|_{L^2} \big\| \rho^2 - \rho_0^2 \big\|_{L^\infty} + m_0 \big\| \psi_0' (\psi - \psi_0) \big\|_{L^2} \Big).
$$
Invoking~\eqref{def:d-c} for estimating the first norm in the right-hand side of
this inequality, Lemma~\ref{lem:cont-modulus} for the second one, and
arguing as in the proof of Lemma~\ref{lem:cont-modulus} for the last
one, we infer that the map $\psi \mapsto \theta'$ is continuous from $NV\!X(\R)$ to $L^2(\R)$.

Concerning differentiability, we deduce from the Sobolev embedding theorem the existence of a positive number $C$ such that
\begin{equation}
\label{toulon}
\inf_{x \in \R} |\psi(x) + h(x)| \geq \inf_{x \in \R} |\psi(x)| - \| h \|_{L^\infty} \geq \inf_{x \in \R} |\psi(x)| - C \| h \|_{H^1} > \frac{1}{2} \inf_{x \in \R} |\psi(x)| > 0,
\end{equation}
when $\| h \|_{H^1} < \delta = \inf_{x \in \R} |\psi(x)|/(2 C)$. In this case, the function $\psi + h$ belongs to $NV\!X(\R)$, so that the ball $B(\psi, \delta)$ is a subset of $NV\!X(\R)$.

We next consider a function $h \in \boC_c^\infty(\R)$ such that $\psi + h \in B(\psi, \delta)$. Combining the inequality
$$
\Big| \frac{\psi + h}{|\psi + h|} - \frac{\psi}{|\psi|} \Big| \leq \frac{3 |h|}{|\psi + h|} + \frac{|h|^2}{|\psi + h| (|\psi| + |\psi + h|)},
$$
with~\eqref{toulon} and the Sobolev embedding theorem, we can find a further positive number $C$, depending only on $\psi$, such that
$$
\Big\| \frac{\psi + h}{|\psi + h|} - \frac{\psi}{|\psi|} \Big\|_{L^\infty} \leq
C \big( 1 + \| h \|_{H^1} \big) \| h \|_{H^1}.
$$
Decreasing the value of $\delta$ if necessary, we can assume that
\begin{equation}
\label{perpignan}
\Big\| \frac{\psi + h}{|\psi + h|} - \frac{\psi}{|\psi|} \Big\|_{L^\infty} < 1.
\end{equation}
In another direction, it follows from the fact that $h$ has compact support that the phase functions $\theta_h$ and $\theta$ of the functions $\psi + h$, respectively $\psi$, are equal at $\pm \infty$ up to constants $2 k_\pm \pi$, with $k_\pm \in \Z$. We can choose the integer $k_- = 0$ and also deduce from~\eqref{perpignan} and a continuation argument that $|\theta_h - \theta| < 2 \pi$ on $\R$. In this case, we necessarily have $k_+ = 0$, so that $\theta_h = \theta$ at infinity.

Going back to~\eqref{def:P-1}, we can choose a cut-off function $\chi$ and a number $r$ in this definition such that the support of the functions $h$ and $\chi_r$ are disjoint. Since the values of the phase functions $\theta_h$ and $\theta$ are equal at $\pm \infty$, we have
\begin{align*}
P(\psi + h) = \frac{1}{2} \int_\R \big( \langle i (\psi' + h'), \psi + h \rangle_\C + (\chi_r \, \theta)' \big) = P(\psi) + \frac{1}{2} \int_\R \big( \langle i (\psi' + h'), h \rangle_\C + \langle i h', \psi \rangle_\C \big),
\end{align*}
which yields~\eqref{eq:dev-P-H1} by integrating by parts the last term in the right-hand side of the previous formula.

Given an arbitrary function $h \in H^1(\R)$, with $\psi + h \in B(\psi, \delta)$, we next introduce a sequence of functions $h_n \in \boC_c^\infty(\R)$ such that $h_n \to h$ in $H^1(\R)$ as $n \to \infty$. At least for $n$ large enough, we have
\begin{equation}
\label{pau}
P(\psi + h_n) = P(\psi) + \int_\R \langle i \psi', h_n \rangle_\C + \frac{1}{2} \int_\R \langle i h_n', h_n \rangle_\C.
\end{equation}
In the limit $n \to \infty$, the right-hand side of this identity tends to the right-hand side of~\eqref{eq:dev-P-H1}. Concerning the left-hand side, we show that $\psi + h_n \to \psi + h$ in $X(\R)$ as $n \to \infty$. This convergence holds in $H(\R)$ due to the property that $h_n \to h$ in $H^1(\R)$ as $n \to \infty$. Moreover, we compute
$$
\big( 1 - |\psi + h|^2 \big) - \big( 1 - |\psi + h_n|^2 \big) = 2 \langle \psi, h_n - h \rangle_\C + |h_n|^2 - |h|^2.
$$
Since the function $\psi$ is bounded on $\R$, it follows from the Sobolev embedding theorem that
$$
\big\| (1 - |\psi + h|^2) - (1 - |\psi + h_n|^2) \big\|_{L^2} \to 0,
$$
in the limit $n \to \infty$. Now that the convergence in $X(\R)$ is proved, we infer from the continuity of the momentum $P$ that the left-hand side of~\eqref{pau} tends to $P(\psi + h)$ as $n \to \infty$. This concludes the proof of~\eqref{eq:dev-P-H1}. The continuous differentiability of the restriction of $P$ to the ball $B(\psi, \delta)$ is then a direct consequence of the quadratic expansion in~\eqref{eq:dev-P-H1}. This completes the proof of Lemma~\ref{lem:cont-P}.
\end{proof}

At this stage, it is tempting to extend by continuity the momentum $P$ to the
whole set $X(\R)$, but this is not possible. Consider indeed two smooth cut-off functions $\chi : \R \to [0, 1]$ and $\theta : \R \to [0, 1]$, with $\chi(x) = 1$ for $|x| \leq 1$ and $\chi(x) = 0$ for $|x| \geq 2$, respectively $\theta(x) = 0$ for $x \leq - 2$ and $\theta(x) = 1$ for $x \geq 2$. Given a fixed integer $k \in\Z$, set
$$
\psi_n^k(x) = \Big( \gu_0(x) + \frac{i}{n} \chi(n x) \Big) e^{2 i \pi k \theta(n x)},
$$
for any $n \geq 1$. The functions $\psi_n^k$ belong to $NV\!X(\R)$ and they satisfy
$$
\big( 1 - |\psi_n^k(x)|^2 \big) - \big( 1 - |\gu_0(x)|^2 \big) = - \frac{1}{n^2} \chi(n x)^2,
$$
$$
\big| \psi_n^k(x) - \gu_0(x) \big| \leq |\gu_0(x)| \big| e^{2 i \pi k \theta(n x)} - 1 \big| + \frac{1}{n} \chi(n x),
$$
and
$$
\big| \big( \psi_n^k \big)'(x) - \gu_0'(x) \big| \leq |\gu_0'(x)| \big| e^{2 i \pi k \theta(n x)} - 1 \big| + |\chi'(n x)| + 2 \pi |k| \big( n |\gu_0(x)| |\theta'(n x)| + \chi(n x) |\theta'(n x)| \big).
$$
Using the inequality $|\gu_0(x)| \leq |x|/\sqrt{2}$ and applying the dominated convergence theorem, we deduce from the three previous formulae the convergence in $X(\R)$ of the functions $\psi_n^k$ towards the function $\gu_0$ as $n \to \infty$ for any fixed integer $k \in \Z$. On the other hand, we infer from the formula $\psi_n^{k + 1}(x) = \psi_n^k(x) \, e^{2 i \pi \theta(n x)}$ that
$$
P(\psi_n^{k + 1}) - P(\psi_n^k) = n \pi \int_{\R} \big( 1 - |\psi_n^k(x)|^2 \big) \theta'(n x) \, dx = \pi \int_\R \bigg( 1 - \Big| \gu_0 \Big( \frac{y}{n} \Big) \Big|^2 - \frac{1}{n^2} \chi(y) ^2 \bigg) \theta'(y) \, dy \to \pi,
$$
as $n \to \infty$. As a consequence, the momentum $P$ cannot be extended by continuity for the function $\gu_0$.

However, the previous counter-example fails to contradict the possible
continuity of a momentum that would only be defined modulo $\pi$, and we can
indeed show the continuity of the untwisted momentum $[P]$ on $X(\R)$.

\begin{lem}
\label{lem:cont-[P]}
The untwisted momentum $[P]$ is continuous on $X(\R)$. Moreover, it satisfies
\begin{equation}
\label{eq:dev-[P]-H1}
[P](\psi + h) = [P](\psi) + \int_\R \langle i \psi', h \rangle_\C + \frac{1}{2} \int_\R \langle i h', h \rangle_\C \quad \text{ modulo } \pi,
\end{equation}
for any functions $\psi \in X(\R)$ and $h \in H^1(\R)$.
\end{lem}

\begin{proof}
The proof of continuity is based on Lemma~\ref{lem:cont-modulus}. Consider a
function $\psi_0 \in X(\R)$ and choose a positive number $R$ such that
$|\psi_0| \geq 1/4$ on $I_R^\pm$. Applying Lemma~\ref{lem:cont-modulus}, we can find a positive number $\delta$ such that any function $\psi$ satisfies the condition $|\psi| \geq 1/2$ on $I_R^\pm$, as soon as $d_c(\psi, \psi_0) < \delta$ for a fixed number $0 \leq c < \sqrt{2}$. Setting as before $\psi = \rho e^{i \theta}$ and $\psi_0 = \rho_0 e^{i \theta_0}$ on $I_R^\pm$, the quantities $P_\theta(\psi)$ and $P_{\theta_0}(\psi_0)$ are then given by formula~\eqref{def:P-1} for a suitable cut-off function $\chi$ and a number $r > R$, which is independent of the function $\psi$ satisfying the condition $d_c(\psi, \psi_0) < \delta$. In particular, we obtain
\begin{equation}
\label{ferrand}
\begin{split}
P_\theta(\psi) - P_{\theta_0}(\psi_0) = & \frac{1}{2} \int_{|x| \leq 2 r} \Big( \langle i \psi', \psi \rangle_\C - \langle i \psi_0', \psi_0 \rangle_\C \Big) + \frac{1}{2} \Big( \theta(2 r) - \theta_0(2 r) - \theta(- 2 r) + \theta_0(- 2 r) \Big) \\
& + \frac{1}{2} \int_{|x| \geq 2 r} \Big( \eta \theta' - \eta_0 \theta_0' \Big),
\end{split}
\end{equation}
with $\eta = 1 - |\psi|^2$ and $\eta_0 = 1 - |\psi_0|^2$. When $d_c(\psi, \psi_0) \to 0$, the first term in the right-hand side of~\eqref{ferrand} tends to $0$ by definition of the $\| \cdot \|_{H_c}$-norm. Arguing as in the proof of Lemma~\ref{lem:cont-P}, we check that the third term also tends to $0$. Concerning the second one, we derive from the Sobolev embedding theorem that the convergence in $H(\R)$ implies the local uniform convergence. In particular, we have
$\psi(\pm 2 r) \to \psi_0(\pm 2r)$ as $d_c(\psi, \psi_0) \to 0$. Since $|\psi_0(\pm 2 r)| \geq 1/4$, this in turn implies that $e^{i \theta(\pm 2 r)} \to e^{i \theta_0(\pm 2r)}$, so that
$$
\theta(\pm 2 r) \to \theta_0(\pm 2 r) \quad \text{ modulo } 2 \pi.
$$
In view of~\eqref{ferrand}, we conclude that
$$
P_\theta(\psi) \to P_{\theta_0}(\psi_0) \quad \text{ modulo } \pi,
$$
which is enough to guarantee the continuity of the untwisted momentum $[P]$ on $X(\R)$.

Concerning~\eqref{eq:dev-[P]-H1}, we argue as for~\eqref{eq:dev-P-H1}. Assume first that $h$ is smooth and compactly supported. With the notation of Lemma~\ref{lem:def-P-1}, we can choose the number $R$ in the definition of the quantity $P_\theta(\psi)$ such that the support of $h$ is a subset of $[- R, R]$. In this case, the function $\psi + h$ owns the same phase $\theta$ as the function $\psi$ on the intervals $I_R^\pm$. Hence the quantity $P_\theta(\psi + h)$ is well-defined by
$$
P_\theta(\psi + h) = \frac{1}{2} \int_\R \Big( \big\langle i (\psi' + h'), \psi + h \big\rangle_\C + \big( \chi_r \theta \big)',
$$
which is equal to
$$
P_\theta(\psi + h) = P_\theta(\psi) + \int_\R \langle i \psi', h \rangle_\C + \frac{1}{2} \int_\R \langle i h', h \rangle_\C,
$$
by integration by parts. In view of Lemma~\ref{lem:def-P-1}, this is
exactly~\eqref{eq:dev-[P]-H1}. For an arbitrary $h \in H^1(\R)$, we
argue by density, as in the proof of Lemma~\ref{lem:cont-P}, using the continuity of the untwisted momentum and the property that the right-hand side of~\eqref{eq:dev-[P]-H1} is continuous with respect to the convergence in $H^1(\R)$. This completes the proof of Lemma~\ref{lem:cont-[P]}.
\end{proof}

Due to the previous dual definition of the momentum, two strategies are at hand when we aim at minimizing a quantity under a fixed momentum $p$. The first one is to minimize under a fixed untwisted momentum $[P]$, but in this case, the constraint $p$ must be assumed to be in $\R/\pi \Z$. The second one is to restrict the minimization set to the non-vanishing energy set $NV\!X(\R)$ in case it is possible to define the corresponding minimization problem for any number $p \in \R$. However, this minimization problem does not necessarily own a minimizer due to the fact that a minimizing sequence could converge to a function, which vanishes on $\R$, and so does not remain in $NV\!X(\R)$.

When the goal is to minimize the Ginzburg-Landau energy $E$, this second strategy leads to the minimization problem
\begin{equation}
\label{def:boI-problem}
\boI(p) := \inf \big\{ E(\psi) : \psi \in NV\!X(\R) \text{ s.t. } P(\psi) = p \big\},
\end{equation}
where the number $p$ can take any arbitrary value in $\R$. Note that this problem is well-defined. Consider indeed a function $\psi = \rho e^{i \theta} \in NV\!X(\R)$, with $P(\psi) \neq 0$ (for instance a dark soliton $\gu_c$ for $c \neq 0$) and set $\psi_\mu = \rho e^{i \mu \theta}$ for any number $\mu \in \R$. The functions $\psi_\mu$ remain in $NV\!X(\R)$ and their momentum
$$
P(\psi_\mu) = \mu P(\psi),
$$
take any arbitrary value in $\R$. Hence, the minimization problems $\boI(p)$ do make sense. An important tool in order to solve them is the following lemma.

\begin{lem}[\cite{BetGrSa2}]
\label{lem:min-E-0}
Let
$$
\boE_0 := \inf \Big\{ E(\psi) : \psi \in X(\R) \textup{ s.t. } \inf_{x \in \R} |\psi(x)| = 0 \Big\}.
$$
The black soliton $\gu_0$ is the unique minimizer of the minimization problem $\boE_0$ up to the invariances by translation and phase shift. In particular, when
$$
E(\psi) < \boE_0 = E(\gu_0) = \frac{2 \sqrt{2}}{3},
$$
the function $\psi$ does not vanish on $\R$, so that it belongs to $NV\!X(\R)$.
\end{lem}

Given a fixed number $p \in \R$, and provided that there exists a function $\psi \in X(\R)$ 
such that $E(\psi) < 2 \sqrt{2}/3$ and $P(\psi) = p$, Lemma~\ref{lem:min-E-0} guarantees 
that the possible limits of a minimizing sequence for the problem $\boI(p)$ still belong 
to $NV\!X(\R)$. This property was invoked in~\cite{BetGrSa2} to address the resolution 
of the minimization problem $\boI(p)$ for $|p| < \pi/2$. For an arbitrary choice
of $p$, we have

\begin{prop}
\label{prop:boI-problem}
$(i)$ For $|p| < \pi/2$, denote by $\gc_p$ the unique number in $(0, \sqrt{2})$, 
which solves
\begin{equation}
\label{def:gc-p-app}
\frac{\pi}{2} - \arctan \bigg( \frac{\gc_p}{\sqrt{2 - \gc_p^2}} \bigg) -
\frac{\gc_p}{2} \sqrt{2 - \gc_p^2} = |p|.
\end{equation}
and set $c_p = \text{sign}(p) \, \gc_p$. 
The dark soliton profile $\gu_{c_p}$ is the unique minimizer of the variational 
problem~\eqref{def:boI-problem} up to translation and phase shift. 
Moreover, the corresponding minimal value is given by
\begin{equation}
\label{eq:gE-val-app}
\gI(p) = E(\gu_{c_p}) = \frac{1}{3} \big(2 - c_p^2 \big)^\frac{3}{2}.
\end{equation}

$(ii)$ For $|p| \geq \pi/2$, the variational problem~\eqref{def:boI-problem} 
does not own any minimizer, and its minimal value is equal to
$$
\boI(p) = \frac{2 \sqrt{2}}{3}.
$$
\end{prop}

\begin{rem}
We can use Proposition~\ref{prop:boI-problem} to complement the proof of 
Proposition~\ref{prop:exist-min1} with respect to~\cite{BetGrSa2}. 
Observe indeed that
$$
\gI(p) \leq \inf_{k \in \Z} \boI(p + k \pi),
$$
for any number $p \in (- \pi/2, \pi/2]$. Combining Lemma~\ref{lem:min-E-0} 
and Proposition~\ref{prop:boI-problem}, we deduce that $\gI(p) = \boI(p)$ 
for $|p| < \pi/2$. In particular, the conclusion in 
Proposition~\ref{prop:exist-min1} follows from 
Proposition~\ref{prop:boI-problem} for this range of values of $p$.
\end{rem}

\begin{proof}[Proof of Proposition~\ref{prop:boI-problem}]

In view of Lemma~\ref{lem:min-E-0}, statement $(i)$ is 
exactly~\cite[Theorem 2]{BetGrSa2}. 
We turn now to statement $(ii)$. First, it was proved in~\cite[Theorem 2]{deLaMen1} 
that the minimal energy $\boI$ is a non-negative, even, continuous 
function on $\R$, whose restriction to $\R_+$ is concave. 
Moreover, it was computed in~\cite[Theorem 2]{BetGrSa2} that 
$$
\boI(p) = \frac{1}{3} \big( 2 - c_p^2 \big)^\frac{3}{2},
$$
for $0 \leq p \leq \pi/2$.
Since
\begin{equation}
\label{eq:deriv-cp-app}
\frac{d \gc_p}{dp} = - \frac{1}{(2 - \gc_p^2)^{\frac12}},
\end{equation}
$\boI$ is continuously differentiable on $(0, \pi/2)$ and
$$
\boI'(p) = \gc_p \to 0,
$$
as $p \to \pi/2$. Since $\boI$ is also concave on $\R$, 
we deduce that
\begin{equation}
\label{lou}
\boI(p) \leq \boI(\pi/2) = \frac{2 \sqrt{2}}{3},
\end{equation}
for any $p \geq \pi/2$.

Assume next the existence of a number $p > \pi/2$ such that 
$\boI(p) < 2 \sqrt{2}/3$. Since $\boI(\pi/2) = 2 \sqrt{2}/3$, 
we again infer from the concavity of the function $\boI$ the existence of a 
number $q > p$ such that $\boI(q) < 0$. This inequality contradicts the 
non-negativity of the function $\boI$, so that $\boI(p) \geq 2 \sqrt{2}/3$ 
for any number $p > \pi/2$. In view of~\eqref{lou}, this inequality is an 
equality, and since $\boI$ is an even function, it also holds for $p < - \pi/2$.

In order to complete the proof of statement $(ii)$, we next assume the 
existence of a minimizer $\psi_p$ for the variational problem $\boI(p)$ 
with $p \in \R \setminus (- \pi/2, \pi/2)$ being fixed. 
In view of Lemma~\ref{lem:cont-P}, this minimizer is characterized by 
the equation $dE(\psi_p) = \sigma dP(\psi_p)$ for a suitable Lagrange 
multiplier $\sigma \in \R$. The differentials $dE$ and $dP$ in this 
identity are chosen acting on the space $H^1(\R)$. Again by 
Lemma~\ref{lem:cont-P}, the minimizer $\psi_p$ is then a solution 
to~\eqref{eq:gu-p} in $X(\R)$. Since $P(\psi_p) \neq 0$, this solution is 
not constant. As a consequence, the minimizer $\psi_p$ is equal to the dark 
soliton $\gu_\sigma$ up to the invariances by translation and phase shift. 
In particular, the number $\sigma$ lies in $(- \sqrt{2}, \sqrt{2})$, 
with $\sigma \neq 0$ since the black soliton vanishes. However, it follows 
from~\cite[Proposition 1]{BetGrSa2} that the momentum $P(\gu_\sigma)$ belongs 
to the interval $(-\pi/2, \pi/2)$. This contradicts the fact that 
$|P(\psi_p)| \geq \pi/2$, so that there is no minimizer for $|p| \geq \pi/2$. 
\end{proof}

%%%%%%%%%%%%%%%%%%%%%%%%%%%%%%%%%
%%%%%%%%%%%%%%%%%%%%%%%%%%%%%%%%%
%%%%%%%%%%%%%%%%%%%%%%%%%%%%%%%%%
\section{Properties of the energy set $X(\R \times \T)$}
\label{sec:energy-set}
%%%%%%%%%%%%%%%%%%%%%%%%%%%%%%%%%
%%%%%%%%%%%%%%%%%%%%%%%%%%%%%%%%%
%%%%%%%%%%%%%%%%%%%%%%%%%%%%%%%%%

In this section, we gather some properties of the energy set
$$
X(\R \times \T) = \Big\{ \psi \in H_\text{loc}^1(\R \times \T) : \nabla \psi \in L^2(\R \times \T) \text{ and } 1 - |\psi|^2 \in L^2(\R \times \T) \Big\},
$$
which are required for defining properly the momentum and providing a suitable functional framework to solve the minimization problems $\boI_\lambda(p)$. The derivation of these properties heavily relies on the following links between the energy sets $X(\R)$ and $X(\R \times \T)$.

\begin{prop}
\label{prop:X-1-2}
Let $\lambda$ be a fixed positive number.

$(i)$ Given a function $\psi \in X(\R)$, set $\Psi(x, y) = \psi(x)$ for any $(x, y) \in \R \times \T$. The function $\Psi$ is in $X(\R \times \T)$, with
$$
E_\lambda(\Psi) = E(\psi).
$$

$(ii)$ Given a function $\psi \in X(\R \times \T)$, set $\hat{\psi}_0(x) =
\int_0^1 \psi(x, y) \, dy$ and $w_0(x, y) = \psi(x, y)- \hat{\psi}_0(x)$
for almost any $(x, y) \in \R \times \T$. The functions $\hat{\psi}_0$ and $w_0$ belong to $X(\R)$, respectively $H^1(\R \times \T)$, with
\begin{equation}
\label{eq:dev-E-1-2}
\begin{split}
E_\lambda(\psi) = & E(\hat{\psi}_0) + \frac{1}{2} \int_{\R \times \T} \big( |\partial_x w_0|^2 + \lambda^2 |\partial_y w_0|^2 \big)\\
& + \int_{\R \times \T} \Big( \langle \hat{\psi}_0, w_0 \rangle_\C^2 - \frac{1}{2}
|w_0|^2 (1 - |\hat{\psi}_0|^2) + |w_0|^2 \langle \hat{\psi}_0, w_0 \rangle_\C + \frac{1}{4} |w_0|^4 \Big).
\end{split}
\end{equation}
\end{prop}

\begin{rem}
\label{rem:1-2}
In view of statement $(i)$, we have made the choice to use the same notation for all the objects and quantities that are defined identically on $\R$ and $\R \times \T$. With a slight abuse of notation, we have also identified any function in $X(\R)$ with the corresponding function in $X(\R \times \T)$.
\end{rem}

\begin{proof}
Statement $(i)$ is a direct consequence of the property that the torus $\T$ has a finite measure equal to $1$ and that the derivative $\partial_y \psi$ of a function $\psi \in X(\R \times \T)$ depending only on the variable $x$ is equal to $0$.

Concerning statement $(ii)$, we first infer from the Plancherel formula that the gradients $\nabla \hat{\psi}_0$ and $\nabla w_0$ belong to $L^2(\R \times \T)$, with moreover $\partial_y \hat{\psi}_0 = 0$ and $\partial_y w_0 = \partial_y \psi$. Invoking the Poincar\'e-Wirtinger inequality, we obtain
$$
\big\| w_0 \big\|_{L^2(\R \times \T)} \leq \frac{1}{2 \pi} \big\| \partial_y \psi \big\|_{L^2(\R \times \T)},
$$
so that the function $w_0$ is indeed in $H^1(\R \times \T)$.

By definition, we also compute
$$
1 - |\hat{\psi}_0|^2 = 1 - |\psi|^2 + 2 \langle \psi, w_0 \rangle_\C - |w_0|^2.
$$
Using the inequality
$$
|\psi| \leq \sqrt{2} \, \1_{|\psi| \leq \sqrt{2}} + 2\sqrt{|\psi|^2 - 1} \, \1_{|\psi| > \sqrt{2}},
$$
we deduce from the Sobolev embedding theorem that the functions $\langle \psi, w_0 \rangle_\C$, and then $1 - |\hat{\psi}_0|^2$ are in $L^2(\R \times \T)$. Since $\hat{\psi}_0$ only depends on the variable $x$, we conclude that this function lies in $X(\R)$. Formula~\eqref{eq:dev-E-1-2} finally follows from the fact that the functions $w_0$ and $\nabla w_0$ are orthogonal in $L^2(\R \times \T)$ to all the functions depending only on the variable $x$.
\end{proof}

\begin{rem}
\label{rem:X+H-2}
Arguing as for the proof that the function $1 - |\hat{\psi}_0|^2$ is in $L^2(\R \times \T)$, we can show that a function of the form $\psi + w$ belongs to $X(\R \times \T)$ when $\psi$ and $w$ are in $X(\R \times \T)$, respectively $H^1(\R \times \T)$.
\end{rem}

Statement $(ii)$ in Proposition~\ref{prop:X-1-2} provides a uniquely determined decomposition of an arbitrary function $\psi \in X(\R \times \T)$ as a function in $X(\R)$ plus a function in $H^1(\R \times \T)$. It is natural to take into account this decomposition in order to endow the energy set $X(\R \times \T)$ with a metric structure. In this direction, we first set
$$
H(\R \times \T) := \Big\{ \psi = \hat{\psi}_0 + w_0 \in H_\text{loc}^1(\R \times \T) \text{ s.t. } \hat{\psi}_0 \in H(\R) \text{ and } w_0 \in H^1(\R \times \T) \Big\}.
$$
The set $H(\R \times \T)$ is then a Hilbert space for the norms given by the formula
\begin{equation}
\label{def:norm-Hc-2}
\| \psi \|_{H_c}^2 = \int_{\R \times \T} \big( |\nabla \psi|^2 + \eta_c |\psi|^2 \big),
\end{equation}
for $0 \leq c < \sqrt{2}$. This definition is exactly the same as the one of the norm $\| \cdot \|_{H_c}$ in $H(\R)$, so that we have kept the same notation. Observe in particular that the norm $\| \psi \|_{H_c}$ in $H(\R \times \T)$ of a function $\psi \in H(\R)$ is exactly equal to its norm $\| \psi \|_{H_c}$ in $H(\R)$.

Note also that the previous norm is equivalent to the norm given by
\begin{equation}
\label{eq:equiv-norm}
\| \psi \|^2 = \| \hat{\psi}_0 \|_{H_c}^2 + \| w_0 \|_{H^1}^2.
\end{equation}
Due to the orthogonality of the functions $\nabla \hat{\psi}_0$ and $\nabla w_0$
in $L^2(\R \times \T)$, the norm $\| \psi \|_{H_c}$ indeed controls the norms $\| \nabla \hat{\psi}_0 \|_{L^2}$ and $\| \nabla w_0 \|_{L^2}$, and then the norm $\| w_0 \|_{L^2}$ by the Poincar\'e inequality. The reverse inequality follows from the property that the norm $\| w_0 \|_{H^1}$ controls the norm $\| w_0 \|_{H_c}$.

At this stage, it is natural to endow the energy set $X(\R \times \T)$ with the metric structure corresponding to the distances
$$
d_c(\psi_1, \psi_2) := \Big( \| \psi_1 - \psi_2 \|_{H_c}^2 + \| \eta_1 - \eta_2 \|_{L^2}^2 \Big)^\frac{1}{2},
$$
with $\eta_1 = 1 - |\psi_1|^2$ and $\eta_2 = 1 - |\psi_2|^2$, as before. This definition is again exactly the same as in $X(\R)$, and the distance $d_c(\psi_1, \psi_2)$ in $X(\R \times \T)$ between functions $\psi_1$ and $\psi_2$ in $X(\R)$ remains equal to their distance in $X(\R)$. This is the reason why we have again kept the same notation for the two quantities. A useful property of this metric structure is

\begin{lem}
\label{lem:conv-1-2}
$(i)$ Let $\psi = \hat{\psi}_0 + w_0 \in X(\R \times \T)$. Consider a sequence of functions $\psi^n \in X(\R \times \T)$ such that $\psi^n \to \psi$ in $X(\R \times \T)$ as $n \to \infty$ and denote $\psi^n = \hat{\psi}_0^n + w_0^n$ the decomposition given by Proposition~\ref{prop:X-1-2}. In the limit $n \to \infty$, we have
$$
\hat{\psi}_0^n \to \hat{\psi}_0 \in X(\R) \quad \text{and} \quad w_0^n \to w_0 \text{ in } H^1(\R \times \T).
$$

$(ii)$ Let $g \in X(\R)$, $h \in H^1(\R \times \T)$, and set $\psi = g + h$.
Consider sequences of functions $g_n \in X(\R)$ and $h_n \in H^1(\R
\times \T)$ such that $g_n \to g$ in $X(\R)$, and $h_n \to h$ in $H^1(\R \times \T)$, as $n \to \infty$. Then, the functions $\psi_n = g_n + h_n$ satisfy
\begin{equation}
\label{eq:conv-1-2}
\psi_n \to \psi \text{ in } X(\R \times \T),
\end{equation}
as $n \to \infty$.
\end{lem}

\begin{proof}
Concerning statement $(i)$, we deduce from the equivalence between the $H_c$-norms and the norms in~\eqref{eq:equiv-norm} that $\hat{\psi}_0^n \to \hat{\psi}_0$ in $H(\R)$ and $w_0^n \to w_0$ in $H^1(\R \times \T)$. The fact that $1 - |\hat{\psi}_0^n|^2 \to 1 - |\hat{\psi}_0|^2$ in $L^2(\R)$ then follows from the identity
\begin{equation}
\label{montpellier}
\big( 1 - |\hat{\psi}_0^n|^2 \big) - \big( 1 - |\hat{\psi}_0|^2 \big) = \big( |\psi|^2 - |\psi^n|^2 \big) + \big( |w_0^n|^2 - |w_0|^2 \big) + 2 \langle \hat{\psi}_0^n - \hat{\psi}_0, w_0 \rangle_\C + 2 \langle \hat{\psi}_0^n, w_0^n - w_0 \rangle_\C.
\end{equation}
The first term in the right-hand side of this expression tends to $0$ in $L^2(\R
\times \T)$ due to the convergence $\psi_n \to \psi$ in $X(\R \times
\T)$. The second one also tends to $0$ in $L^2(\R \times \T)$ due to the
convergence $w_0^n \to w_0$ in $H^1(\R \times \T)$ and the Sobolev
embedding theorem. For the third one, we recall that the convergence in
$H(\R)$ implies the convergence in $L_\text{loc}^\infty(\R)$ by the
Sobolev embedding theorem. Since the energy set $X(\R)$ is a subset of
$Z^1(\R)$, the function $\hat{\psi}_0$ is also bounded on $\R$. In
particular, it follows from the dominated convergence theorem that the
third term in~\eqref{montpellier} also converges to $0$ in $L^2(\R
\times \T)$. In view of Lemma~\ref{lem:cont-modulus}, the functions
$\hat{\psi}_0^n$ are then uniformly bounded on $\R$. Similarly, the fourth term in~\eqref{montpellier} also tends to $0$ in $L^2(\R \times \T)$ due to the convergence $w_0^n \to w_0$ in $H^1(\R \times \T)$. In conclusion, the left-hand side of~\eqref{montpellier} converges to $0$ in $L^2(\R \times \T)$, and then in $L^2(\R)$ since it only depends on the variable $x$.

The proof of statement $(ii)$ is very similar. Observe first that the functions
$\psi$ and $\psi_n$ are in $X(\R \times \T)$ by
Proposition~\ref{prop:X-1-2} and Remark~\ref{rem:X+H-2}. The convergence
$\psi_n \to \psi$ in $H(\R\times \T)$ then follows from the fact that the $H^1$-norm controls the $H_c$-norms. Moreover, we compute
$$
\big( 1 - |\psi_n|^2 \big) - \big( 1 - |\psi|^2) = \big( |g|^2 - |g_n|^2 \big) + \big( |h|^2 - |h_n|^2 \big) + 2 \langle g - g_n, h \rangle_\C + 2 \langle g_n, h - h_n \rangle_\C.
$$
The convergence $1 - |\psi_n + w_n|^2 \to 1 - |\psi + w|^2$ in $L^2(\R \times \T)$ follows as for~\eqref{montpellier}. This completes the proofs of~\eqref{eq:conv-1-2} and of Lemma~\ref{lem:conv-1-2}.
\end{proof}

Note also that the energies $E_\lambda$ are continuous with respect to the distances $d_c$. Moreover, we can show the following density result, which is useful for describing the minimal energy $\boI_\lambda$.

\begin{cor}
\label{cor:dens-X}
Let $\lambda$ be a positive number. Consider a function $\psi \in X(\R \times \T)$ and decompose it as $\psi = \hat{\psi}_0 + w_0$ according to Proposition~\ref{prop:X-1-2}. There exist two sequences of functions $g_n \in X(\R)$ and $h_n \in H^1(\R \times \T)$, which satisfy the following properties.

$(i)$ The functions $g_n$ are smooth on $\R$ and there exist numbers $R_n^\pm > 0$ and $\theta_n^\pm \in \R$ for which $g_n(x) = e^{i \theta_n^\pm}$ for any $\pm x \geq \pm R_n^\pm$. 

$(ii)$ The functions $h_n$ are smooth on $\R \times \T$ and compactly supported in $[- R_n^-, R_n^+] \times \T$.

$(iii)$ We have the convergences
\begin{equation}
\label{eq:conv-g-h}
g_n \to \hat{\psi}_0 \text{ in } X(\R) \quad \text{and} \quad h_n \to w_0 \text{ in } H^1(\R \times \T),
\end{equation}
as $n \to \infty$.

$(iv)$ The functions $\psi_n = g_n + h_n$ are in $X(\R \times \T)$, with
$$
\psi_n \to \psi \text{ in } X(\R \times \T) \quad \text{ and } \quad E_\lambda(\psi_n) \to E_\lambda(\psi),
$$
as $n \to \infty$.
\end{cor}

\begin{proof}
The proof is based on a decomposition of the functions in $X(\R)$, which was established by P.~G\'erard in~\cite[Theorem 1.8]{Gerard2}. Given an arbitrary function $\tilde{\psi} \in X(\R)$, there exist a real-valued function $\tilde{\phi} \in \boC^0(\R)$, with $\tilde{\phi}' \in L^2(\R)$, and a complex-valued function $\tilde{\varpi} \in H^1(\R)$ such that
\begin{equation}
\label{eq:psi-decomp}
\tilde{\psi} = e^{i \tilde{\phi}} + \tilde{\varpi}.
\end{equation}
Moreover, the phase function $\tilde{\phi}$ is determined up to adding a real-valued function $\varphi \in \boC^0(\R)$, with $\varphi' \in L^2(\R)$, and such that there exist $k_\pm \in \Z$ with $\varphi - 2 \pi k_\pm \in L^2(\R_\pm)$. Since $\hat{\psi}_0$ belongs to $X(\R)$ by Proposition~\ref{prop:X-1-2}, we can decompose it as $\hat{\psi}_0 = e^{i \phi} + \varpi$, with $\phi$ and $\varpi$ satisfying the previous conditions. 

We next invoke the density of smooth, compactly supported functions in $L^2(\R)$ and $H^1(\R)$ so as to find two sequences of functions $\varphi_n$ and $\varpi_n$ in $\boC_c^\infty(\R)$ such that $\varphi_n \to \phi'$ in $L^2(\R)$, and $\varpi_n \to \varpi$ in $H^1(\R)$, as $n \to \infty$. Since the function $\phi$ is continuous, we are then allowed to define functions $\phi_n$ by the formula
$$
\phi_n(x) = \phi(0) + \int_0^x \varphi_n(t) \, dt, 
$$
for any $x \in \R$. By the inequality
$$
|\phi_n(x) - \phi(x)| \leq \bigg| \int_0^x \big( \varphi_n(t) - \phi'(t) \big) \, dt \bigg| \leq \sqrt{R} \, \big\| \varphi_n - \phi' \big\|_{L^2([- R, R])},
$$
which holds for any positive number $R$, we obtain that $\phi_n \to \phi$ in $L_\text{loc}^\infty(\R)$, while in addition $\phi'_n \to \phi'$ in $L^2(\R)$, when $n \to \infty$.

At this stage, we set $g_n = e^{i \phi_n} + \varpi_n$. The functions $g_n$ satisfy statement $(i)$ in Corollary~\ref{cor:dens-X}. Given a number $0 \leq c < \sqrt{2}$, we moreover have
$$
\eta_c^\frac{1}{2} \big( g_n - \hat{\psi}_0 \big) = \eta_c^\frac{1}{2} \big( e^{i \phi_n} - e^{i \phi} \big) + \eta_c^\frac{1}{2} \big( \varpi_n - \varpi \big),
$$
$$
g_n' - \hat{\psi}_0' = i \big( \phi_n' - \phi' \big) e^{i \phi_n} + i \phi' (e^{i \phi_n} - e^{i \phi}) + \varpi_n' - \varpi',
$$
as well as
$$
\big( 1 - |g_n|^2 \big) - \big( 1 - |\hat{\psi}_0|^2 \big) = 2 \langle e^{i \phi} - e^{i \phi_n}, \varpi \rangle_\C + 2 \langle e^{i \phi_n}, \varpi - \varpi_n \rangle_\C + |w|^2 - |w_n|^2.
$$
Invoking the Sobolev embedding theorem, and applying the dominated convergence theorem when necessary, we are led to
$$
d_c(g_n, \hat{\psi}_0) \to 0,
$$
as $n \to \infty$. Note also that $g_n \to \hat{\psi}_0$ in $L_\text{loc}^\infty(\R)$ by the Sobolev embedding theorem.

We finally complete the proof of statement $(iii)$, and provide the one of statement $(ii)$, by introducing a further sequence of functions $h_n \in \boC_c^\infty(\R \times \T)$ such that $h_n \to w_0$ in $H^1(\R \times \T)$, as $n \to \infty$.

The convergence of the functions $\psi_n = g_n + h_n$ towards the function $\psi$ in $X(\R \times \T)$ is then a direct consequence of statement $(iii)$ and Lemma~\ref{lem:conv-1-2}. The convergence of the energies $E_\lambda(\psi_n)$ towards the energy $E_\lambda(\psi)$ follows by continuity of the energy $E_\lambda$ on $X(\R \times \T)$. This completes the proof of Corollary~\ref{cor:dens-X}.
\end{proof}

%%%%%%%%%%%%%%%%%%%%%%%%%%%%%%%
%%%%%%%%%%%%%%%%%%%%%%%%%%%%%%%
%%%%%%%%%%%%%%%%%%%%%%%%%%%%%%%
\section{Definition and properties of the momentum}
\label{sec:momentum}
%%%%%%%%%%%%%%%%%%%%%%%%%%%%%%%
%%%%%%%%%%%%%%%%%%%%%%%%%%%%%%%
%%%%%%%%%%%%%%%%%%%%%%%%%%%%%%%

In this section, we provide the definition of the momentum in the energy set $X(\R \times \T)$ and describe its main properties. Our starting point is the decomposition $\psi = \hat{\psi}_0 + w_0$ of a function $\psi \in X(\R \times \T)$, which is given by Proposition~\ref{prop:X-1-2}. Using this decomposition, the formal density of momentum writes as
$$
\langle i \partial_x \psi, \psi \rangle_\C = \langle i \partial_x \hat{\psi}_0, \hat{\psi}_0 \rangle_\C + \langle i \partial_x \hat{\psi}_0, w_0 \rangle_\C + \langle i \partial_x w_0, \hat{\psi}_0 \rangle_\C + \langle i \partial_x w_0, w_0 \rangle_\C.
$$
The first term in the right-hand side of this identity is the formal density of
the momentum of a function $\hat{\psi}_0 \in X(\R)$, so that we can define it
rigorously by invoking Lemma~\ref{lem:def-P-1}. The second and third terms are
scalar products of functions, which are at least formally orthogonal in $L^2(\R \times \T)$. Hence, their integral is at least formally equal to $0$. Finally, the last term is integrable on $\R \times \T$ since $w_0 \in H^1(\R \times \T)$. As a conclusion, it is natural to define the momentum of the function $\psi$ as
$$
P(\psi) = Q(\hat{\psi}_0) + \frac{1}{2} \int_{\R \times \T} \langle i \partial_x w_0, w_0 \rangle_\C.
$$
In this expression, the quantity $Q(\hat{\psi}_0)$ refers to a 1D momentum of the function $\hat{\psi}_0$, which can be either equal to the quantity $P_{\theta_0}(\hat{\psi}_0)$ in~\eqref{def:P-1}, the momentum $P(\hat{\psi}_0)$ when $\hat{\psi}_0 \in NV\!X(\R)$, or the untwisted momentum $[P](\hat{\psi}_0)$. More precisely, we have

\begin{lem}
\label{lem:def-P-2}
Given a function $\psi$ in $X(\R \times \T)$, decompose it as $\psi = \hat{\psi}_0 + w_0$, with $\hat{\psi}_0$ and $w_0$ as in Proposition~\ref{prop:X-1-2}. Consider a positive number $R_0$ such that $|\hat{\psi}_0(x)| \geq 1/2$ for $|x| \geq R_0$ and a phase function $\theta_0 \in \boC^0(I_{R_0}^\pm)$ such that $\hat{\psi}_0 = |\hat{\psi}_0| e^{i \theta_0}$ on $I_{R_0}^\pm$. Choose a smooth cut-off function $\chi : \R \to [0, 1]$ such that $\chi(x) = 0$ for $|x| \leq 1$ and $\chi(x) = 1$ for $|x| \geq 2$, and set $\chi_r(x) = \chi(x/r)$ for a number $r > R_0$.

$(i)$ The quantity
\begin{equation}
\label{def:P-2}
P_{\theta_0}(\psi) = P_{\theta_0}(\hat{\psi}_0) + \frac{1}{2} \int_{\R \times \T} \langle i \partial_x w_0, w_0 \rangle_\C,
\end{equation}
is well-defined and does not depend on the choice of neither the function $\chi$, nor the number $r$.

$(ii)$ When the function $\hat{\psi}_0$ does not vanish on $\R$, the quantity $P_{\theta_0}(\psi)$ does not depend on the choice of the phase function $\theta_0$. In the sequel, this quantity is called momentum and simply denoted by $P(\psi)$.

$(iii)$ In the general case, the value modulo $\pi$ of the quantity $P_{\theta_0}(\psi)$ does not depend on the choice of the phase function $\theta_0$, and it is possible to fix this choice such that $P_{\theta_0}(\psi) \in (- \pi/2, \pi/2]$. In particular, the untwisted momentum $[P] : X(\R \times \T) \to \R / \pi \Z$ defined by $[P](\psi) = P_{\theta_0}(\psi)$ modulo $\pi$ is well-defined.
\end{lem}

\begin{rem}
\label{rem:def-P-1-2}
In view of Remark~\ref{rem:1-2}, a function $\psi \in X(\R)$ is also a function in $X(\R \times \T)$, so that we can define its momentum as a function in $X(\R)$ or in $X(\R \times \T)$. Lemma~\ref{lem:def-P-2} guarantees that these definitions are identical whatever is the definition of the momentum ($P_\theta(\psi)$, $P(\psi)$ or $[P](\psi)$) under consideration. In this case, the functions $\psi$ and $\hat{\psi}_0$ are indeed equal, so that the function $w_0$ identically vanishes.
\end{rem}

\begin{proof}
Lemma~\ref{lem:def-P-2} is a direct consequence of Lemma~\ref{lem:def-P-1} since
the term depending on the function $w_0$ in~\eqref{def:P-2} is
well-defined for $w_0 \in H^1(\R \times \T)$.
\end{proof}

At this stage, it is natural to introduce the set
$$
Y(\R \times \T) := \Big\{ \psi = \hat{\psi}_0 + w_0 \in X(\R \times \T) \text{ s.t. } \hat{\psi}_0 \in NV\!X(\R) \Big\}.
$$
Though this open set plays the role of the set $NV\!X(\R)$ in the context of the
product space $\R \times \T$, it is not the subset $NV\!X(\R \times \T)$ of
non-vanishing functions in $X(\R \times \T)$. With the definition of $Y(\R
\times \T)$ at hand, we can extend Lemmas~\ref{lem:cont-P} and~\ref{lem:cont-[P]} as

\begin{lem}
\label{lem:cont-P2}
$(i)$ The momentum $P$ is continuous on the subset $Y(\R \times \T)$. Moreover, given a function $\psi \in Y(\R \times \T)$, there exists a positive number $\delta$ such that the ball $B(\psi, \delta) := \{ \psi + h : h \in H^1(\R \times \T) \text{ s.t. } \| h \|_{H^1} < \delta \}$ is a subset of $Y(\R \times \T)$ on which
\begin{equation}
\label{eq:dev-P-H1-2}
P(\psi + h) = P(\psi) + \int_{\R \times \T} \langle i \partial_x \psi, h \rangle_\C + \frac{1}{2} \int_{\R \times \T} \langle i \partial_x h, h \rangle_\C.
\end{equation}
In particular, the restriction of the momentum $P$ to the ball $B(\psi, \delta)$ is continuously~\footnote{With respect to the metric structure induced by the $H^1$-norm.} differentiable, with
$$
dP(\psi)(h) = \int_{\R \times \T} \langle i \partial_x \psi, h \rangle_\C,
$$
for any function $h \in H^1(\R \times \T)$.

$(ii)$ The untwisted momentum $[P]$ is continuous on $X(\R \times \T)$.
\end{lem}

\begin{proof}
The continuity of the momentum $P$ and untwisted momentum $[P]$ is a direct consequence of Lemmas~\ref{lem:cont-P} and~\ref{lem:cont-[P]} applying statement $(i)$ of Lemma~\ref{lem:conv-1-2}.

Concerning the proof of~\eqref{eq:dev-P-H1-2}, we consider a function $\psi = \hat{\psi}_0 + w_0 \in Y(\R \times \T)$ and invoke Lemma~\ref{lem:cont-P} in order to exhibit a positive number $\delta$ such that the functions $\hat{\psi}_0 + g$ lie in $NV\!X(\R)$ when $g \in H^1(\R)$ with $\| g \|_{H^1} < \delta$. Assume here that $h \in H^1(\R \times \T)$ with $\| h \|_{H^1} < \delta$. We can decompose $h$ as $h = \hat{h}_0 + w$, with $\hat{h}_0(x) = \int_\T h(x, y) \, dy$ as before, and use the orthogonality of this decomposition in order to check that $\| \hat{h}_0 \|_{H^1} < \delta$. As a consequence, the function $\psi + h = \hat{\psi}_0 + \hat{h}_0 + w_0 + w$ lies in $Y(\R \times \T)$, which amounts to say that the ball $B(\psi, \delta)$ is a subset of $Y(\R \times \T)$. Moreover, we can combine~\eqref{eq:dev-P-H1} and~\eqref{def:P-2} in order to develop the quantity $P(\psi + h)$ as
\begin{align*}
P \big( \psi + h \big) = & P \big( \hat{\psi}_0 + \hat{h}_0 \big) + \frac{1}{2} \int_{\R \times \T} \big\langle i \partial_x (w + w_0), w +w_0 \rangle_\C \\
= & P \big( \hat{\psi}_0 \big) + \frac{1}{2} \int_{\R \times \T} \big\langle i \partial_x w_0, w_0 \big\rangle_\C + \int_{\R \times \T} \Big( \big\langle i \partial_x \hat{\psi}_0, \hat{h}_0 \big\rangle_\C + \big\langle i \partial_x w_0, w \big\rangle_\C \Big) \\
& + \frac{1}{2} \int_{\R \times \T} \Big( \big\langle i \partial_x \hat{h}_0, \hat{h}_0 \big\rangle_\C + \big\langle i \partial_x w, w \big\rangle_\C \Big).
\end{align*}
Formula~\eqref{eq:dev-P-H1-2} then follows from the orthogonality conditions between the functions $\hat{\psi}_0$ and $\hat{h}_0$ on the one hand, and $w$ and $w_0$ on the other hand. The value of the differential $dP(\psi)$ and its continuity are then a direct consequence of this formula. This ends the proof of Lemma~\ref{lem:cont-P2}.
\end{proof}

We next relate the momentum of a function $\psi \in X(\R \times \T)$ with the untwisted momenta of its slices $\psi(\cdot, y)$ for $y$ ranging in $\T$.

\begin{lem}
\label{lem:P-1-2}
Let $\psi = \hat{\psi}_0 + w_0 \in X(\R \times \T)$. Consider a positive number $R_0$ such that $|\hat{\psi}_0(x)| \geq 1/2$ for $|x| \geq R_0$ and a phase function $\theta_0 \in \boC^0(I_{R_0}^\pm)$ such that $\hat{\psi}_0 = |\hat{\psi}_0| \, e^{i \theta_0}$ on $I_{R_0}^\pm$. For almost every $y \in \T$, the functions $w_0(\cdot, y)$ and $\psi(\cdot, y)$ are well-defined in $H^1(\R)$, respectively in $X(\R)$. In particular, the quantities
\begin{equation}
\label{def:p-slice}
p_{\theta_0} \big( \psi(\cdot, y) \big) := P_{\theta_0}(\hat{\psi}_0) + \int_\R \big\langle i \hat{\psi}_0', w_0(\cdot , y) \big\rangle_\C + \frac{1}{2} \int_\R \big\langle i \partial_x w_0(\cdot, y), w_0(\cdot , y) \big\rangle_\C,
\end{equation}
are well-defined for almost any $y \in \T$, and they satisfy
\begin{equation}
\label{eq:P-Fubini}
P_{\theta_0} \big( \psi \big) = \int_\T p_{\theta_0} \big( \psi(\cdot, y) \big) \, dy,
\end{equation}
as well as
\begin{equation}
\label{eq:p-slice}
p_{\theta_0} \big( \psi(\cdot, y) \big) = [P] \big( \psi(\cdot, y) \big) \quad \text{ modulo } \pi.
\end{equation}
\end{lem}

\begin{proof}
Recall first that $w_0$ is in $H^1(\R \times \T)$, so that the slices
$w_0(\cdot, y)$ belong to $H^1(\R)$ for almost any $y \in \T$. Since the
energy set $X(\R)$ remains stable by addition of functions in $H^1(\R)$,
the slices $\psi(\cdot, y)$ are in $X(\R)$ for almost any $y \in \T$.
The quantity $p_{\theta_0}(\psi(\cdot, y))$ is also well-defined and
depends only on the function $\psi(\cdot, y)$ due to the uniqueness of
the decomposition $\psi(\cdot, y) = \hat{\psi}_0 + w_0(\cdot, y)$.

Going back to the definition of the quantity $P_{\theta_0}(\psi)$ in Lemma~\ref{lem:def-P-2} and using the fact that $\hat{\psi}_0'$ and $w_0$ are orthogonal in $L^2(\R \times \T)$, we next invoke the Fubini theorem in order to write
\begin{align*}
P_{\theta_0}(\psi) = & P_{\theta_0}(\hat{\psi_0}) + \frac{1}{2} \int_{\R \times \T} \langle i \partial_x w_0, w_0 \rangle_\C \\
= & \int_\T \bigg( P_{\theta_0}(\hat{\psi}_0) + \int_\R \big\langle i \hat{\psi}_0', w_0(\cdot , y) \big\rangle_\C + \frac{1}{2} \int_\R \big\langle i \partial_x w_0(\cdot, y), w_0(\cdot , y) \big\rangle_\C \bigg) dy.
\end{align*}
This is exactly~\eqref{eq:P-Fubini}, so that it only remains to establish~\eqref{eq:p-slice}. This latter inequality is a direct consequence of~\eqref{eq:dev-[P]-H1} since $P_{\theta_0}(\hat{\psi}_0) = [P](\hat{\psi}_0)$ modulo $\pi$ by definition of the untwisted momentum. This completes the proof of Lemma~\ref{lem:P-1-2}.
\end{proof}

Going back to the density result in Corollary~\ref{cor:dens-X}, we finally
derive the following useful formula for the momentum of smooth functions with compactly supported gradients.

\begin{lem}
\label{lem:val-P-2}
Let $g$ be a smooth function in $X(\R)$ such that there exist numbers $R^\pm > 0$ and $\theta^\pm \in \R$ for which $g(x) = e^{i \theta^\pm}$ for any $\pm x \geq R^\pm$. Consider a function $h \in \boC_c^\infty(\R \times \T)$ with support in $[- R^-, R^+] \times \T$ and set $\psi = g + h$. Then, the function $\hat{\psi}_0$ writes as $\hat{\psi}_0(x) = e^{i \theta_0(x)}$ for $\pm x \geq R^\pm$, with $\theta_0(x) = \theta^+$ if $x \geq R^+$ and $\theta_0(x) = \theta^-$ for $x \leq - R^-$. Moreover, the quantities $p_{\theta_0}(\psi(\cdot, y))$ in Lemma~\ref{lem:P-1-2} are given by
\begin{equation}
\label{eq:val-p-2}
p_{\theta_0} \big( \psi(\cdot, y) \big) = \frac{1}{2} \int_\R \langle i \partial_x \psi(\cdot, y), \psi(\cdot, y) \rangle_\C + \frac{1}{2} \big( \theta^+ - \theta^-\big).
\end{equation}
for almost any $y \in \T$. As a consequence, we have
\begin{equation}
\label{eq:val-P-2}
P_{\theta_0}(\psi) = \int_\T p_{\theta_0} \big( \psi(\cdot, y) \big)\, dy,
\end{equation}
with $P_{\theta_0}(\psi) = [P](\psi)$ modulo $\pi$, and $p_{\theta_0} \big( \psi(\cdot, y) \big) = [P] \big( \psi(\cdot, y) \big)$ modulo $\pi$, for almost any $y \in \T$. When the function $\hat{\psi}_0$ does not vanish on $\R$, the momentum $P(\psi)$ is also given by~\eqref{eq:val-P-2}.
\end{lem}

\begin{proof}
Observe first that $\psi(x, y) = g(x)$ when $\pm x \geq R^\pm$, so that
$$
\hat{\psi}_0(x) = \int_\T \psi(x, y) \, dy = g(x) = e^{i \theta^\pm} = e^{i \theta_0(x)}.
$$
For almost every $y \in \T$, we therefore deduce from~\eqref{def:p-slice} that
\begin{equation}
\label{francais}
\begin{split}
p_{\theta_0} \big( \psi(\cdot, y) \big) = \frac{1}{2} \int_\R \Big( \big\langle i \hat{\psi}_0', \hat{\psi}_0 \big\rangle_\C + \big( \chi_r \, \theta_0 \big)' + 2 \big\langle i \hat{\psi}_0', w_0(\cdot , y) \big\rangle_\C + \big\langle i \partial_x w_0(\cdot , y), w_0(\cdot , y) \big\rangle_\C \Big),
\end{split}
\end{equation}
with $w_0 = \psi - \hat{\psi}_0$ and $r > \max \{ R^-, R^+ \}$. We next have
$$
\int_\R (\chi_r \, \theta_0)' = \theta^+ - \theta^-.
$$
Since $w_0(x, y) = 0$ for $\pm x \geq R^\pm$, we also deduce from an integration
by parts that
$$
\int_\R \big\langle i \hat{\psi}_0', w_0(\cdot , y) \big\rangle_\C = \frac{1}{2}
\int_\R \Big( \big\langle i \hat{\psi}_0', w_0(\cdot , y) \big\rangle_\C
+ \big\langle i \partial_x w_0(\cdot , y), \hat{\psi}_0 \big\rangle_\C
\Big).
$$
Formula~\eqref{eq:val-p-2} then follows from~\eqref{francais}. Formula~\eqref{eq:val-P-2}, as well as the other statements in Lemma~\ref{lem:val-P-2}, then result from the definitions in Lemma~\ref{lem:def-P-2} and the properties in Lemma~\ref{lem:P-1-2}. This concludes the proof of Lemma~\ref{lem:val-P-2}.
\end{proof}

\begin{merci}
The authors acknowledge support from the project ``Dispersive and random waves'' (ANR-18-CE40-0020-01) of the Agence Nationale de la Recherche. A.~de Laire was also supported by the the Labex CEMPI (ANR-11-LABX-0007-01), and
P.~Gravejat, by CY Initiative of Excellence (Grant ``Investissements d'Avenir'' ANR-16-IDEX-0008).
\end{merci}

\bibliographystyle{plain}
\bibliography{Bibliogr}

\end{document}